%% file: A_1.tex
\documentclass{ut-thesis}
\usepackage{amsmath}
\usepackage{amssymb}
\usepackage{amsthm}
\newtheorem{theorem}{Theorem}
\newtheorem{proposition}{Proposition}
\newtheorem{lemma}{Lemma}
\newtheorem{claim}{Claim}
\newtheorem{Remark}{Remark}
\newtheorem{notation}{Notation}
\newtheorem{definition}{Definition}
\usepackage{subfiles}
\usepackage{import}
\usepackage{hyperref}
\title{A Gluing problem for a gauged hyperbolic PDE}
\author{Amirmasoud Geevechi}
\degree{PhD}
\department{Mathematics}
\gradyear{2024}
\begin{document}
\maketitle
\newpage
{
\Large
\textit{Dedicated to MirHossein Mousavi}
}
\newpage
{
\Large
\textit{Supervised by Prof. Robert Jerrard}
}
\newpage
\begin{abstract}
    In this thesis, we study the dynamic Abelian Higgs model in dimension $3$ at the critical coupling. This is a system of partial differential equations which enjoys local symmetries known as gauge transformations. The stationary finite energy solutions to these equations in dimension $2$ have been classified by Jaffe and Taubes in 1980, the so called vortex configurations. In 1992, Stuart has proved that one can construct solutions near the critical coupling regime in dimension $1+2$ whose dynamics are approximated by a finite dimension Hamiltonian system to the moduli space which reduces to the geodesic flow at the critical coupling.
\\
\\
In this project, we study how one can glue the vortex configurations to find dynamic solutions in dimension $3$. More precisely, we prove that one can construct solutions which are approximated by wave maps to the moduli space of vortex configurations. The proof involves an ansatz to construct approximate solutions and then add perturbations. In the ansatz, we go through an iterative mechanism to reduce the error of the approximate solution so that it is prepared to be perturbed to find an honest solution.  
\\
\\
In both steps of the project, the ansatz and perturbation, the choice of gauge is crucial. It is noteworthy that the choice of gauge is different in these two steps. We proceed by a choice for gauge, simplify the equations and then we have to decompose the quantities into two components, the zero modes (the tangent vectors in the moduli space) and the orthogonal complement to zero modes.
According to the Higgs mechanism, stability is available for the components orthogonal to zero modes. In this regard, in the ansatz, the dynamics of zero modes are designed in such a way that some orthogonality condition is satisfied. In the perturbation part, the dynamics of zero modes is forced by the evolution of orthogonal components. Obtaining desired estimates for the tangential part requires taking advantage of explicit structure of equations, rather than the usual estimates. Also, the number of iterations in the ansatz should be high enough so that the desired estimates hold for the tangential part. 
\\
\\
This thesis has been supervised by Prof. Robert Jerrard from the department of mathematics at the University of Toronto.
\end{abstract}
\tableofcontents
\chapter{Abelian Higgs Model}
\label{ch1}
\section{Introduction}
\label{ch1.sec1}
Abelian Higgs model is a physical model which can be regarded as part of the Higgs sector of the standard model in particle physics. In particle physics, there is a well-known concept of wave-particle duality by which studying the dynamics of particles can be replaced by the study of quantum fields defined over a space-time. In this regard, Abelian Higgs model describes the interaction of the  so called Higgs field and the electromagnetic field. 
\\
Consider the spacetime ${\mathbb{R}^{1+3}}$. The Higgs field is a complex function
\[
\Phi:\mathbb{R}^{1+3}\to \mathbb{C}
\]
and the electromagentic potential is a real one-form
\[
A=A_{0}dt+A_1dx^1+A_2dx^2+A_3dz
\]
where $t$ is the time component. The components of the electromganetic field are given by
\begin{equation}
\label{eq6}
\mathcal{F}_{_{\alpha\beta}}=\partial_{_{\alpha}}A_{_{\beta}}-\partial_{_{\beta}}A_{_{\alpha}}
\end{equation}
where
\begin{equation}
\label{eq4}
\partial_0=\partial_t\hspace{.3cm},\hspace{.3cm}\partial_1=\partial_{_{x_1}}\hspace{.3cm},\hspace{.3cm}\partial_2=\partial_{_{x_2}}\hspace{.3cm},\hspace{.3cm}\partial_3=\partial_z
\end{equation}
The Abelian Higgs model with the coupling parameter $\lambda$ is described by the following equations:
\begin{align}
&D_0D_0\Phi-\sum_{j=1}^3D_jD_j\Phi+\frac{\lambda}{2}\big(|\Phi|^2-1\big)\Phi=0\label{eq1.new1}\\
&\Big(\partial_0\mathcal{F}_{0j}-\sum_{k=1}^3\partial_k\mathcal{F}_{kj}\Big)-\big(i\Phi,D_j\Phi\big)=0 \hspace{.4cm}\big(\text{for $j=0,1,2,3$}\big)\label{eq3}
\end{align}
where
\begin{equation}
\label{eq5}
D_{\mu}=\partial_{\mu}-iA_{\mu}
\end{equation}
for $\mu=0,1,2,3$ and
\begin{equation}
\label{eq7}
(a,b)=\Re(a\bar{b})
\end{equation}
for any two complex numbers $a,b$. These notations will be used throughout this thesis. 
\\
The above equations correspond to the finding the critical points of the functional:
\begin{equation}
\label{eq7.01}
E=\int_{_{\mathbb{R}^{1+3}}}\Big(-|D_0\Phi|^2+\sum_{j=1}^3|D_j\Phi|^2\Big)+\frac{1}{2}\sum_{\alpha,\beta}\big(\mathcal{F}_{_{\alpha\beta}}\big)^2+\frac{\lambda}{4}\big(|\Phi|^2-1\big)^2
\end{equation}
The Abelian Higgs model has a crucial property which is invariance under some local symmetries called gauge transformations. If $(\phi,A)$ is a solution for the above equations, then for any smooth function $\chi$, $(\tilde{\phi},\tilde{A})$ obtained by
\begin{equation}
\label{eq7.1.new}
\begin{split}
&\tilde{\Phi}(t,x)=\Phi(t,x)e^{i\chi(t,x)}\\
&\tilde{A}=A+d\chi
\end{split}
\end{equation}
will be another solution. The transformation described by \eqref{eq7.1.new} is called a gauge transformation in the literature of mathematical physics.
\\
One can consider similar equations for the spacetime $\mathbb{R}^{1+2}$. Also, by removing the time variable, one obtains the static two-dimensional Abelian Higgs model described by the following equations:
\begin{equation}
\label{eq8.new}
\begin{split}
&-\sum_{j=1}^2D_jD_j\phi+\frac{\lambda}{2}\big(|\phi|^2-1\big)\phi=0\\
&\sum_{k=1}^2\partial_k\mathcal{F}_{kj}+\big(i\phi,D_j\phi\big)=0 \hspace{.4cm}\big(\text{for $j=1,2$}\big)
\end{split}
\end{equation}
From now on, we use the varibles $(\phi,\alpha)$ for the static two-dimesnial Abelian Higgs model with
\[
\alpha=\alpha_1dx^1+\alpha_2dx^2
\] 
The equations \eqref{eq8.new} appear in the theory of superconductivity.
\\
The finite energy stationary solutions to the Abelian Higgs model in $2D$ in the so called critical coupling $\lambda=1$ have been classified by the work of Jaffe and Taubes in 1990 \cite{JT80}. They proved that every finite energy solution to these equations can be characterized by the zero set of the Higgs field, up to a gauge field. More precisely, they proved that given any finite set $S=\{z_1,z_2,\cdots,z_N\}\subset$ with possible multiplicities, there exists a unique smooth solution $(\phi,\alpha)$ up to gauge transformations such that $\phi$ vanishes exactly on $S$. The points $\{z_j\}_{j=1}^N$ are called vortex centers. These solutions are called vortex configurations. In section \ref{sec1.2}, we give a detailed description of the construction of vortex configurations.
\\
The question that we address in this thesis is a gluing problem; how can one use the vortex solutions to find dynamics solutions in $\mathbb{R}^{1+2}$ or $\mathbb{R}^{1+3}$. This problem has been studied in dimension $2$ by Stuart in \cite{Stu92} in 1994 and revisited later on by Palvelev in \cite{Pav11} in 2008. In this thesis, we answer the question in dimension $\mathbb{R}^{1+3}$.
\\
To answer this question, the space of whole $N$-vortex configurations has been considered as a smooth Riemannain manifold named the moduli space $M_N$. Then, in \cite{Stu92}, it has been proved that if $(q,p)=(q_{\mu},p_{\mu})_{\mu}:[0,T)\to (TM_2)^{*}$ satisfies a Hamiltonian system:
\begin{equation}
\label{hseq1}
\begin{split}
\frac{dp_{\mu}}{d\tau}=-\frac{\partial H_{\pm}}{\partial q_{\mu}}\hspace{.3cm},\hspace{.3cm}\frac{dq_{\mu}}{d\tau}=\frac{\partial H_{\pm}}{\partial p_{\mu}}
\end{split}
\end{equation}
for some $H_{\pm}$, then
for any $\epsilon>0$ small enough, corresponding to the coupling constant $\lambda=1\pm \epsilon^2$, one can construct solutions to the AHM equations of the form:
\begin{equation}
\label{in.in.eq1}
\begin{pmatrix}
\Phi
\\
A_1,A_2
\\
A_0
\end{pmatrix}(x,t)
=\begin{pmatrix}
\phi((x;q(\epsilon t))e^{i\Sigma}
\\
\alpha(x;q(\epsilon t))+d\Sigma
\\
0
\end{pmatrix}+ O(\epsilon)
\end{equation}
for some function $\Sigma$, valid over an interval of the form $[0,\frac{T^{\prime}}{\epsilon}^2)$, where the perturbation is measured in some Sobolev norm.
\\
In \cite{Pav11}, by similar arguments, Palvelev has proved that if $\lambda=1$, then one can consider a geodesic $q:[0,T]\to M_N$ and construct solutions to the AHM of the form \eqref{in.in.eq1}. We have given a detailed description of these two results and the main ideas of the proofs in section \ref{sec1.3}.
\\
In this thesis, we prove that if $q:[0,T]\times \mathbb{R}\to M_N$ is a smooth wave map, then for any $\epsilon>0$ small enough, one can construct solutions to the AHM close to $q$ in the sense that:
\begin{equation}
\label{in.in.eq2}
\begin{pmatrix}
\Phi
\\
A_1,A_2
\\
A_0,A_3
\end{pmatrix}(t,x,z)
=\begin{pmatrix}
\phi((x;q(\epsilon t,\epsilon z))
\\
\alpha(x;q(\epsilon t, \epsilon z))
\\
0
\end{pmatrix}+ 
\begin{pmatrix}
O(\epsilon)
\\
O(\epsilon)
\\
O(\epsilon ^{\frac{1}{2}})
\end{pmatrix}
\end{equation}
where the perturbation is measured in some Sobolev norm and the result is valid over a time interval of the form $[0,\frac{T^{\prime}}{\epsilon})$. This result is proved in chapter \ref{ch2}.
\\
Amongst the many ideas used in this construction, I would like to mention a heuristic on how the gauge symmetry contributes to the the aforementioned gluing problem.
\subsection{Organization of the Thesis}
The rest of chapter \ref{ch1} is devoted to a detailed description of the construction of vortex configurations based on \cite{JT80} in section \ref{sec1.2}, explaining the main ideas of the work of Stuart \cite{Stu92} and Palvelev \cite{Pav08} in \ref{sec1.3} and finally the physical implication of these results in section \ref{ch1.4}.
\\
In chapter \ref{ch2}, we give a proof of the aforementioned main result of the thesis on the approximation of the AHM equations by the wave maps to the moduli space of vortex configurations. Finally, in the appendix, we prove analytical results on the elliptic equations used in the chapter \ref{ch2}.
\section{Static 2D Abelian Higgs Model}
\label{sec1.2}
In this section, I review some well-known literature on the time independent Abelian Higgs model in dimension $2$, based on \cite{JT80} The stationary $2D$ equations are of independent interest, however their almost explicit structure is of great usage in the dynamic problem in chapter \ref{ch2}. In subsection \ref{cvc}, I go over their construction. Following \cite{Stu92},  in subsection \ref{msvg}, I will describe the moduli space of vortex configuration as a smooth Riemannian manifold, and in subsection \ref{rmvg}, I will explain te smooth Riemannian manifold structure on the moduli space.
\subsection{Construction of vortex configurations}
\label{cvc}
Consider the system of equations \eqref{eq8.new}. These equations correspond to the Lagrangian:
\begin{equation}
\label{eq12.new}
\mathcal{L}=\int_{_{\mathbb{R}^2}}\Bigg(\sum_{j=1}^2|D_j\phi|^2+\mathcal{F}_{_{12}}^2+\frac{\lambda}{4}\big(|\phi|^2-1\big)^2\Bigg)
\end{equation}
By some integration by parts,
\begin{equation}
\label{eq13.new.1}
\begin{split}
\mathcal{L}&=\int_{_{\mathbb{R}^2}}\Big(\big(\partial_1\phi_1+\alpha_1\phi_2\big)\mp\big(\partial_2\phi_2-\alpha_2\phi_1\big)\Big)^2\\
&+\int_{_{\mathbb{R}^2}}\Big(\big(\partial_2\phi_1+\alpha_2\phi_2\big)\pm\big(\partial_1\phi_2-\alpha_1\phi_1\big)\Big)^2\\
&+\int_{_{\mathbb{R}^2}}\Big[\mathcal{F}_{_{12}}\pm\big(|\phi|^2-1\big)\Big]^2\pm\int_{\mathbb{R}^2}\mathcal{F}_{12}\\
&+\frac{(\lambda-1)}{4}\big(1-|\phi|^2\big)^2
\end{split}
\end{equation}
where $\phi_1,\phi_2$ denote the real and imaginary parts of $\phi$.
\\
Suppose that $\lambda=1$. Then, according to \eqref{eq12.new}, finite energy solutions satisfy the conditions:
\begin{equation}
\label{eq12.1}
\begin{split}
&|\phi|\to 1\\
&|D_j\phi|\to 0 \hspace{.4cm}j=1,2
\end{split}
\end{equation} 
at infinity. The behavior of $\phi$ at infinity can be described by
\begin{equation}
\label{eq12.2}
\phi_{_{|x|\to \infty}}:S^1\to S^1 
\end{equation}
One can assign a winding number to this map which coincides with the quantity:
\begin{equation}
\label{eq12.3}
K=\frac{1}{2\pi}\int_{\mathbb{R}^2}\mathcal{F}_{_{12}}d^2x
\end{equation}
When $\lambda=1$, one can fix $K$ and consider the minmization problem for the Lagrnagian $\mathcal{L}$ and obtain the following equations:
\begin{equation}
\label{eq14}
\begin{split}
&\big(\partial_1\phi_1+\alpha_1\phi_2\big)\mp\big(\partial_2\phi_2-\alpha_2\phi_1\big)=0\\
&\big(\partial_2\phi_1+\alpha_2\phi_2\big)\pm\big(\partial_1\phi_2-\alpha_1\phi_1\big)=0\\
&\mathcal{F}_{12}\pm\frac{1}{2}\big(|\phi|^2-1\big)=0
\end{split}
\end{equation}
One can also write these equations in the form
\begin{equation}
\label{eq14.1}
\begin{split}
&\big(D_1\pm i D_2\big)\phi=0\\
&\mathcal{F}_{12}\pm\frac{1}{2}\big(|\phi|^2-1\big)=0
\end{split}
\end{equation}
which are called the Bolgomony equations. The sign in the above equations coincides with the sign of $K$. The first solutions found for these equations were the rotationally symmetric ones \cite{NiOl}. As mentioned in \cite{Sa}, if the winding number is $n$, these solutions take the form
\begin{equation}
\label{eq14.2}
\begin{split}
&\phi=e^{i \pi n}\rho(r)\\
&A_r=0\hspace{.3cm},\hspace{.3cm}A_{_{\theta}}=na(r)
\end{split}
\end{equation}
in polar coordniates where $\rho,a$ satisfy the equations
\begin{equation}
\label{eq14.3}
\begin{split}
&r\frac{d\rho}{dr}-n(1-a)\rho=0\\
&\frac{2n}{r}\frac{da}{dr}+(\rho^2-1)=0
\end{split}
\end{equation}
with the boundary conditions
\begin{equation}
\label{eq14.4}
\rho(0)=a(0)=0\hspace{.3cm},\hspace{.3cm}\rho(\infty)=a(\infty)=1
\end{equation}
The function $\rho$ behaves as below:
\begin{equation}
\label{eq14.5}
\begin{split}
&\rho(r)\sim Ar^n\hspace{.5cm}r\to 0\\
&\rho(r)\sim  1-BK_0(r)\hspace{.5cm}r\to \infty
\end{split}
\end{equation}
where $BK_0$ is the zeroth order modified Bessel function.
\\
\\
To study Bogomolny equations, one can linearize it around a given static 2D solution. Due to the gauge invariance property of the equations, onc can accompany the linearized equations with a gauge condition. Corresponding to the gauge transformations \eqref{eq7.1.new}, one can define the infinitesimal gauge transformations
\begin{equation}
\label{eq24}
\big(\tilde{\phi},\tilde{a}\big)\to \big(\tilde{\phi}+i\phi\chi,\tilde{a}+d\chi\big)
\end{equation}
Now, Consider a vector $\big(\tilde{\phi},\tilde{a}\big)$. We say that it satisfies the gauge orthogonality condition with respect to the vortex solution $(\phi,\alpha)$, if
\begin{equation}
\label{eq25.new.ch1}
\nabla.\tilde{a}-\big(i\phi,
\tilde{\phi}\big)=0
\end{equation}
The linearized Bogomolny equations accompanied with the gauge orthogonality condition can be written in the form:
\begin{equation}
\label{eq14.6}
\mathcal{D}_{_{(\phi,\alpha)}}
\begin{pmatrix}
\tilde{\phi}\\
\tilde{a}
\end{pmatrix}=
\begin{pmatrix}
\bar{\partial}\tilde{\phi}-i\bar{a}\tilde{\phi}-i\bar{\tilde{a}}\phi\\
\partial \bar{\tilde{a}}+\frac{i}{4}\bar{\phi}\tilde{\phi} 
\end{pmatrix}=\begin{pmatrix}
0\\
0
\end{pmatrix}
\end{equation}
where
\begin{equation}
\label{eq14.7}
\begin{split}
&a=\frac{1}{2}\big(\alpha_1-i\alpha_2\big)\\
&\tilde{a}=\frac{1}{2}\big(\tilde{a}_1-i\tilde{a}_2\big)\\
&\partial=\frac{1}{2}\big(\partial_1-i\partial_2\big)
\end{split}
\end{equation}
In \cite{Wein}, it has been shown that around a solution $(\phi,\alpha)$ with the winding numnber $n$, there are $2n$ linearly independent solutions of 
the linearized equations \eqref{eq14.6}. These solutions are called zero modes. This suggests that there might be a $2n$-parameter family of solutions with the winding number $n$. In \cite{JT80}, such solutions have been constructed. They are called "vortices" and "anti-vortices".
\begin{theorem}
\label{thm1.new}
Given an integer $N\ge 0$ and points $z_1,z_2,\cdots,z_N$ in the complex plane. Suppose that $z_j$ repeats $n_j$ times in the sequence. There exists a smooth finite energy solution $(\phi,\alpha)$ to equations \eqref{eq14.1} with the positive choice for sign, satisfying the following properties:
\begin{enumerate}
\item
The zeros of $\phi$ are the set of points $z_1,z_2,\cdots,z_N$ and the behavior of $\phi$ around any $z_j$ is:
\begin{equation}
\label{eq9}
\phi(z)\sim c_j\big(z-z_j\big)^{n_j}
\end{equation}
for some $c_j\neq 0$.
\item
\begin{equation}
\label{eq10}
|D_1\phi|+|D_2\phi|\le C\big(1-|\phi|\big) \le C \exp\big(-(1-\delta)|z|\big)
\end{equation}
for any $\delta$ where $C=C(\delta)$ is a constant.
\item
\begin{equation}
\label{eq11}
N=\frac{1}{2\pi}\int_{_{\mathbb{R}^2}}\mathcal{F}_{12}
\end{equation}
\end{enumerate}
\end{theorem}
In the above theorem, the points $\big(z_1,z_2,\cdots,z_N\big)$ are called the vortex centers and solutions are called vortices or $N$-vortex solutions. If $(\phi,\alpha)$ is an $N$-vortex solution, then $(\bar{\phi},-a)$ is an $N$-antivortex solution for equation \eqref{eq14.6}, but with the negative choice for sign.
\\
\\
To prove the above theorem, one can notice that if $u$ solves the equation
\begin{equation}
\label{eq15.new.ch1}
-\Delta u+e^{u}-1=-4\pi\sum_{_{k=1}}^N\delta(z-z_k)
\end{equation}
with
\begin{equation}
\label{eq16}
\lim _{_{|z|\to\infty}}u=0
\end{equation}
then the followings will solve \eqref{eq14.1}
\begin{equation}
\label{eq17.new}
\begin{split}
&\phi=\exp \frac{1}{2}\big[u+i\Theta\big]\\
&\alpha_1=\frac{1}{2}\big(\partial_2 u+\partial_1 \Theta\big)\\
&\alpha_2=\frac{-1}{2}\big(\partial_1u-\partial_2\Theta\big)
\end{split}
\end{equation}
where
\begin{equation}
\label{eq18}
\Theta=2\sum_{i=1}^N \arg(z-z_{i})
\end{equation}
To solve \eqref{eq15.new.ch1}, one can consider the functions
\begin{equation}
\label{eq19.new}
\begin{split}
&u_0=-\sum_{k=1}^N\ln\Big(1+\mu|z-z_k|^{-2}\Big)\\
&g_0=4\sum_{k=1}^N\frac{\mu}{\big(|z-z_k|^2+\mu\big)^2}
\end{split}
\end{equation}
for some $\mu>4N$. Then, regarding $u_0$ as a distribution, we have:
\begin{equation}
\label{eq20.new}
\Delta u_0=-4\sum_{k=1}^N\frac{\mu}{\big(|z-z_k|^2+\mu\big)^2}+4\pi\sum_{k=1}^N\delta\big(z-z_k\big)
\end{equation}
Therefore, by setting
\[
v=u-u_0
\]
\eqref{eq20.new} is equivalent to
\begin{equation}
\label{eq21.new.ch1}
\Delta v=e^{v+u_0}+(g_0-1)
\end{equation}
with 
\begin{equation}
\label{eq22}
\lim_{_{|z|\to \infty}}|v|=0
\end{equation}
Equation \eqref{eq21.new.ch1} is equivalent to finding the critical point of the functional
\begin{equation}
\label{eq23}
\begin{split}
&T:H^1\to \mathbb{R}\\
&T(v)=\int_{_{\mathbb{R}^2}}\Big[\frac{1}{2}|\nabla v|^2+v(g_0-1)+e^{u_0}(e^v-1)\Big]
\end{split}
\end{equation}
It has been proved in \cite{JT80} that this functional has a unique minimizer and therefore the Bolgomony equations have the prescribed solutions.
\subsection{Moduli Space of vortex configurations}
\label{msvg}
 In \cite{Stu92}, the moduli space $M_N$, the space of all $N$-vortex solutions modulo gauge equivalence, provided by theorem \ref{thm1.new}, has been discussed. $M_N$ is homemorphic to $\mathbb{R}^{2N}$. This is because corresponding to a vortex solution $(\phi,\alpha)$ with the vortex centers
\[
\Big(z_1,z_2,\cdots,z_N\Big)
\]
one can consider the coefficients of the polynomial
\[
p(z)=(z-z_1)(z-z_2)\cdots (z-z_N)
\]
and notice that this defines a bijection between $\mathbb{R}^{2N}$ and $M_N$. 
\\
\\
From now on, we assume that
\[
z_j=z_{_{j,1}}+iz_{_{j,2}}
\]
for real numbers $z_{_{j,1}},z_{_{j,2}}$. Also, we write
\[
p(z)=(z-z_1)(z-z_2)\cdots (z-z_N)=S_0+S_1z+S_2z^2+\cdots+z^N
\]
and assume that
\[
S_j=S_{j,1}+iS_{j,2}
\]
for real numbers $S_{j,1},S_{j,2}$. Also, let
\[
S=\big(S_0,S_1,\cdots,S_{N-1}\big)
\]
\begin{theorem}{\cite{Pav08}}
The vortex solution $(\phi,\alpha)$ depends smoothly on  $S$; for any $z\in \mathbb{C}$, the function $(\phi,\alpha)(z;S)$ depends smoothly on $z$.
\end{theorem}
To justify this theorem, one can note that the functions $u_0$ and $g_0$ in \eqref{eq19.new} depend smoothly on $S$ and therefore, it is expectable that $v$ behaves smoothly with respect to $S$. To be more precise, one can use \eqref{eq17.new} to write
\[
\phi(z)=p(z)f(z)
\]
where
\begin{equation}
\label{eq23.01}
p(z)=(z-z_1)(z-z_2)\cdots(z-z_N)
\end{equation}
and $f(z)>0$. Let
\[
w=2\ln f+ \ln\big(1+|p|^2\big)
\]
Then,
\begin{equation}
\label{eq23.02}
w=v+\ln \big(1+|p|^2\big)-\sum_{k=1}^N \ln\big(\mu+|z-z_k|^2\big)
\end{equation}
and according to equation \eqref{eq21.new.ch1},
\begin{equation}
\label{eq23.03}
\Delta w=\frac{|p|^2}{1+|p|^2}e^w-1+\Delta \ln \big(1+|p|^2\big)
\end{equation}
In \cite{JT80}, it has been proved that $v\in H^2$. Therefore, $w\in H^2$. To prove that $w$ behaves smoothly with respect to $S$, first note that for any $h\in H^2(\mathbb{R}^2)$,
\[
(e^h-1)\in L^2(\mathbb{R}^2)
\]
and the map
\begin{equation}
\label{eq23.04}
\begin{split}
&T:H^2(\mathbb{R}^2)\to L^2(\mathbb{R}^2)\\
&T(h)=e^h-1
\end{split}
\end{equation}
is smooth. This can be proved by considering the Taylor explansion of $e^h$ and using the Sobolev embedding theorem. This implies that the map
\begin{equation}
\label{eq23.05}
\begin{split}
&F:H^2(\mathbb{R}^2)\times \mathbb{C}^N \to L^2 \big(\mathbb{R}^2\big)\\
&F(v,S)=-\Delta v+\frac{|p(S)|^2}{1+|p(S)|^2}e^v-1+ \Delta \ln \big(1+|p(S)|^2\big)
\end{split}
\end{equation}
is smooth. Now, suppose that $w(S)$ denotes the solution to equation \eqref{eq23.03}. In \cite{Pav08}, it has been shown that the map 
\begin{equation}
\label{eq23.06}
\begin{split}
&I:H^2(\mathbb{R}^2)\to L^2(\mathbb{R}^2)\\
&I(h)=F^{\prime}_v\big(w(S);S\big)h
\end{split}
\end{equation}
is invertible for any $S\in \mathbb{C}^N$, where $F^{\prime}_v$ denotes differentiation with respect to the first input variable of $F$. This together with an implicit function theorem for Banach spaces implies that $w$ and henceforth $(\phi,\alpha)$ behave smoothly with respect to $S$. Furthermore, one can prove the following:
\begin{proposition}
\label{p1.new}
For any multi-index $r$ and every compact subset $K\subset M_N$, there exist numbers $A,B,R>0$ such that the function $u=\ln|\phi|^2$ corresponding to a point $p\in K$ has the property 
\begin{equation}
\label{eq23.1}
|D^r u(x)|\le Ae^{-B|x|}
\end{equation} 
for every $x\in\mathbb{R}^2$ with $|x|\ge R$, where $D$ denotes differentiation with respect to $x_1,x_2$ and $S_{j,1},S_{j,2}$ with $j\in\{0,1,2,\cdots,N\}$.
\end{proposition}
\subsection{A Riemannian metric on the moduli space of vortex configurations}
\label{rmvg}
One can define a Riemannian metric on $M_N$. Consider a smooth curve
\[
\gamma:[0,1]\to M_N
\]
In \cite{Stu92}, the procedure to define $||\gamma^{\prime}(0)||$ is as follows:
\\
\\
Suppose that
\[
(\phi,\alpha)=\big(\phi(\gamma(0)),\alpha(\gamma(0))\big)
\]
First, one finds a smooth function $\chi$ such that the vector
\begin{equation}
\label{eq26.new}
\big(\tilde{\phi},\tilde{a}\big)=\Big(\frac{d\phi}{dt}(0),\frac{d\alpha}{dt}(0)\Big)+\big(i\phi\chi,d\chi)                                                        
\end{equation}
satisfies the gauge-orthogonality condition with respect to $(\phi,\alpha)$. Following \eqref{eq25.new.ch1}, this leads to the equation
\begin{equation}
\label{eq26.1.new}
\Delta \chi-|\phi|^2\chi=-\frac{1}{2}\big(\partial_t \Theta\big)|\phi|^2
\end{equation}
To construct a solution for this equation, one writes
\begin{equation}
\label{eq27}
\chi=\frac{\rho}{2}\big(\partial_t \Theta\big)+\zeta
\end{equation}
where $\rho$ is a smooth cut-off function which is equal to $1$ if $|x|$ is large enough and vanishes inside a disk containing the zeros of $\phi$.
Then, $\zeta$ satisfies:
\begin{equation}
\label{eq28}
\Delta\zeta-|\phi|^2\zeta=g
\end{equation}
where $g$ is a smooth and compactly supported function. According to lemma \ref{al1} in the appendix, this equation has a smooth solution $\zeta$ which decays exponentially to zero at infinity. In this way, we find a solution $\chi$ for the equation \eqref{eq26.1.new}. 
According to proposition \ref{p1.new}, the function $u=\ln|\phi|^2$ has exponential decay. This implies that the vector $(\tilde{\phi},\tilde{a})$ defined by \eqref{eq26.new} belongs to $L^2$ and one defines
\begin{equation}
\label{eq29}
||\gamma^{\prime}(0)||=||(\tilde{\phi},\tilde{a})||_{_{L^2}}
\end{equation}
Using the procedure above, one can find the vectors
\[
n_{j,\alpha}=\Big(\frac{\partial \phi}{\partial S_{j,\alpha}}+i\phi \chi_{_{j,\alpha}},\frac{\partial a}{\partial S_{j,\alpha}}+d\chi_{_{j,\alpha}}\Big)
\]
which are gauge-orthogonal with respect to $\big(\phi,\alpha\big)$, for $j=1,2,\cdots,N$ and $\alpha=1,2$. These vectors are called the zero modes. The complex and real parts of $n_{\mu}$ are denoted by $n_{_{\mu,\phi}}$ and $n_{_{\mu,a}}$, respectively. Using proposition \ref{p1.new}, one obtains the following:
\begin{proposition}
\label{p2}
Suppose that $K$ is a compact subset of $M_N$. For every multi-index $r$, there exists numbers $A,\delta>0$ such that for every $p\in M_N$ and every $x\in \mathbb{R}^2$,
\begin{equation}
\label{eq30}
|D^rn_{_{j,\alpha}}(p)(x)|\le Ae^{-\delta|x|}
\end{equation}
where $D$ denotes differentiation with respect to $x_1$ or $x_2$ or $S_{_{j,1}}$ or $S_{_{j,2}}$ with $j\in \{0,1,2,\cdots,N\}$.
\end{proposition}
One can also obtain the following estimates by looking at the construction of the vortex solutions and proposition \ref{p1.new}
\begin{proposition}
\label{p30.1}
Suppose that $s$ is a multi-index and $D$ denotes differentiation with respect to $x_1$ or $x_2$ or $S_{_{j,1}}$ or $S_{_{j,2}}$ with $j\in \{0,1,2,\cdots,N\}$. Then, for $p\in M_N$, there exists $A,C,\delta,\epsilon>0$ such that if
\[
|q-p|<\epsilon
\]
then
\begin{equation}
    \label{eq32.new}
    \begin{split}
    &|D^s\phi(z;q)|< C(|z|+1)^{-|s|}\\
    &|D^s\alpha_j(z;q)|< C(|z|+1)^{-|s|-1}\hspace{.2cm}j=1,2\\
    &|D^s(|\phi|^2-1)|(z;q)<Ae^{-\delta|z|}\\
    &|D^s(D_{\alpha_j}\phi)|(z;q)<Ae^{-\delta|z|}
    \end{split}
\end{equation}
\end{proposition}
\section{Dynamic 2D Abelian Higgs Model}
\label{sec1.3}
The vortex configurations are static solutions to the Abelian Higgs Model. In \cite{Stu92} and \cite{Pav11}, the authors have constructed solutions for the dynamic Abelian Higgs model in $\mathbb{R}^{1+2}$ by using vortices, when $\lambda$ is close enough to $1$. In this section, we go over these the main ideas of these papers and mention some of the calcualutions used there.
\\
The general idea is to consider a path 
\[
q:[0,T)\to M_N
\]
and for small enough $\epsilon>0$, consider the ansatz
\begin{equation}
\label{eq31.new}
\begin{split}
&\Phi(t,x)=\phi(x;q(\epsilon t))e^{i\Sigma}+\epsilon^2\tilde{\phi}=\varphi\big(x, t\big)+\epsilon^2\tilde{\phi}\\
&A_0=\epsilon^3\tilde{a}_0\\
&A_i=\alpha_i(x;q(\epsilon t))+\partial_i \Sigma+\epsilon^2 \tilde{a}_i=a_i\big(x,t\big)+\epsilon^2 \tilde{a}_i \hspace{.5cm}i=1,2
\end{split}
\end{equation}
to find a solution for the Abelian Higgs Model such that
\[
||\big(\tilde{\phi},\tilde{a}\big)||_{_{H^3}}+||\big(\tilde{\phi}_t,\tilde{a}_t\big)||_{_{H^2}}
\]
remains bounded on a time interval of length $O(\frac{1}{\epsilon})$.
\\
\\
From now one, the input parameter of the path $q$ is denoted by $\tau$ and we will assume $\tau=\epsilon t$.
\\
\\
The gauge function $\Sigma$ is chosen such that the vector $\big(\frac{\partial \varphi}{\partial \tau},\frac{\partial a}{\partial \tau}\big)$ satisfies the gauge orthogonality condition.
\\
\begin{itemize}
\item 
In \cite{Stu92}, the following result has been proved:
\begin{theorem}
\label{thm1}
Consider the initial value problem for the Abelian Higgs model with 
\[
\lambda=1+\l\epsilon^2
\]
where $l=\pm 1$. Suppose that the initial data is close to a two vortex $\big(\phi(q(0)),\alpha(q(0))\big)$ in the sense that:
\begin{equation}
\label{t1.e1}
\begin{split}
&A(0,x)=\alpha\big(x;q(0)\big)+\epsilon^2\tilde{a}(0)\\
&\Phi(o,x)=\phi\big(x;q(0)\big)+\epsilon^2\tilde{\phi}(0)\\
&A_t(0,x)=\epsilon\sum_{_{\mu}}q_{_{\mu}}^{\prime}(0)n_{_{\mu,a}}+\epsilon^2\tilde{a}_t(0)\\
&\Phi_t\big(0,x\big)=\epsilon \sum_{\mu} q_{\mu}^{\prime}(0)n_{\mu,\phi}+\epsilon^2\tilde{\phi}_t(0)
\end{split}
\end{equation}
where $\big(\tilde{\phi}(0),\tilde{a}(0)\big)$ satisfy the following conditions:
\begin{align}
&\big(n_{_{\mu}},(\tilde{\phi}(0),\tilde{a}(0)\big)=\frac{d}{dt}\Big{|}_{_{t=0}}\big(n_{\mu},(\tilde{\phi},\tilde{a})\big)=0\label{t1.e2}\\
&\nabla.\tilde{a}(0,.)-\big(i\phi(0;q(0)),\tilde{\phi}(0,.)\big)=0\label{t1.e3}
\end{align}
Then, there exists $K$ such that if 
\[
\Big{|}\big(\tilde{\phi}(0),\tilde{a}(0)\big)\Big{|}_{_{H^3}}+\Big{|}\big(\tilde{\phi}_t(0),\tilde{a}_t(0)\big)\Big{|}_{_{H^2}}\le K
\]
then, for sufficiently small $\epsilon$, there exists a time $T_{*}=O\big(\frac{1}{\epsilon}\big)$ such that on the interval $[0,T_{*}]$, there exists a solution of the form
\begin{equation}
\label{t1.e3.new}
\begin{split}
&\Phi=\phi\big(x;q(t)\big)e^{i\sigma}+\epsilon^2\tilde{\phi}\\
&A=\alpha\big(x;q(t)\big)+d\sigma+\epsilon^2\tilde{a}(t,x)
\end{split}
\end{equation}
with
\begin{equation}
\label{t1.e4}f
\begin{split}
&q_{_{\mu}}(t)=q_{_{\mu}}^{0}(t)+\epsilon \tilde{q}_{_{\mu}}(t)\\
&p_{_{\mu}}(t)=p_{_{\mu}}^{0}(t)+\epsilon \tilde{p}_{_{\mu}}(t)
\end{split}
\end{equation}
where $p^0(\tau),q^0(\tau)$ are solutions of the Hamiltonian system
\begin{equation}
\label{t1.e5}
\frac{dp_i}{d\tau}=-\frac{\partial H}{\partial q_i}\hspace{.3cm},\hspace{.3cm}\frac{dq_i}{d\tau}=\frac{\partial H}{\partial p_i}
\end{equation}
where $H$ is defined by 
\begin{equation}
\label{eq35.5}
H(p,q)=\frac{1}{2}g^{*}(p,p)+V(q)
\end{equation}
and $g^{*}$ is the metric dual to the metric $g:T{{M_2}}\to \mathbb{R}$ and the potential $V$ is defined by
\begin{equation}
\label{eq35.7}
V(q)=\frac{l}{8}\int_{\mathbb{R}^2}\big(1-|\phi|^2\big)^2
\end{equation}
and the initial conditions are
\begin{equation}
\label{t1.e6}
q_{\mu}^0(0)=q_{\mu}(0)\hspace{.3cm},\hspace{.3cm}p_{\mu}^0(0)=\sum_{\nu}g_{_{\mu\nu}}q_{\nu}^{\prime}(0)
\end{equation}
and $\big(\tilde{\phi},\tilde{a}\big)$ satisfy the condition
\begin{equation}
\label{t1.eq7}
\begin{split}
&\nabla.\tilde{a}(t,.)-\big(i\phi\big(.;q(t)\big),\tilde{\phi}(t,.)\big)=0\\
&\big(n_{_{\mu}},(\tilde{\phi},\tilde{a})\big)=0
\end{split}
\end{equation}
and the maps
\begin{equation}
\label{t1.e6.new}
\begin{split}
&t\to \tilde{p}(t)\hspace{.2cm},\hspace{.2cm}t\to \frac{1}{\epsilon}\frac{dp}{dt}\\
&t\to \tilde{q}(t)\hspace{.2cm},\hspace{.2cm}t\to \frac{1}{\epsilon}\frac{dq}{dt}\\
&t\to \big((\tilde{\phi},\tilde{a}),(\tilde{\phi}_t,\tilde{a}_t)\big)\in H^3\oplus H^2
\end{split}
\end{equation}
are continuous and bounded independent of $\epsilon$. In addition $|A_0(t,.)|_{_{L^{\infty}}}=O(\epsilon^3)$ and the map $t\to \sigma (t)\in C^{\infty}\big(\mathbb{R}^2\big)$ is twice differentiable and the solution has the regularity
\begin{equation}
\label{t1.e7}
(\tilde{q},\tilde{p})\in C^2\big([0,T_{*}]\big)\oplus C^1\big([0,T_{*}]\big)
\end{equation}
and 
\begin{equation}
\label{t1.e8}
(\tilde{\phi},\tilde{a})\in C^1\big([0,T_{*}],H^1\oplus L^2\big)\cap C\big([0,T_{*}],H^3\oplus L^2\big)
\end{equation}
\end{theorem}
\item
In \cite{Pav11}, the following result has been proved:
\begin{theorem}
\label{thm6}
Consider a geodesic $Q:[0,\tau_0]\to M_N$. Suppose that
\[
Q(\tau)=\big(\phi(\tau),\alpha_1(\tau),\alpha_2(\tau)\big)
\]
Then, there exist positive numbers $\tau_1\le \tau_0$ and $\epsilon_0,M$ and a smooth family
$t\to \Sigma(t;.)\in C^{\infty}(\mathbb{R}^2;\mathbb{R})$
defined on $[0,\frac{\tau_1}{\epsilon}]$, with the follwoing properties:
\\
For each $\epsilon\in [0,\epsilon_0]$ there exists a solution $\big(\Phi^{\epsilon}(t),A_0^{\epsilon}(t),\alpha_1^{\epsilon}(t),\alpha_2^{\epsilon}(t)\big)$ of the Abelian Higgs model defined on the interval $[0,\frac{\tau_1}{\epsilon_0}]$ of the form
\begin{equation}
\label{eq1.thm6}
\begin{split}
&A_0^{\epsilon}(t,x_1,x_2)=\epsilon^3a_0^{\epsilon}(t,x_1,x_2)\\
&\alpha_1^{\epsilon}(t,x_1,x_2)=\alpha_1(\epsilon t;x_1,x_2)+\partial_{_{x_1}}\Sigma(t;x_1,x_2)+\epsilon^2\alpha_1^{\epsilon}(t,x_1,x_2)\\
&\alpha_2^{\epsilon}(t,x_1,x_2)=\alpha_2(\epsilon t;x_1,x_2)+\partial_{_{x_2}}\Sigma(t;x_1,x_2)+\epsilon^2\alpha_2^{\epsilon}(t,x_1,x_2)\\
&\phi^{\epsilon}(t,x_1,x_2)=\phi(\epsilon t;x_1,x_2)e^{i\Sigma}+\epsilon^2\phi^{\epsilon}(t,x,_1,x_2)
\end{split}
\end{equation}
such that 
\begin{equation}
\label{eq2.thm7}
||a_{0}^{\epsilon}(t)||_{_{H^3}}+||a_{1}^{\epsilon}(t)||_{_{H^3}}+||a_{2}^{\epsilon}(t)||_{_{H^3}}+||\phi^{\epsilon}(t)||_{_{H^3}}\le M
\end{equation}
for every $t$ with $0\le t \le \frac{\tau_1}{\epsilon}$.
\end{theorem}
\end{itemize}
Solving the ansatz \eqref{eq31.new} requires the linearization of the Abelian Higgs model equations. The linearization involves an operator $\mathcal{L}$ defined by
\begin{equation}
\label{eq33}
\mathcal{L}[\phi,\alpha](\tilde{\varphi},\tilde{a})=\begin{pmatrix}
-\sum_{i=1}^2\Big(\big(D_i^{(0)}\big)^2\tilde{\phi}-2i\tilde{a}_iD_i^{(0)}\phi\Big)+\frac{1}{2}\big(3|\phi|^2-1\big)\tilde{\phi}\\
\\
-\Delta \tilde{a}_i+|\phi|^2\tilde{a}_i-2\big(i\tilde{\phi},D_i^{(0)}\phi\big)
\end{pmatrix}_{{i=1,2}}
+\begin{pmatrix}
i\phi\big(\nabla.\tilde{a}-\big(i\phi,\tilde{\phi}\big)\big)\\
\\
\partial_i\big(\nabla.\tilde{a}-\big(i\phi,\tilde{\phi}\big)\big)
\end{pmatrix}_{{i=1,2}}
\end{equation}
where
\begin{equation}
\label{cvd.eq1}
D_i^{(0)}=\partial_i-i\alpha_i
\end{equation}
for $i=1,2$. If one assumes the gauge orthogonality condition for $(\tilde{\varphi},\tilde{a})$, the operator $\mathcal{L}$ changes to the following elliptic operator:
\begin{equation}
\label{eq33.14}
{L}[\phi,\alpha](\tilde{\varphi},\tilde{a})=\begin{pmatrix}
-\sum_{i=1}^2\Big(\big(D_i^{(0)}\big)^2\tilde{\phi}-2i\tilde{a}_i D_i^{(0)}\phi\Big)+\frac{1}{2}\big(3|\phi|^2-1\big)\tilde{\phi}\\
\\
-\Delta \tilde{a}_i+|\phi|^2\tilde{a}_i-2\big(i\tilde{\phi},D_i^{(0)}\phi\big)
\end{pmatrix}_{_{i=1,2}}
\end{equation}
In the the proof of theorems \ref{thm1} and \ref{thm6}, a coercivity result about the operator $L$ on the subspace orthogonal to zero modes is crucial. Indeed, the kernel of the elliptic operator $L$ is spanned by the zero modes at $\big(\phi,\alpha\big)$. We have:
\begin{equation}
\int_{{\mathbb{R}}^2}\big(\psi,L\psi\big)= Hess(E)_{_{(\phi,\alpha)}}\big(\psi,\psi\big)+\int_{{\mathbb{R}}^2}\Big(\nabla.\tilde{a}-\big(i\phi,\tilde{\phi}\big)\Big)^2
\end{equation}
where $E$ is the $(0+2)$-dimensional version of the energy functional \eqref{eq7.01}.
The above quantity is denoted by
\[
\overline{Hess}(E)_{_{(\phi,\alpha)}}\big(\psi,\psi\big)
\]
\begin{proposition}\big(Coercivity of the corrected Hessian\big)
\label{eq35.1}
There exists a universal number $\gamma>0$ such that for every 
\[
\psi=\big(\tilde{\phi},\tilde{a}\big)\in H^1
\]
with $\langle\psi,n_{\mu}\rangle_{_{L^2}}=0$ for every zero mode $n_{_{\mu}}$, there holds
\[
\overline{Hess}(E)_{(\phi,\alpha)}\big(\psi,\psi\big)\ge \gamma||(\tilde{\phi},\tilde{a})||_{_{H^1}}^2
\]
\end{proposition}
This was proved in \cite{Stu92}. It is noteworthy that part of the proof relies on the fact that the corrected Hessian can be written as:
\begin{equation}
\label{chess.eq1}
\overline{Hess}(E)_{_{(\phi,\alpha)}}\big(\psi,\psi\big)=\int_{_{\mathbb{R}^2}}|\mathcal{D}_{_{(\phi,\alpha)}}(\tilde{\varphi},\tilde{a})|^2
\end{equation}
where $\mathcal{D}$ is the first order operator defined by \eqref{eq14.6}, regarding the linearization of the Bolgomony equations \eqref{eq14.1} mixed with the gauge orthogonality condition.
\section{Physical Applications}
\label{ch1.4}
One can imagine that when the vortices are far away from each other, their dynamics is almost independent, but the question is how vortices interact when they are close to each other? These issues can be studied by inevstigating the metric and using the aformentioned theorems on the geodesic description of dynamics of vortices. Any element in $M_2$ can be identified with the center of mass and the relative position of its vortices: 
\[M_2=\mathbb{R}^2\times M_2^0\]
The metric on $M_2$ is invariant wih respect to translation and rotation of vortices. Therefore, as mentioned in \cite{VilSh}, the metric on $M_2^0$ can be desribed by:
\[ds^2=F^2(r)dr^2+G^2(r)d\theta^2\]
Here, $(r,\theta)$ are chosen such that $\big(rcos(2\theta),rsin(2\theta)\big)$ represents the relative position of the vortices. When, $r$ is large enough, the metric is almost flat:
\[ds^2\approx dr^2+r^2d\theta^2\]
It has been shown in \cite{Rub} that near the origin, the metric takes the form
\[
ds^2\approx (cr)^2(dr^2+r^2d\theta^2)
\]
This is in accordance with the fact that the metric on $M_2$ behaves smoothly with respect to the local coordinates $(z_1+z_2,z_1z_2)$. The above form of the metric implies that the curvature at the origin is positive. The surface $M_2^0$ can be visualized by a cone which is smoothed at the corner.
\\
In \cite{Stu92}, two scenarios for the interaction of vortices when they are close to each other has been discussed:
\begin{enumerate}
\item{\textbf{Repulsion and attraction of vortices:}}
\\
Suppose that $\lambda=1+\epsilon^2$ with $\epsilon$ small enough. In the settings of theorem \ref{thmstl}, the dynamic of vortices can be approximately modeled by the Hamiltonian dynamics corresponding to the Hamiltionian:
\begin{equation}
\label{eq55}
H(p,q)=\frac{1}{2}g^{*}(p,p)+V(q)
\end{equation}
where 
\begin{equation}
\label{eq56}
V(q)=\frac{1}{8}\int_{\mathbb{R}^2}\big(1-|\phi|^2\big)^2
\end{equation}
Suppose that $Z_{1}(\tau),Z_{2}(\tau)$ denote the location of the voritices as complex numbers. Consider the coordinate system
\begin{equation}
\label{eq57}
\begin{split}
&P(\tau)=Z_1(\tau)+Z_2(\tau)\\
&Q(\tau)=Z_1(\tau)Z_2(\tau)
\end{split}
\end{equation}
Suppose that
\begin{equation}
\label{eq57.1}
\begin{split}
&Z_1(0)=-Z_2(0)\\
&|Z_1|^{\prime}(0)=|Z_2^{\prime}|(0)=0
\end{split}
\end{equation}
Then, $Z_1(\tau)=-Z_2(\tau)$ for $\tau>0$. Since the potential $V$ depends only on $|Z_1-Z_2|$, then we can write 
\begin{equation}
\label{eq58}
V(q(\tau))=u(|Z_1-Z_2|)
\end{equation}
for some function $u$. Also, we can write
\begin{equation}
\label{eq59}
\frac{1}{2}g(q^{\prime},q^{\prime})=f(|Q|)|Q^{\prime}|^2
\end{equation}
where $f$ is nonnegative. Using the conservation of energy, we have:
\begin{equation}
\label{eq60}
f(|Q|)|Q^{\prime}|^2+u(|Q(\tau)|)=u(|Q(0)|)
\end{equation}
This implies that $|Q(\tau)|$ is decreasing, if the initial condition is nonzero. It has been conjectured that the above function $u$ is a decreasing function. If this is true, we can deduce that for $\lambda=1+\epsilon^2$, the vortices repell under the prescribed initail conditions. Similar arguments imply the attraction of voritces for $\lambda=1-\epsilon^2$.
\item{\textbf{Scattering of vortices:}}
Consider the initial position where $Z_1,Z_2$ are located symmterically with respect to origin on the real line. Suppose that $Q=q_1+iq_2$. Let
\begin{equation}
\label{eq60.new}
\begin{split}
&P(0)=P^{\prime}(0)=0\\
&q_1(0)<0\hspace{.2cm},\hspace{.2cm}q_1^{\prime}(0)=M>0\hspace{.2cm},\hspace{.2cm}q_2(0)=q_2^{\prime}(0)=0
\end{split}
\end{equation}
The vortcies will remain on the real line and when $M$ is large enough, the dynamics can be described by:
\begin{equation}
\label{eq61}
q_2(\tau)=0\hspace{.2cm},\hspace{.2cm}q_1(\tau)=f(\tau)
\end{equation}
where $f$ eventually becomes positive. This implies that
\begin{equation}
\label{eq62}
Z_1(\tau)=-Z_2(\tau)=\sqrt{-f(\tau)}
\end{equation}
Therefore, the vortices move towards each other along the real line and when $f$ changes sign from negative to positive, they move away from each orther along the $y$-axis. This is the so-called right angle scattering. This can also be seen by looking at the geodesics on the manifold $M_2^0$. When a geodesic passes from the origin, the angle $\theta$ changes from $0$ to $\frac{\pi}{2}$ which is another way to represent the scattering.
\end{enumerate} 
\pagebreak
\chapter{Dynamics of Abelian Higgs vortex lines at the critical coupling}
\label{ch2}
Consider the Minkowski space-time $\mathbb{R}^{1+3}$ and let $(t,x_1,x_2,z)$ denote the coordinates where $t$ is the time. As explained in section \ref{ch1.sec1}, we aim to glue the vortex configuration solutions along the $t$ and $z$-direction in the Minkowski space-time $\mathbb{R}^{1+3}$ to construct solutions for the AHM modulated by wave maps. That is given a wave map $q:[0,T)\to \mathbb{R}\to M_N$, for every $\epsilon>0$ we want to find solutions to AHM close to $q(\epsilon t,\epsilon z)$ in the $(\Phi,A_1,A_2)$ variables and also with small $(A_0,A_3)$ variables, over a domain $[0,\frac{T^{\prime}}{\epsilon})\times \mathbb{R}^3$ for some $T^{\prime}>0$. The precise statement is in section \ref{ch2.mres}.
\\
In gluing constructions, the general recipe is to have an ansatz to construct approximate solutions and then add perturbations to find an honest solution to the underlying PDE. The ansatz involves iteration of a procedure to improve the order of error. There is no reason that by repeating the procedure, we converge to an exact solution to the differential equations. This is similar to say that when we solve an ODE by considering a Taylor series for the solution, the series does not necessarily converge to the solution.
\\
In the literature of PDEs, gluing constructions have been applied to various equations in different contexts, including dispersive equations, fluid dynamics and geometric wave equations. In terms of the nonlinear Schrodinger equations, one can look at \cite{FlWe} by Floer and Weinstein. In fluid dynamics, there are several papers including \cite{DDMW}. In hyperbolic PDEs, once can refer to \cite{PJM}.
\\
 Gluing construction in a gauged system of partial differential equations is the main matter of concern in this thesis. As we will see, considering a correct gauge condition is the main concept here. The crucial point is that the gauge condition in the ansatz and the perturbation part are different from each other. 
\section{Notation}
\begin{enumerate}
\item 
In this chapter, in the context, when there are more than one choice for the gauge fields under discussion, corresponding to the vortex configuration $(\phi,\alpha)$, we use the following notation for the covariant derivative:
\begin{equation}
\label{cdr.eq1}
D_{\mu}^{(0)}=\partial_{_{\mu}}-i\alpha_{_{\mu}}\hspace{.3cm}\mu=1,2
\end{equation}
If there are no other gauge fields available in the context, we use the notation $D_{\mu}$ as in chapter \ref{ch1}.
\item 
Suppose that $(\phi,\alpha)$ is a vortex configuration. Then,
\begin{equation}
\label{eq33.ch2}
\mathcal{L}[\phi,\alpha](\tilde{\varphi},\tilde{a})=\begin{pmatrix}
-\sum_{i=1}^2\big(D_i^2\tilde{\phi}-2i\tilde{a}_iD_i\phi\big)+\frac{1}{2}\big(3|\phi|^2-1\big)\tilde{\varphi}\\
-\Delta \tilde{a}_i+|\phi|^2\tilde{a}_i-2\big(i\tilde{\varphi},D_i\phi\big)
\end{pmatrix}_{i=1,2}
+\begin{pmatrix}
i\phi\big(\nabla.\tilde{a}-\big(i\phi,\tilde{\varphi}\big)\big)\\
\partial_i\big(\nabla.\tilde{a}-\big(i\phi,\tilde{\varphi}\big)\big)
\end{pmatrix}_{i=1,2}
\end{equation}
\item 
Suppose that $(\phi,\alpha)$ is a vortex configuration. Then,
\begin{equation}
\label{eq33.ch2.new}
{L}[\phi,\alpha](\tilde{\varphi},\tilde{a})=\begin{pmatrix}
-\sum_{i=1}^2\big(D_i^2\tilde{\varphi}-2i\tilde{a}_iD_i\phi\big)+\frac{1}{2}\big(3|\phi|^2-1\big)\tilde{\varphi}\\
-\Delta \tilde{a}_i+|\phi|^2\tilde{a}_i-2\big(i\tilde{\varphi},D_i\phi\big)
\end{pmatrix}_{i=1,2}
\end{equation}
\item 
\label{bcs.1}
The components of a point in the spacetime $\mathbb{R}^{1+3}$ are denoted as:
\[
(t,x)=(t,x_1,x_2,z)
\]
and also, we encapsulate $x_1$ and $x_2$ by
\[
y=(x_1,x_2)
\]
Also, we use the coordinate system
\[
(y,\tau,\zeta)=((x_1,x_2),\epsilon t,\epsilon z)
\]
where $\epsilon$ is a scaling parameter in the context.
\item 
In the calculations, $\Delta$ refers to the Laplacian operator on $\mathbb{R}^3$:
\[
\Delta=\partial_{x_1}^2+\partial_{x_2}^2+\partial_{z}^2
\]
and
\[
\Delta_y=\partial_{x_1}^2+\partial_{x_2}^2
\]
\item 
\label{zmod1}
If $(r_{\alpha})_{\alpha}$ denotes the special coordinate system introduced in section \ref{msvg} for the moduli space $M_N$, then the zero modes are denoted by
\[
\tilde{n}_{\alpha}=\frac{\partial}{\partial r_{\alpha}}
\]
\item 
If we apply the Gram-Schmidt process to $(\tilde{n}_{\alpha})_{\alpha}$, we obtain an orthonormal basis $(n_{\alpha})_{\alpha}$ for each point on the tangent space $TM_N$ which depends smoothly on the base point.
\item
\label{ch2.n1}
Suppose that $X,Y,Z$ are normed vector spaces over $\mathbb{R}$ or $\mathbb{C}$ and $U$ is an open subset of $Y$.  We say that a function $f:X\times U\to Z$ is of class $\mathcal{E}_m(X,U,Z)$ if for every multi-index $r$ with $|r|\le m$, one can find $A,\gamma>0$ such that
\[
|D^r f(x,y)|\le Ae^{-\gamma|x|}
\] 
for every $(x,y)\in X\times U$.
\item
\label{ch2.n2}
The components of the zero mode $n_{\alpha}$ or $\tilde{n}_{\alpha}$ at a vortex configuration is denoted by
\[
\big(n_{_{\alpha,\varphi}},n_{_{\alpha,1}},n_{_{\alpha,2}}\big)
\]
or 
\[
\big(\tilde{n}_{_{\alpha,\varphi}},\tilde{n}_{_{\alpha,1}},\tilde{n}_{_{\alpha,2}}\big)
\]
where $n_{_{\alpha,\phi}},\tilde{n}_{_{\alpha,\phi}}$ correspond to the Higgs field  component and $n_{_{\alpha,1}},\tilde{n}_{_{\alpha,1}}$ and $n_{_{\alpha,2}},\tilde{n}_{_{\alpha,2}}$ correspond to the gauge field components.
\item
\label{ch2.n3}
For a set $S$, $\mathcal{R}[S]$ refers to the ring generated by elements of $S$.  
\end{enumerate}
\section{Main result}
\label{ch2.mres}
The final result of this thesis is the following
\begin{theorem}
\label{thm2.main}
Suppose that $q:[0,T]\times \mathbb{R}\to M_N$ is a smooth wave map and it is constant at $(+\infty)$ and $(-\infty)$, then there exist numbers $T_0,\epsilon_0,M$ such that for every $\epsilon$ with $0< \epsilon \le \epsilon_0$, there exists a solution
\[
\big(\Phi^{\epsilon},\sum_{j=0}^3A_j^{\epsilon} dx_j\big)
\]
of the Abelian Higgs model with $\lambda=1$ on 
\[
\big{[}0,\frac{T_0}{\epsilon}\big{)}\times \mathbb{R}^3
\]
with 
\begin{equation}
\label{eq1.new}
\begin{split}
&\Phi^{\epsilon} (t,x,z)=\phi (x;q(\epsilon t, \epsilon z))+{\varphi}^{\epsilon}\\
&A_j^{\epsilon}(t,x,z)=\alpha_j (x;q(\epsilon t, \epsilon z))+ {a}_j^{\epsilon}\hspace{1cm}j=1,2
\end{split}
\end{equation}
where $(\phi,\alpha_1,\alpha_2)$ are the components of the vortex configuration solutions and:
\begin{equation}
\label{eq2}
||\varphi^{\epsilon}(t,.)||_{H^3}+||\partial_t\varphi^{\epsilon}(t,.)||_{H^2}+\sum_{j=1}^2\Big(||a_j^{\epsilon}(t,.)||_{_{H^3}}+||\partial_t a_j^{\epsilon}(t,.)||_{_{H^2}}\Big)\le \epsilon^{\frac{3}{2}} M
\end{equation}
and
\begin{equation}
\label{eq1001.1}
\sum_{j=0,3}||A_j^{\epsilon}(t,.)||_{L^{\infty}}\le \epsilon M
\end{equation}
and
\begin{equation}
\label{eq1001}
\sum_{j=0,3}||\nabla {A}_j^{\epsilon}(t,.)||_{_{H^2}}+||\partial_t {A}_j^{\epsilon}(t,.)||_{_{H^2}}\le \epsilon^{\frac{1}{2}} M
\end{equation}
for some $M$ and every $t$, where $\nabla$ refers to spatial derivatives.
\end{theorem}
\section{Sketch of the proof}
As described above, the two main steps are
\begin{enumerate}
\item
An ansatz to construct approximation solutions
\item
Perturbation of the constructed approximate solution
\\
In subsections \ref{an.up} and \ref{per.up}, we are going to explain the big picture of these two steps. The detailed proof will be in sections \ref{an} and \ref{per}.
\subsection{Ansatz: An overview}
\label{an.up}
In the first step, the ansatz, we look for a series of the form 
\begin{align}
\label{sk1.p1.eq0}
&\Phi_m(t,x)=\phi(y;q(\epsilon t,\epsilon z))+\sum_{i=1}^{m}\epsilon^{2i}\varphi_i(\epsilon t,y,\epsilon z)\\
&A_{m,j}(t,x)=\alpha_j(y;q(\epsilon t,\epsilon z))+\sum_{i=1}^{m}\epsilon^{2i}a_{j,i}(\epsilon t,y,\epsilon z)\hspace{.7cm}{j=1,2}\\
&A_{m,j}(t,x)=\sum_{i=1}^{m}\epsilon^{2i-1}a_{j,i}(\epsilon t,y,\epsilon z)\hspace{2.9cm}{j=0,3}
\end{align}
such that the following conditions are satisfied:
\begin{align}
&S_{\varphi}[\Phi_m,A_m]=D_0D_0\Phi_m-\sum_{j=1}^3D_jD_j\Phi_m+\frac{1}{2}\big(|\Phi_m|^2-1\big)\Phi_m=O(\epsilon ^{2m+2})\label{sk.1ch2.eq1.1}\\
&S_{a_j}[\Phi_m,A_m]=\big(\partial_0\mathcal{F}_{0j}-\sum_{k=1}^3\partial_k\mathcal{F}_{kj}\big)-\big(i\Phi_m,D_j\Phi_m\big)=O(\epsilon ^{2m+2})\hspace{.3cm}j=1,2\label{sk.1ch2.eq1.2}
\\
&S_{a_j}[\Phi_m,A_m]=\big(\partial_0\mathcal{F}_{0j}-\sum_{k=1}^3\partial_k\mathcal{F}_{kj}\big)-\big(i\Phi_m,D_j\Phi_m\big)=O(\epsilon ^{2m+1})\hspace{.3cm}j=0,3\label{sk.1ch2.eq1.3}
\end{align}
The process is inductive. The base case goes in this way:
By looking at the error condition \eqref{sk.1ch2.eq1.3}, we come up with two equations 
\begin{equation}
\label{sk.ch2.eq1.1}
-\Delta_y {a}_{j,1}+|\phi|^2{a}_{j,1}=f_j
\end{equation}
for $j=0,3$ where $f_j$ depends on the wave map. These equations are solvable by lemma \ref{al1} in the appendix. Then, the error condition \eqref{sk.1ch2.eq1.1} and \eqref{sk.1ch2.eq1.2}, leads to the following system of equations
\begin{equation}
\label{ed.eq2}
\mathcal{L}\psi=E
\end{equation}
where $\psi=(\varphi_{1,1},a_{1,1},a_{2,1})$ and $E$ depends on the wave map. (It is the error from the step zero in the ansatz where we simply plug-in the wave map $q$ in the equations and let $A_0=A_3=0$) As mentioned in chapter 1, the operator $\mathcal{L}$ has an infinite dimensional kernel spanned by the infinitesimal gauge transformations and the zero modes. We choose the gauge orthogonality condition for $\psi$:
\begin{equation}
\label{goc.ed1}
\nabla.\tilde{a}-(i\phi,\tilde{\varphi})=0
\end{equation}
Then, equation \eqref{ed.eq2} reads
\begin{equation}
\label{ed.eq3}
L\psi=E
\end{equation}
 Following the coercivity statement \ref{eq35.1} from the Stuart's paper, we have proved in lemma \ref{al2} in the appendix that the system \eqref{ed.eq3} is solvable if $E$ is orthogonal to zero modes and the solution $\psi$ will be again in the orthogonal subspace to zero modes. This orthogonality condition leads to the wave map equation. But, now we need to check that $\psi$ satisfies the gauge orthogonality condition. In lemma \ref{al2} in the appendix, we have proved that if $E$ satisfies the gauge orthogonality condition, then $\psi$ also does. We separately proved that $E$ satisfies this condition and this finishes the first step. 
\\
The induction step goes in this way: Suppose that we have constructed an approximate solution $(\Phi_m,A_m)$ in the form of \eqref{sk1.p1.eq0}. If we want to update the number $m$ and follow the similar procedure as the first step explained above, we encounter the issue that the orthogonality condition may not be satisfied. In this regard, we use some degrees of freedom in the current stage for the approximate solution $(\Phi_m,A_m)$ that we have not used so far. In \eqref{sk1.p1.eq0}, we can replace the vector 
\[
\psi_m=({\varphi}_m,a_{1,m},a_{2,m})
\]
by 
\[
\psi_m+\sum_{\mu}c_{\mu}(\epsilon t,\epsilon z)n_{\mu}(y)
\]
for any $t$ and $z$ and this does not disturb the leading order term of the error from the approximate solution $(\Phi_m,A_m)$, as the zero modes belong to the kernel of the operators $\mathcal{L}$ and $L$. 
\\
Therefore, we consider the following:
\begin{equation}
\begin{split}
\begin{pmatrix}
(\Phi_{m+1}, A_{_{m+1,1}}, A_{_{m+1,2}})
\\
\\
(A_{_{m+1,0}},A_{_{m+1,3}})
\end{pmatrix}
&=\begin{pmatrix}
(\Phi_{m}, A_{_{m,1}}, A_{_{m,2}})
\\
\\
(A_{_{m,0}},A_{_{m,3}})
\end{pmatrix}
+\epsilon^{2m}
\begin{pmatrix}
\sum_{{\mu}}c_{\mu}(\epsilon t,\epsilon z)n_{\mu}(y)
\\
\\
0
\end{pmatrix}
\\&
\\&
+\begin{pmatrix}
\epsilon^{2m+2}(\varphi_{_{m+1}},{a}_{_{1,m+1}},a_{_{2,m+1}})(\epsilon t,y,\epsilon z)
\\
\\
\epsilon^{2m+1}\big({a}_{_{0,m+1}},a_{_{3,m+1}}\big)(\epsilon t,y,\epsilon z)
\end{pmatrix}
\end{split}
\end{equation}
We let $\{{c_{\mu}}\}_{\mu}$ to be unknown. Now, we repeat the same process and let the new functions in the ansatz depend on $\{c_{\mu}\}_{\mu}$. When we want to find $(\varphi_{m+1},a_{1,m+1},a_{2,m+1})$, we observe that the orthogonality condition leads to a hyperbolic PDE for the functions $\{c_{\mu}\}_{\mu}$ which has local well-posedness. We find $\{c_{\mu}\}_{\mu}$ in this way. Afterwards, the other pieces of the ansatz will be automatically determined.
\subsection{Perturbation of approximate solutions: An overview}
\label{per.up}
Suppose that we have found an approximate solution $v=(\varphi,a)$ as mentioned above. We look for a perturbation $u=(\tilde{\varphi},\tilde{a})$ such that $v+u$ is a solution for the AHM. The strategy is as follows:
\begin{itemize}
\item
\textbf{Imposing a gauge condition:}
\begin{equation}
\label{sk.chp2.gc}
-\partial_0\tilde{a}_0+\sum_{j=1}^3\partial_j\tilde{a}_j-(i\varphi,\tilde{\varphi})=0
\end{equation}
\item
\textbf{Rewriting the equations based on the gauge condition:}
\\
Let
\[
\psi=(\tilde{\varphi},\tilde{a}_1dx_1+\tilde{a}_2dx_2)
\]
After imposing the gauge condition \ref{sk.chp2.gc}, the equations can be written as:
\begin{equation}
\label{sk.ch2.rw1}
u_{tt}-u_{zz}+Mu=Pu+N+E
\end{equation}
where 
\begin{equation}
\label{sk.ch2.rw2}
Mu=\begin{pmatrix}
L\psi
\\
-\Delta_y \tilde{a}_0+|\phi|^2 \tilde{a}_0
\\
-\Delta_y \tilde{a}_3+|\phi|^2 \tilde{a}_3
\end{pmatrix}
\end{equation}
and $Pu$ consists of small linear terms, $Nu$ consists of nonlinear terms and $E$ is the error from the approximate solution $v$.
\item 
\textbf{Local Existence:}
\\
For the system of equations such as \eqref{sk.ch2.rw2}, there are some standard local wellposed-ness results. We apply such an statement to find a solution defined over an interval $[0,T_0)$. These statements come with an apriori estimate.
\item 
\textbf{Bootstrap:}
Consider the constructed localized in time solution in the previous step. We want to use equations \eqref{sk.chp2.gc} and the apriori estimates to find refined estimates for $u$. The idea is again to look at each slice $\{t=t_0,z=z_0\}$ and look at the orthogonal and tangential components to zero modes. 
\\
Let us fix the time as $t$. Consider the following energy quantities:
\begin{equation}
\label{sk.ch2.b1}
Q_1(t)=\int_{\mathbb{R}^3}|u_t|^2+|u_z|^2+(Mu,u)
\end{equation}
and
\begin{equation}
\label{sk.ch2.b2}
Q_2(t)=\int_{\mathbb{R}^3}|(Mu)_t|^2+|(Mu)_z|^2+(M^2u,Mu)
\end{equation}
By using equations \eqref{sk.ch2.rw1}, one can verify that 
\begin{equation}
\label{sk.ch2.b3}
Q_1^{\prime}(t)+Q_2^{\prime}(t)\le C\epsilon F(||u||_{_{H^3}},||u_t||_{_{H^2}}, ||E||_{_{H^2}}, ||N||_{_{H^2}})
\end{equation}
where $F$ is an smooth function with $F(0)=DF(0)=0$
\\
The quantity $Q_1(t)+Q_2(t)$ bounds $\tilde{a}_0,\tilde{a}_3,(\tilde{a}_0)_t, (\tilde{a}_3)_t$ and $\psi^{\perp}$ where for every $t,z$ the function $\psi^{\perp}$ is such that
\begin{equation}
\label{sk.ch2.b4}
\psi(.,t,z)=\psi^{\perp}(.,t,z)+\sum_{\mu}c_{\mu}(t,z)n_{\mu}(.)
\end{equation}
where $\{n_{\mu}\}$ are the corresponding zero modes and 
\[\langle\psi^{\perp},n_{\mu}\rangle=0\]
We write the variables $\tilde{a}_0$ and $\tilde{a}_3$ in the form:
\begin{equation}
\label{sk.ch2.b4.1}
\tilde{a}_j=f_j+\partial_0 \chi_j
\end{equation}
where $\chi_j$ solves the equations
\begin{equation}
\label{sk.ch2.b4.2}
(-\Delta+|\varphi|^2)\chi_j=\partial_0 \tilde{a}_j
\end{equation}
By taking the inner product of equation \eqref{sk.ch2.rw1} in the Higgs part and the gauge field part for $A_1$ and $A_2$, we observe that
\begin{equation}
\label{sk.ch2.b5}
(c_{\mu})_{tt}-(c_{\mu})_{zz}=\partial_0 h_0+\partial_3h_3+h
\end{equation}
for some function $h_0,h_1$ and $h_3$. Then, we use the explicit formula for the wave equations to find $c_{\mu}$ and the final expression for $c_{\mu}$ can be expressed in terms of $h_0, h_3$ and $h$, rather than $\partial_0h_0$ or $\partial_3 h_3$. The expression of the function $h_0$ involves $\chi_j$. Based on this, we have good enough estimates for $h_0,h_3$ and $h$ and hence for $c_{\mu}$. 
\\
We also use similar ideas to control $\partial_t\psi$ by decomposing it into zero modes and orthogonal components. In the end, we observe that these estimates are enough to control $||u||_{_{H^3}}+||u_{t}||_{_{H^2}}$.
\\
Using the aforementioned estimates, we afford to do a bootstrap argument. That is we prove that if we have an estimate for $u$ on a time interval $[0,T)$:
\begin{equation}
\label{res}
||u||_{_{H^3}}+||u_t||_{_{H^2}}\le K\le \epsilon^3
\end{equation}
Then, 
\begin{equation}
\label{res.1}
||u||_{_{H^3}}+||u_t||_{_{H^2}}\le C(\epsilon t)^{\frac{1}{2}}K+\epsilon^4
\end{equation}
\item {\textbf{Moving forward in time:}}
\\
Estimate \ref{res.1} suffices to go through an iterative process of applying a local existence result and apriori estimate, refining the estimate, again applying a local existence theorem to the end point of the interval and again refining the estimates and so on. In this way, we prove the desired estimates over a time interval of length $\frac{\kappa}{\epsilon}$ for some $\kappa$ independent of $\epsilon$
\end{itemize}
\end{enumerate}
\section{Approximate solutions for the Abelian Higgs Model}
\label{an}
For the complex function
$\Phi:\mathbb{R}^{1+3}\to\mathbb{C}$ and the real 1-form
\[
A=A_0dt+\sum_{i=1}^3A_idx^i
\]
consider the quantities:
\begin{equation}
\label{ch2.eq1}
\begin{split}
&S_{\varphi}[\Phi,A]=D_0D_0\Phi-\sum_{j=1}^3D_jD_j\Phi+\frac{1}{2}\big(|\Phi|^2-1\big)\Phi\\
&S_{a_j}[\Phi,A]=\big(\partial_0\mathcal{F}_{0j}-\sum_{k=1}^3\partial_k\mathcal{F}_{kj}\big)-\big(i\Phi,D_j\Phi\big)\hspace{1cm}j=0,1,2,3
\end{split}
\end{equation}
Let
\[
S[\Phi,A]=\big(S_{\varphi}[\Phi,A],S_{a_0}[\Phi,A],\cdots,S_{a_3}[\Phi,A]\big)
\]
For given $\epsilon>0$, we find a pair $(\Phi_n,A_n)$ such that
\[
S[\Phi_n,A_n]=O(\epsilon^{2n+1},\epsilon^{2n-1},\epsilon^{2n+1},\epsilon^{2n+1},\epsilon^{2n-1})
\] 
and
\[
\begin{pmatrix}
\Phi_n\\
A_n
\end{pmatrix}
(t,x)=\begin{pmatrix}
(\phi,\alpha)(q(y;\epsilon t,\epsilon z))\\
0
\end{pmatrix}
+\begin{pmatrix}
O(\epsilon^2)
\\
O(\epsilon)
\end{pmatrix}
\]
where $q:[0,T]\times \mathbb{R}\to M_N$ is a wave map. The precise statement is as follows:
\begin{proposition}
\label{2.3.p1}
Suppose that $q:[0,T]\times\mathbb{R}\to M_N$ is a smooth wave map with compactly supported image and bounded derivatives of any order. There exists $\epsilon_0$ such that if $\epsilon<\epsilon_0$, then for every $m$ one can find smooth $(\Phi_m,A_m)$ the form
\begin{equation}
\label{p1.eq0.new}
\begin{split}
&\Phi_m(t,x)=\phi(y;q(\epsilon t,\epsilon z))+\sum_{i=1}^{m}\epsilon^{2i}\varphi_i(\epsilon t,y,\epsilon z)\\
&A_{m,j}(t,x)=\alpha_j(y;q(\epsilon t,\epsilon z))+\sum_{i=1}^{m}\epsilon^{2i}a_{j,i}(\epsilon t,y,\epsilon z)\hspace{.7cm}{j=1,2}\\
&A_{m,j}(t,x)=\sum_{i=1}^{m}\epsilon^{2i-1}a_{j,i}(\epsilon t,y,\epsilon z)\hspace{2.9cm}{j=0,3}
\end{split}
\end{equation}
defined on $[0,\frac{T_m}{\epsilon}]\times \mathbb{R}$ for some $T_m>0$ such that
if 
\[
A_m=(A_{m,j})_{_{j=0}}^{^{3}}
\]
then
\begin{equation}
\label{p1.eq0.1}
\begin{pmatrix}
S_{\varphi}[\Phi_m,A_m]\\
S_{a_1}[\Phi_m,A_m]\\
S_{a_2}[\Phi_m,A_m]
\end{pmatrix}(t,x)=\sum_{i=m+1}^{3m}\epsilon^{2i}f_i(\epsilon t,y,\epsilon z)
\end{equation}
and
\begin{equation}
\label{p1.eq0.2}
(S_{a_j}[\Phi_m,A_m])(t,x)=\sum_{i=m+1}^{3m}\epsilon^{2i-1}g_{j,i}(\epsilon t,y,\epsilon z)\hspace{.4cm}j=0,3
\end{equation}
where if
\begin{equation}
h\in \{f_i,g_{j,i},\varphi_i\}_{_{i,j}}\cup\{a_{j,i}\}_{_{j=1,2}}\cup\{a_{j,i}\}_{_{ i>1,j\notin\{1,2\}}}
\end{equation}
then for any multi-index $r$, one can find $A,\delta>0$ such that
\begin{equation}
\label{prop.ans.eq1}
|D^rh(\tau,y,\zeta)|\le Ae^{-\delta|y|}
\end{equation}
and 
\begin{equation}
\label{prop.ans.eq2}
|D^r a_{j,1}(\tau,y,\zeta)|\le A(|y|+1)^{-|r|-1}
\end{equation}
for $j\in \{0,3\}$ and any $y,\tau,\zeta$.
\end{proposition}
\begin{proof}
The proof of this statement goes by induction on $m$. Following notation \ref{ch2.n1}, $\mathcal{E}$ refers to the class of all functions $f:\mathbb{R}^2\times [0,S]\times \mathbb{R}^{1}\to \mathbb{C}$ for some $S$ such that for every multi-index $r$, one can find $A,\gamma>0$ such that
\[
D^rf(y,\tau,\zeta)\le Ae^{-\gamma|y|}
\]
for every $y,\tau,\zeta$.
\\
Suppose that $m=1$. Let
\begin{equation}
\label{p1.eq1}
\begin{split}
&\Phi_1(t,x)=\phi(y;q(\tau, \zeta))+\epsilon^2{\varphi}_1(\tau,y,\zeta)\\
&A_{1,j}(t,x)=\alpha_j(y;q(\tau, \zeta))+\epsilon^2{a}_{{1,j}}(\tau,y,\zeta)\hspace{.8cm}j=1,2\\
&A_{1,j}(t,x)=\epsilon {\tilde{a}}_{{j}}(\tau,y,\zeta)\hspace{3.5cm}j=0,3
\end{split}
\end{equation}
For simplicity, we write
\begin{equation}
\label{p1.eq0}
\begin{split}
&\phi(.)=\phi(.;q(\tau, \zeta))\\
&\alpha_j(.)=\alpha_j(.;q(\tau,\zeta))\hspace{.5cm}j=1,2
\end{split}
\end{equation}
Following notation \ref{bcs.1}, assume that $\tau=\epsilon t$ and $\zeta=\epsilon z$ and we use the coordinate system
\[
(y,\tau,\zeta)=((x_1,x_2),\epsilon t,\epsilon z)
\]
By plugging in the ansatz \eqref{p1.eq1} into \eqref{ch2.eq1} and using the fact that
\[
\partial_{t}=\epsilon \partial_{\tau}\hspace{.2cm},\hspace{.2cm}\partial_{z}=\epsilon \partial_{\zeta}
\]
one gets:
\begin{equation}
\label{p1.eq1.9}
\begin{split}
S_{a_j}[\Phi_1,A_1]=\epsilon g_{j,1}+\epsilon^3g_{j,2}+\epsilon^5g_{j,3}\hspace{.5cm}j=0,3
\end{split}
\end{equation}
where the functions $g_{j,i}(\tau,y,\zeta)$ are independent of $\epsilon$ and
\begin{equation}
\label{p1.eq1.9.1}
\begin{split}
g_{j,1}=-\big(\Delta_y \tilde{a}_{j}-|\phi|^2\tilde{a}_{j}+(i\phi,\frac{\partial \phi }{\partial \tau})\big)\hspace{.4cm}j=0
\\
g_{j,1}=-\big(\Delta_y \tilde{a}_{j}-|\phi|^2\tilde{a}_{j}+(i\phi,\frac{\partial \phi }{\partial \zeta})\big)\hspace{.4cm}j=3
\end{split}
\end{equation}
where $\Delta_y$ refers to the laplacian with respect to the variables $x_1,x_2$. We find the functions $\tilde{a}_{{j}}$ such that $g_{j,1}=0$ for $j=0,3$ for every $\tau,\zeta$. These equations are equivalent to the gauge orthogonality condition for the vectors
\[
\Big(\frac{\partial \phi}{\partial \tau}-i\phi \tilde{a}_{0},\frac{\partial \alpha}{\partial \tau}-d \tilde{a}_{0}\Big)
\]
and
\[
\Big(\frac{\partial \phi}{\partial \zeta}-i\phi \tilde{a}_{3},\frac{\partial \alpha}{\partial \zeta}-d \tilde{a}_{3}\Big)
\]
Therefore, by following the strategy in chapter 1 to make the above vectors gauge orthogonal, we can construct the functions $\tilde{a}_{{j}}$ for $j=0,3$. The construction is as follows:
\\
Suppose that
\[
\tilde{\partial}_0=\partial_{\tau}
\]
and
\[
\tilde{\partial}_3=\partial_{{\zeta}}
\]
Suppose that $R>0$ is such that the ball $B(0,R)$ contains all of the vortex centers in the image of the wave map $q$. Suppose that $\rho:\mathbb{R}^2\to \mathbb{R}$ is a nonnegative smooth cutoff function which is $0$ inside the ball $B(0,{R})$ and is $1$ outside of the ball $B(0,2{R})$. Then,
\begin{equation}
\label{p1.eq1.9.2}
\tilde{a}_j=-\frac{1}{2}\rho\tilde{\partial}_j\Theta+b_j
\end{equation}
where $\Theta$ is as in \eqref{eq18} and $b_j$ solves the equation
\begin{equation}
\label{p1.eq1.9.3}
\Delta_y b_j-|\phi|^2b_j=c_j
\end{equation}
for each $\tau,\zeta$ and $c_j=c_j(y,\tau,\zeta)$ is a smooth function of $y,q$ and $\tilde{\partial}_j{q}$ and is supported inside the ball $B(0,2{R})$ in the $y$-plane. By the regularity properties of $q$ and using lemma \ref{l3} in the appendix, one realizes that $b_j\in \mathcal{E}$ and therefore by \eqref{p1.eq1.9.2}, $\tilde{a}_j$ depends smoothly on $y,\tau,\zeta$ and 
\[
(\partial_j \tilde{a}_l-\partial_l \tilde{a}_j) \in \mathcal{E}
\] 
for any $j,l\in \{0,3\}$. Furthermore, according to the formula and \ref{p1.eq1.9.2} and lemma \ref{al0} in the appendix, for any multi-index $r$, there exists $A>0$ such that
\begin{equation}
\label{p1.eq1.9.4}
|D^r \tilde{a}_j(y,\tau,\zeta)|\le A (|y|+1)^{-|r|+1}
\end{equation}
for every $\tau,\zeta$. We define
\[
\tilde{D}_j=\tilde{\partial}_j-i\tilde{a}_j
\]
for $j=0,3$. Also, Suppose that 
\[
\eta=(\phi,\alpha_1,\alpha_2)
\]
and
\[
\psi_1=(\varphi_1,a_{1,1},a_{1,2})
\]
By using the fact that $\eta$ satisfies the static two dimensional AHM for each $\tau,\zeta$, we have:
\begin{equation}
\label{p1.eq2}
\begin{pmatrix}
S_{\varphi}[\Phi_1,A_1]\\
S_{a_j}[\Phi_1,A_1]
\end{pmatrix}_{j=1,2}
=\epsilon^2 (T+\mathcal{L}\psi_1)+\epsilon^4f_2+\epsilon^6 f_3
\end{equation}
where $f_2(\tau,y,\zeta)$ and $f_3(\tau,y,\zeta)$ do not depend on $\epsilon$ and
\[
T=\begin{pmatrix}
\big(\tilde{D}_{_0}^{^2}-\tilde{D}_{_3}^{^2}\big)\phi\\
\\
\big(\tilde{\partial}_{_0}^{^2}-\tilde{\partial}_{_{3}}^{^2}\big)\alpha_j-\partial_{j}\big(\tilde{\partial}_{0}\tilde{a}_{0}-\tilde{\partial}_{{3}}\tilde{a}_{3}\big)
\end{pmatrix}_{_{j=1,2}}
\]
\begin{proposition}
\label{c1}
The vector $T$ is orthogonal to the zero modes at $\eta$ for any $\tau,\zeta$.
\end{proposition}
\begin{Remark}
In fact, we will prove that for $(\phi,\alpha)(.,q(\tau,\zeta))$ and $\tilde{a}_j$ for $j=0,3$ as above, the vector $T$ is orthogonal to all zero modes, if and only if $q$ is a wave map.
\end{Remark}
\begin{proof}
In what follows, Greek letters correspond to the indices for the coordinate system on the moduli space $M_N$. We have:
\begin{equation}
\label{p1.eq3}
T=\tilde{\partial}_0\tilde{D}_0\eta-\tilde{\partial}_3\tilde{D}_3\eta-\Big(i\big(\tilde{a}_0\tilde{D}_0-\tilde{a}_3\tilde{D}_3\big)\phi,0\Big)
\end{equation}
We have:
\begin{equation}
\label{p1.eq4}
\tilde{D}_j\eta=\sum_{\alpha}(\tilde{\partial}_j q^{\alpha}) \tilde{n}_{\alpha}
\end{equation}
for $j=0,3$ where
\[
q=(q^{\alpha})_{\alpha}
\]
in the coordinate system introduced in the first chapter and $\tilde{n}_{\alpha}$'s denote the corresponding zero modes introduced in notation \ref{zmod1}. Therefore,
\begin{equation}
\label{p1.eq5}
T=\sum_{{\alpha}}\Big((\Box q^{\alpha}) \tilde{n}_{_{\alpha}}+(\tilde{\partial}_0q^{\alpha})(\tilde{\partial}_0 \tilde{n}_{\alpha})-(\tilde{\partial}_3q^{\alpha})(\tilde{\partial}_3\tilde{n}_{\alpha})\Big)-\sum_{\alpha}\Big(i(\tilde{a}_0\tilde{\partial}_0 q^{\alpha}-\tilde{a}_3\tilde{\partial}_3 q^{\alpha})\tilde{n}_{_{\alpha,\varphi}},0\big)\Big)
\end{equation}
Therefore,
\begin{equation}
\label{p1.eq6}
\begin{split}
\langle T,\tilde{n}_{\beta}\rangle &=\sum_{\alpha}\Box q^{\alpha}\langle \tilde{n}_{\alpha},\tilde{n}_{\beta}\rangle +\sum_{\alpha}\big((\tilde{\partial}_0q^{\alpha})\langle\tilde{\partial}_0 \tilde{n}_{\alpha},\tilde{n}_{\beta}\rangle-(\tilde{\partial}_3q^{\alpha})\langle\tilde{\partial}_3\tilde{n}_{\alpha},\tilde{n}_{\beta}\rangle\big)\\
&-\sum_{\alpha}(\tilde{a}_0\tilde{\partial}_0 q^{\alpha}-\tilde{a}_3\tilde{\partial}_3 q^{\alpha})\langle i\tilde{n}_{_{\alpha,\varphi}},\tilde{n}_{_{\beta,\varphi}}\rangle
\end{split}
\end{equation}
Since $q$ is a wave map, we have:
\begin{equation}
\label{p1.eq7}
\begin{split}
\Box q^{\alpha}\langle \tilde{n}_{\alpha},\tilde{n}_{\beta}\rangle&=\Gamma^{\alpha}_{\mu\lambda}\big(-\tilde{\partial}_0q^{\mu}\tilde{\partial}_0q^{\lambda}+\tilde{\partial}_3q^{\mu}\tilde{\partial}_3q^{\lambda}\big)\langle \tilde{n}_{\alpha},\tilde{n}_{\beta}\rangle\\
&=\big(-\tilde{\partial}_0q^{\mu}\tilde{\partial}_0q^{\lambda}+\tilde{\partial}_3q^{\mu}\tilde{\partial}_3q^{\lambda}\big)\langle\nabla_{\tilde{n}_{_{\mu}} \tilde{n}_{_{\lambda}},\tilde{n}_{\beta}}\rangle
\end{split}
\end{equation}
where $\{\Gamma^{\alpha}_{\mu\lambda}\}$ denote the Christoffel symbols of the Levi Civita connection corresponding to the metric on $M_N$. According to the Koszul formula, we have:
\begin{equation}
\label{p1.eq8}
\langle\nabla_{\tilde{n}_{\mu}}\tilde{n}_{\lambda},\tilde{n}_{\beta}\rangle=\frac{1}{2}\big(\partial_{\mu}\langle \tilde{n}_{\lambda},\tilde{n}_{\beta}\rangle+\partial_{\lambda}\langle \tilde{n}_{\mu},\tilde{n}_{\beta}\rangle-\partial_{\beta}\langle \tilde{n}_{\lambda},\tilde{n}_{\mu}\rangle \big)
\end{equation}
By combining \eqref{p1.eq7} and \eqref{p1.eq8} and then using \eqref{p1.eq6}, we get:
\begin{equation}
\label{p1.eq9}
\begin{split}
\langle T,\tilde{n}_{\beta}\rangle=\sum_{k\neq 1,2}s(k)\Big[(\tilde{\partial}_kq^{\mu})(\tilde{\partial}_kq^{\lambda})\langle\nabla_{\tilde{n}_{\mu}} \tilde{n}_{\lambda}-\partial_{\lambda}\tilde{n}_{\mu},\tilde{n}_{\beta}\rangle+\tilde{a}_k(\tilde{\partial}_{k}q^{\mu})\langle i\tilde{n}_{\mu,\varphi},\tilde{n}_{\beta,\varphi}\rangle\Big]\\
\end{split}
\end{equation}
where
\begin{equation}
\label{p1.eq9.1}
s(k)=\begin{cases}
1, & \text{if $k=0$}.\\
-1, & \text{if $k>0$}.
\end{cases}
\end{equation}
But,
\begin{equation}
\label{p1.eq10}
\tilde{a}_k=-(\tilde{\partial}_{k}q^{\lambda})\chi_{\lambda}
\end{equation}
where $\chi_{_\lambda}:\mathbb{R}^2\to \mathbb{R}$ is the function which makes the zero mode
\[
\tilde{n}_{\lambda}=(\partial_{\lambda}{\phi}+i\chi_{_\lambda}\phi, \partial _{\lambda}\alpha+d\chi_{_\lambda})
\] 
gauge orthogonal. Therefore,
\begin{equation}
\label{p1.eq11}
\begin{split}
\langle T,\tilde{n}_{\beta}\rangle&=\sum_{k\neq 1,2}s(k)(\tilde{\partial}_kq^{\mu})(\tilde{\partial}_kq^{\lambda})\big(\langle\nabla_{\tilde{n}_{\mu}} \tilde{n}_{\lambda}-\partial_{\lambda}\tilde{n}_{\mu},\tilde{n}_{\beta}\rangle-\chi_{\lambda}\langle i\tilde{n}_{\mu,\varphi},\tilde{n}_{\beta,\varphi}\rangle\big)\\
&=\sum_{k\neq 1,2}s(k)\Big[\sum_{\mu<\lambda}(\tilde{\partial}_kq^{\mu})(\tilde{\partial}_kq^{\lambda})\big(\langle2\nabla_{\tilde{n}_{\mu}} \tilde{n}_{\lambda}-\partial_{\lambda}\tilde{n}_{\mu}-\partial_{\mu}\tilde{n}_{\lambda},\tilde{n}_{\beta}\rangle-\langle i\chi_{_\lambda}\tilde{n}_{\mu,0}+i\chi_{\mu}\tilde{n}_{\lambda,0},\tilde{n}_{\beta,0}\rangle\big)\Big]\\
&+\sum_{k\neq 1,2}s(k)\Big[\sum_{\mu}(\tilde{\partial}_kq^{\mu})^2\big(\langle \nabla_{\tilde{n}_{\mu}} \tilde{n}_{\mu}-\partial_{\mu}\tilde{n}_{\mu},n_{\beta}\rangle-\chi_{\mu}\langle i\tilde{n}_{\mu,\varphi},\tilde{n}_{\beta,\varphi}\rangle\big)\Big]\\
\end{split}
\end{equation}

But according to the Koszul's formula, we have:
\begin{equation}
\label{p1.eq12}
\begin{split}
2(\nabla _{\tilde{n}_{\mu}}\tilde{n}_{\lambda},\tilde{n}_{\beta})&=\partial_{\lambda}\langle \tilde{n}_{\mu},\tilde{n}_{\beta} \rangle+\partial_{\mu}\langle \tilde{n}_{\lambda},\tilde{n}_{\beta} \rangle-\partial_{\beta} \langle \tilde{n}_{\mu},\tilde{n}_{\lambda} \rangle\\
&=\langle \tilde{n}_{\mu},\partial_{\lambda}\tilde{n}_{\beta} \rangle+\langle \partial_{\mu}\tilde{n}_{\lambda},\tilde{n}_{\beta} \rangle+ \langle \tilde{n}_{\lambda},\partial_{\mu}\tilde{n}_{\beta} \rangle+\langle \partial_{\lambda}\tilde{n}_{\mu},\tilde{n}_{\beta} \rangle-\langle \partial_{\beta}\tilde{n}_{\lambda},\tilde{n}_{\mu} \rangle-\langle \tilde{n}_{\mu},\partial_{\beta}\tilde{n}_{\lambda} \rangle
\end{split}
\end{equation}
Therefore,
\begin{equation}
\label{p1.eq14}
\begin{split}
\langle 2\nabla_{\tilde{n}_{\mu}} \tilde{n}_{\lambda}-\partial_{\lambda}\tilde{n}_{\mu}-\partial_{\mu}\tilde{n}_{\lambda},\tilde{n}_{\beta}\rangle=\langle \tilde{n}_{\mu},\partial_{\lambda}\tilde{n}_{\beta}-\partial_{\beta}\tilde{n}_{\lambda}\rangle+\langle \tilde{n}_{\lambda},\partial_{\mu}\tilde{n}_{\beta}-\partial_{\beta}\tilde{n}_{\mu}\rangle
\end{split}
\end{equation}
We have:
\begin{equation}
\label{p1.eq15}
\tilde{n}_{l}=\big(\partial_{_{l}}{\phi}+i\chi_{_l}\phi,\partial_{_l}\alpha+d\chi_{_l}\big)
\end{equation}
for any $l$. Therefore,
\begin{equation}
\label{p1.eq16}
\begin{split}
\partial_{\lambda}\tilde{n}_{\beta}-\partial_{\beta}\tilde{n}_{\lambda}&=\Big(i\chi_{\beta}\partial_{\lambda}\phi-\chi_{\lambda}\partial_{\beta}\phi+i(\partial_{\lambda}\chi_{\beta}-i\partial_{\beta}\chi_{\lambda})\phi,d\partial_{_\lambda}\chi_{_\beta}-d\partial_{_\beta}\chi_{_\lambda}\Big)
\end{split}
\end{equation}
for any $\lambda,\beta$. Therefore,
\begin{equation}
\label{p1.eq18}
\begin{split}
\langle \tilde{n}_{\mu},\partial_{\lambda}\tilde{n}_{\beta}-\partial_{\beta}\tilde{n}_{\lambda}\rangle+\langle \tilde{n}_{\lambda},\partial_{\mu}\tilde{n}_{\beta}-\partial_{\beta}\tilde{n}_{\mu}\rangle&=\langle \tilde{n}_{\mu,\phi},i\chi_{\beta}\partial_{\lambda}\phi-i\chi_{\lambda}\partial_{\beta}\phi\rangle+\langle \tilde{n}_{\lambda,\varphi},i\chi_{\beta}\partial_{\mu}\phi-i\chi_{\mu}\partial_{\beta}\phi\rangle\\
&+\langle \tilde{n}_{\mu},(iA_{\lambda,\beta}\phi,dA_{\lambda,\beta})\rangle+\langle \tilde{n}_{\lambda},(iA_{\mu,\beta}\phi,dA_{\mu,\beta})\rangle
\end{split}
\end{equation}
where
\[
A_{rs}=\partial_r\chi_s-\partial_s\chi_r
\]
Since $\tilde{n}_l$'s satisfy the gauge orthogonality condition, we have:
\begin{equation}
\label{p1.eq19}
\langle \tilde{n}_{\mu},(iA_{\lambda,\beta}\phi,dA_{\lambda,\beta})\rangle=\langle \tilde{n}_{\lambda},(iA_{\mu,\beta}\phi,dA_{\mu,\beta})\rangle=0
\end{equation}
Therefore,
\begin{equation}
\label{p1.eq20}
\langle \tilde{n}_{\mu},\partial_{\lambda}\tilde{n}_{\beta}-\partial_{\beta}\tilde{n}_{\lambda}\rangle+\langle \tilde{n}_{\lambda},\partial_{\mu}\tilde{n}_{\beta}-\partial_{\beta}\tilde{n}_{\mu}\rangle=\langle \tilde{n}_{\mu,\varphi},i\chi_{\beta}\partial_{\lambda}\phi-i\chi_{\lambda}\partial_{\beta}\phi\rangle+\langle \tilde{n}_{\lambda,\varphi},i\chi_{\beta}\partial_{\mu}\phi-i\chi_{\mu}\partial_{\beta}\phi\rangle
\end{equation}
Now, by using \eqref{p1.eq14}, \eqref{p1.eq20} and the identity \eqref{p1.eq15} for $l=\mu,\lambda,\beta$, one sees that
\begin{equation}
\label{p1.eq21}\langle2\nabla_{\tilde{n}_{\mu}} \tilde{n}_{\lambda}-\partial_{\lambda}\tilde{n}_{\mu}-\partial_{\mu}\tilde{n}_{\lambda},\tilde{n}_{\beta}\rangle-\langle i\chi_{\lambda}\tilde{n}_{\mu,\varphi}+i\chi_{\mu}\tilde{n}_{\lambda,\varphi},\tilde{n}_{\beta,\varphi}\rangle=0
\end{equation} 
Therefore, by \eqref{p1.eq11}, one deduces that $\langle T,\tilde{n}_{\beta}\rangle =0$.
\end{proof}
\begin{lemma}
\label{c1.l1}
For any choice for $\Phi,A$, we always have
\begin{equation}
\label{p1.eq20.9}
\big(S_{\varphi}[\Phi,A],i\Phi\big)+\partial_0S_{a_0}[\Phi,A]-\sum_{j=1}^3\partial_j S_{a_j}[\Phi,A]=0
\end{equation}
\end{lemma}
\begin{proof}
We compute
\begin{equation}
\label{p1.eq21.new}
\big(S_{\varphi}[\Phi,A],i\Phi\big)=\big(D_0D_0\Phi-\sum_{j=1}^3D_jD_j\Phi,i\Phi\big)
\end{equation}
Also,
\begin{equation}
\label{p1.eq22}
\begin{split}
-\partial_0S_{a_0}+\sum_{j=1}^3\partial_jS_{a_j}[\Phi,A]&=\partial_0\big(\sum_{j=1}^n\partial_j\mathcal{F}_{j0}+(i\Phi,D_0\Phi)\big)+\sum_{j=1}^n\partial_j\big(\partial_0\mathcal{F}_{0j}-\sum_{k=1}^n\partial_{k}{\mathcal{F}_{kj}}-(i\Phi,D_j\Phi)\big)\\
&=\sum_{j=1}^n\partial_0\partial_j(\mathcal{F}_{j0}+\mathcal{F}_{0j})-\sum_{j=1}^n\sum_{k=1}^n\partial_j\partial_k\mathcal{F}_{kj}+(i\partial_0\Phi,D_0\Phi)+(i\Phi,\partial_0 D_0\Phi)\\&-\sum_{j=1}^n\big((i\partial_j\Phi,D_j\Phi)+(i\Phi,\partial_jD_j\Phi)\big)
\end{split}
\end{equation}
But, we have
\[
\mathcal{F}_{\alpha\beta}+\mathcal{F}_{\beta\alpha}=0
\]
for any $\alpha,\beta$. Therefore, by \eqref{p1.eq22}, we have:
\begin{equation}
\label{p1.eq23}
\begin{split}
-\partial_0S_{A_0}+\sum_{j=1}^3\partial_jS_{a_j}[\Phi,A]&=(i\partial_0\Phi,D_0\Phi)+(i\Phi,\partial_0 D_0\Phi)-\sum_{j=1}^3\big((i\partial_j\Phi,D_j\Phi)-(i\Phi,\partial_jD_j\Phi)\big)\\&=(i\Phi,D_0 D_0\Phi)-\sum_{j=1}^3(i\Phi,D_jD_j\Phi)
\end{split}
\end{equation}
\eqref{p1.eq20.9} follows from \eqref{p1.eq21.new} and \eqref{p1.eq23}.
\end{proof}
\begin{proposition}
\label{c2.p3}
The vector $T$ satisfies the gauge orthogonality condition for any $\tau,\zeta$.
\end{proposition}
\begin{proof}
Let
\begin{equation}
\label{p1.eq24}
\begin{split}
&\tilde{\Phi}_1(t,x)=\phi(y;q(\tau, \zeta))\\
&\tilde{A}_{1,j}(t,x)=\alpha_j(y;q(\tau, \zeta))\hspace{.3cm}j=1,2\\
&\tilde{A}_{1,j}(t,x)=\epsilon \tilde{a}_{{j}}(\tau,y,\zeta)\hspace{.3cm}j=0,3
\end{split}
\end{equation}
where $\tilde{a}_{{j}}$ for $j=0,3$ is as constructed before. Suppose that $T=(T_0,T_1,T_2)$ where $T_0$ is the Higgs section component and $T_1$ and $T_2$ correspond to the gauge field section components of $A_1$ and $A_2$. Let
\[
\tilde{A}_1=\big(\tilde{A}_{1,0},\tilde{A}_{1,1},\tilde{A}_{1,2},\tilde{A}_{1,3}\big)
\]
To calculate the error $S[\tilde{\Phi}_1,\tilde{A}_1]$, one can use the results for $S[\Phi_1,A_1]$ by assuming $\psi_1=0$.  According to \eqref{p1.eq2}, since $T$ does not depend on $\psi_1$, we have:
\begin{equation}
\label{p1.eq25}
S_{\varphi}[\tilde{\Phi}_1,\tilde{A}_1]=\epsilon^2 T_0+ O(\epsilon^4)
\end{equation}
Also, according to \eqref{p1.eq1.9}, we have:
\begin{equation}
\label{p1.eq26}
\begin{split}
S_{a_j}[\tilde{\Phi}_1,\tilde{A}_1]=O(\epsilon^3)\hspace{.4cm}j=0,3
\end{split}
\end{equation}
Suppose that the compnents of $T$ are denoted by $(T_0,T_1,T_2)$, where $T_0$ is the complex part and $T_1,T_2$ are the real ones. Using \eqref{p1.eq25} and \eqref{p1.eq26}, the coefficient of $\epsilon^2$ in 
\[
(S_{\varphi}[\tilde{\Phi}_1,\tilde{A}_1],i\tilde{\Phi}_1)+\partial_0S_{a_0}[\tilde{\Phi}_1,\tilde{A}_1]-\sum_{j=1}^3\partial_j S_{a_j}[\tilde{\Phi}_1,\tilde{A}_1]
\]
is 
\[
\big(T_{0},i\phi\big)-\sum_{j=1}^2\partial_j T_j
\]
But, according to lemma \ref{c1.l1},
\[
(S_{\varphi}[\tilde{\Phi}_1,\tilde{A}_1],i\tilde{\Phi}_1)+\partial_0 S_{a_0}[\tilde{\Phi}_1,\tilde{A}_1]-\sum_{j=1}^3\partial_j S_{a_j}[\tilde{\Phi}_1,\tilde{A}_1]=0
\]
Therefore $T$ satisfies the gauge orthogonality condition. 
\end{proof}
According to \eqref{p1.eq3} and \eqref{p1.eq4}, $T\in \mathcal{E}$. Therefore, by lemmas \ref{al2} and \ref{l40.4} in the appendix, we can find $\psi_1\in \mathcal{E}$ which satisfies the gauge orthogonality condition for every $\tau,\zeta$ and
solves the equation
\[
L\psi_1(.,\tau,\zeta)=-T(.,\tau,\zeta)
\]
for each $\tau$ and $\zeta$. Therefore, $\mathcal{L}\psi_1=-T$. According to \eqref{p1.eq2}, the introduced choices for $\psi_1$ and $\tilde{a}_{j}$ for $j=0,3$ ensure that 
\[
S_{\varphi}[\Phi_1,A_1]=O(\epsilon^4)\hspace{.3cm},\hspace{.3cm}S_{a_j}[\Phi_1,A_1]=O(\epsilon^4)\hspace{.4cm}j=1,2
\]
We already established that $S_{a_j}[\Phi_1,A_1]=O(\epsilon^3)$ for $j=0,3$. Noting that all of the components of the functions $f_2,f_3$ in $\eqref{p1.eq2}$ and $g_{j,2}$ and $g_{j,3}$ in \eqref{p1.eq1.9} belong to $\langle\mathcal{O}_1|\mathcal{F}_1\rangle$
where
\[
\mathcal{O}_1=\mathcal{R}\big[\{\partial_1,\partial_2,\{\tilde{\partial}_j\},\alpha_1,\alpha_2,\{\tilde{a}_j\},\phi,\bar{\phi},1,i\}\big]
\]
and
\[
\mathcal{F}_1=\mathcal{R}\big[\{D_1\phi,\overline{D_1\phi}, D_2\phi, \overline{D_2\phi}, |\phi|^2-1,\mathcal{F}_{12},\{\tilde{\mathcal{F}}_{\alpha\beta}\},\tilde{D}_j\phi,\overline{\tilde{D}_j\phi},\varphi_1,\overline{\varphi_1},a_{1,1},a_{1,2}\}\big]
\]
and
\[
\langle \mathcal{O}_1| \mathcal{F}_1 \rangle =\mathcal{R}\big[\{ab| a\in \mathcal{O}_1, b\in \mathcal{F}_1\}\big]
\]
where
\begin{equation}
\label{p1.eq26.2}
\tilde{\mathcal{F}}_{\alpha\beta}=
\begin{cases}
&\tilde{\partial}_{_{\alpha}}\tilde{a}_{_{\beta}}-\tilde{\partial}_{_{\beta}}\tilde{a}_{_{\alpha}}\hspace{1cm}\alpha,\beta\in\{0,3,\cdots,n\}\\
&\tilde{\partial}_{_{\alpha}}{a}_{_{\beta}}-\partial{_{\beta}}\tilde{a}_{_{\alpha}}\hspace{1cm}\alpha\notin\{1,2\}\hspace{.2cm},\hspace{.2cm}\beta\in\{1,2\}
\end{cases}
\end{equation}
Since $\mathcal{F}_1\subset \mathcal{E}$ and the functions $\alpha_1,\alpha_2,\tilde{a}_j,\phi$ and their multi-derivatives with respect to $y,\tau,\zeta$ are bounded, one can deduce that
\[
\langle\mathcal{O}_1|\mathcal{F}_1\rangle \subset \mathcal{E}
\]
Therefore
\[
f_2,f_3,g_{j,2},g_{j,3}\in \mathcal{E}
\]
Therefore, the statement holds for $m=1$.

Now, suppose that the statement holds for $m$ and
\begin{equation}
\begin{pmatrix}
S_{\varphi}\\
S_{a_j}
\end{pmatrix}_{j=1,2}[\Phi_m,A_m](t,x)=\sum_{i=m+1}^{3m}\epsilon^{2i}f_i(\tau,y,\zeta)
\end{equation}
and
\begin{equation}
\label{p1.eq26.2.new}
S_{a_j}[\Phi_m,A_m](t,x)=\sum_{i=m+1}^{3m}\epsilon^{2i-1}g_{j,i}(\tau,y,\zeta)
\end{equation}
for $j=0,3$. Let:
\begin{equation}
\label{p1.eq27}
\begin{split}
&\Phi_{m+1}(t,x)=\Phi_m(t,x)+\epsilon^{2m}\sum_{\mu=1}^{2N}c_{\mu}(\tau,\zeta)n_{\mu,\varphi}(y;\tau,\zeta)+\epsilon^{2m+2}{\varphi}_{m+1}(\tau,y,\zeta)\\
&A_{m+1,j}(t,x)=A_{m,j}(t,x)+\epsilon^{2m}\sum_{\mu=1}^{2N}c_{\mu}(\tau,\zeta)n_{\mu,j}(y;\tau,\zeta)+\epsilon^{2m+2}{a}_{j,m+1}(\tau,y,\zeta)\hspace{.4cm}j=1,2\\
&A_{m+1,j}(t,x)=A_{m,j}(t,x)+\epsilon^{2m+1}{a}_{j,m+1}(\tau,y,\zeta)\hspace{.3cm}j=0,3
\end{split}
\end{equation}
where the new terms are to be found. The procedure to construct the above object is to first let the functions $(c_{\mu})$ to be undetermined. Then, corresponding to $(c_{\mu})$, find the functions $a_{_{j,m+1}}$ for $j\in\{0,3\}$ such that the error term for $S_{a_j}[\Phi_{m+1},A_{m+1}]$ becomes of order $O(\epsilon^{2m+3})$. Then, to reduce the other error terms, an orthogonality condition to zero modes needs to be satisfied. This leads to a hyperbolic PDE for the functions $(c_{\mu})$ which is well-posed. After finding suitable $(c_{\mu})$, one can look back to the above process and find the other terms. 
\newline
Let
\begin{equation}
\label{p1.eq28.01.1}
E_{0}=\sum_{\mu}c_{\mu}n_{\mu,\varphi}
\end{equation}
and
\begin{equation}
\label{p1.eq28.01}
E_j=\sum_{\mu}c_{\mu}n_{\mu,j}\hspace{.3cm}j=1,2
\end{equation}
Then, we have:
\begin{equation}
\label{p1.eq28.1}
\begin{split}
S_{a_0}[\Phi_{m+1},A_{m+1}]&=-\epsilon^{2m+1}\big[\Delta_y a_{0,m+1}-|\phi|^2a_{0,m+1}+g_{_{0,m+1}}\big]\\
&-\epsilon^{2m+1}\Big[\big(i\phi,\tilde{D}_{0}E_0\big)+\big(iE_0,\tilde{D}_0\phi\big)-\tilde{\partial}_0\big(\sum_{\mu}c_{\mu}(\partial_1n_{\mu,1}+\partial_2n_{\mu,2})\big)\Big]\\
&+\epsilon^{2m+3}(...)+\cdots
\end{split}
\end{equation}
\begin{equation}
\label{p1.eq28.1.1}
\end{equation}
Using the fact that the zero modes $n_{\mu}$ satisfy the gauge orthogonality condition, we have:
\begin{equation}
\big(i\phi,\tilde{D}_{0}E_0\big)-\tilde{\partial}_0\big(\sum_{\mu}c_{\mu}(\partial_1n_{\mu,1}+\partial_2n_{\mu,2})\big)=(iE_0,\tilde{D}_0\phi)
\end{equation}
Therefore,
\begin{equation}
\label{p1.eq28.2}
\begin{split}
S_{a_0}[\Phi,A]=-\epsilon^{2m+1}\big[\Delta_y a_{0,m+1}-|\phi|^2a_{0,m+1}+2\sum_{\mu}c_{\mu}(in_{\mu,\varphi},\tilde{D}_0\phi)+g_{_{0,m+1}}\big]+\epsilon^{2m+3}(...)+
\end{split}
\end{equation}
Similarly,
\begin{equation}
\label{p1.eq28.3}
\begin{split}
S_{a_3}[\Phi,A]=-\epsilon^{2m+1}\big[\Delta_y a_{3,m+1}-|\phi|^2a_{3,m+1}+2\sum_{\mu}c_{\mu}(in_{\mu,\varphi},\tilde{D}_3\phi)+g_{_{3,m+1}}\big]+\epsilon^{2m+3}(...)+\cdots
\end{split}
\end{equation}
According to lemma \ref{al1} in the appendix, one can find 
\[
h_{j,\mu},k_{j}:\mathbb{R}^2\times [0,T_m)\times \mathbb{R}\to \mathbb{R}
\]
of class $\mathcal{E}$ for $j=0,3$, regardless of the choice for $(c_{\mu})_{_{\mu}}$, such that  
\begin{equation}
\label{p1.eq33}
a_{_{j,m+1}}=\sum_{\mu}c_{\mu}h_{j,\mu}+k_j\hspace{.4cm}j=0,3
\end{equation}
make the coefficient of $\epsilon^{2m+1}$ in the expressions \eqref{p1.eq28.2} and \eqref{p1.eq28.3} to vanish. (Add details here) Under this assumption for the functions $a_{_{j,m+1}}$ for $j\in\{0,3\}$, we have:
\begin{equation}
\label{p1.eq29}
\begin{pmatrix}
S_{\varphi}[\Phi_{m+1},A_{m+1}]\\
S_{a_{j}}[\Phi_{m+1},A_{m+1}]
\end{pmatrix}_{j=1,2}
=\epsilon^{2m+2}(\mathcal{L}\psi_{m+1}+S_{_{m+1}})+\epsilon^{2m+4}(...)+\cdots\hspace{.3cm}
\end{equation}
 and
\begin{equation}
\label{p1.eq29.1}
\psi_{m+1}=\big(\varphi_{_{m+1}},a_{_{1,m+1}},a_{_{2,m+1}}\big)
\end{equation}
and
\begin{equation}
\label{p1.eq30}
S_{_{m+1}}=(\tilde{\partial}_{0}^2-\tilde{\partial}_3^2)(\sum_{\mu}c_{\mu}n_{\mu})+p_{_{m+1}}(c_{\mu},\tilde{\partial}c_{\mu},\tau,\zeta)
\end{equation}
where $\tilde{\partial}$ denotes differentiation with respect to $\tau,\zeta$ and the function $p_{m+1}$ is a polynomial of degree at most $2$ with respect to $(c_{\mu})$ and $(\tilde{\partial}c_{\mu})$ with coefficients of class $\mathcal{E}$ and independent of $\psi_{m+1}$.
\begin{proposition}
\label{ao}
One can find smooth functions $(c_{\mu})$ with bounded multi-derivatives such that the vector $S_{_{m+1}}$ is orthogonal to the corresponding zero modes for every $\tau,\zeta$ with $\tau\le T_{m+1}$ for some $T_{m+1}>0$. 
\end{proposition}
\begin{proof}
The orthogonality conditions to the vectors $n_{\lambda}$ for $S_{_{m+1}}$ lead to the following equation:
\begin{equation}
\label{p1.eq34}
\tilde{\Box}(c_{\mu})_{\mu}=F((c_{\mu})_{\mu},(\tilde{\partial}c_{\mu})_{\mu},\tau,\zeta)
\end{equation}
where 
\[
\tilde{\Box}=\tilde{\partial}_0^2-\sum_{j=3}^n\tilde{\partial}_j^2
\] 
and for each $\tau,\zeta$, the function $F$ is linear or quadratic with respect to $(c_{\mu}),(\tilde{\partial}c_{\mu})$ with coefficients of class $\mathcal{E}$. This is a well-posed equation and by considering the zero initial condition for $(c_{\mu})_{\mu}$, one can find smooth solutions with bounded derivatives on the time interval $[0,T_{m+1}]$ for some $T_{m+1}>0$.
\end{proof}
Now, by choosing $(c_{\mu})_{\mu}$ as in the above proposition and the functions $a_{_{j,m+1}}$ for $j=0,3$ according to $\eqref{p1.eq33}$, the vector $S_{m+1}$ satisfies the orthogonality condition to the zero modes and one can use lemmas \ref{al2} and \ref{ln4} to find $\psi_{n+1}$ of class $\mathcal{E}$ such that
\[
L\psi_{_{m+1}}+S_{m+1}=0
\] 
\begin{proposition}
The vector $S_{_{m+1}}$ satisfies the gauge orthogonality condition.
\end{proposition}
\begin{proof}
Similar to the proof of proposition \ref{c2.p3}, we use lemma \ref{c1.l1}. Let
\begin{equation}
\label{p1.eq35}
\begin{split}
&\tilde{\Phi}_{m}(t,x)=\Phi_m(t,x)+\epsilon^{2m}\sum_{\mu=1}^{2N}c_{\mu}(\tau,\zeta)n_{\mu,\varphi}(y;\tau,\zeta)\\
&\tilde{A}_{m,j}(t,x)=A_{m,j}(t,x)+\epsilon^{2m}\sum_{\mu=1}^{2N}c_{\mu}(\tau,\zeta)n_{\mu,j}(y;\tau,\zeta)\hspace{.4cm}j=1,2\\
&\tilde{A}_{m,j}(t,x)=A_{m,j}(t,x)+\epsilon^{2m+1}{a}_{m+1,j}(y;\tau,\zeta)\hspace{.3cm}j=0,3,\cdots,n
\end{split}
\end{equation}
By using equation \eqref{p1.eq29} and comparing $S[\tilde{\Phi}_m,\tilde{A}_m]$ with $S[\Phi_{m+1},A_{m+1}]$ and noting that the correspondent of the $\psi_{m+1}$ term of $(\Phi_{m+1},A_{m+1})$ in $(\tilde{\Phi}_m,\tilde{A}_m)$ is zero and the vector $S_{m+1}$ is independent of $\psi_{m+1}$, we deduce that
\begin{equation}
\label{p1.eq35.1.1}
\begin{pmatrix}
S_{\varphi}[\tilde{\Phi}_m,\tilde{A}_m]\\
S_{a_j}[\tilde{\Phi}_m,\tilde{A}_m]
\end{pmatrix}_{_{j=1,2}}=\epsilon^{2m+2}S_{m+1}+\epsilon^{2m+4}(\cdots)+\cdots
\end{equation}
and according to \eqref{p1.eq28.2} and \eqref{p1.eq28.3}, we have:
\begin{equation}
\label{p1.eq36}
S_{a_j}[\tilde{\Phi}_m,\tilde{A}_m]=\epsilon^{2m+3}(\cdots)+\epsilon^{2m+5}(\cdots)+\cdots\hspace{.3cm}j=0,3
\end{equation}
Suppose that
\[
S_{m+1}=(S_{_{m+1,0}},S_{_{m+1,1}},S_{_{m+1,2}})
\]
where $S_{_{m+1,0}}$ denotes the complex part. According to \eqref{p1.eq35.1.1} and \eqref{p1.eq36}, the coefficient of $\epsilon^{2m+2}$ in 
\[
(i\tilde{\Phi}_m,S_{\varphi}[\tilde{\Phi}_m,\tilde{A}_m])+\partial_0 S_{a_0}[\tilde{\Phi}_m,\tilde{A}_m]-\sum_{j=1}^3\partial_j S_{a_j}[\tilde{\Phi}_m,\tilde{A}_m]
\]
is 
\[
(i\phi,S_{_{m+1,0}})-\sum_{j=1}^2\partial_jS_{_{m+1,j}}
\]
Therefore, according to lemma \ref{c1.l1}, the vector $S_{m+1}$ satisfies the gauge orthogonality condition.
\end{proof}
Now, according to the above lemma, one can use lemmas \ref{al2} and \ref{ln4} in the appendix to find $\psi_{m+1}$ of class
$\mathcal{E}$ which satisfies the gauge orthogonality condition for every $\tau,\zeta$ and
\[
L\psi_{_{m+1}}+S_{_{m+1}}=0
\]
Since $\psi_{m+1}$ is gauge orthogonal, then $\mathcal{L}\psi_{_{m+1}}=L\psi_{_{m+1}}$ and therefore
\[
\mathcal{L}\psi_{_
{m+1}}+S_{_{m+1}}=0
\]
and by equations \eqref{p1.eq28.2}, \eqref
{p1.eq28.3}, and \eqref{p1.eq29}, the expected order of error for $S[\Phi_{m+1},A_{m+1}]$ is obtained. The desired estimates for the error terms hold due to the induction hypothesis and the fact that 
\[
(\Phi_{m+1},A_{m+1})-(\Phi_{m},A_{m}\big)
\]
is of class $\mathcal{E}$. This finishes the proof of the theorem.
\end{proof}	
\section{Perturbation of the Approximate solution}
\label{per}
\subfile{Perturbation_part}
\section{Proof of the main theorem}
We use proposition \ref{2.3.p1} with $m=3$ iterations so that the error term $E$ of the approximate solution is of the form
\begin{equation}
\label{pmt.eq1}
E=(E_{\varphi},E_{0},E_1,E_2,E_3)=O(\epsilon^{8},\epsilon^{7},\epsilon^{8},\epsilon^8,\epsilon^{7})
\end{equation}
in a pointwise sense and since the error term is supported on an interval of length $\frac{C}{\epsilon}$, then we have the estimate
\begin{equation}
\label{pmt.eq2}
||E(.;t)||_{_{H^4}}+||E_t(.;t)||_{_{H^3}}\le C\epsilon^6
\end{equation}
for any $t$. Assume that $v=(\varphi,a)$ is the obtained approximate solution. Now, by proposition \ref{prop.fin.per}, we find $u$ with
\begin{equation}
\label{pmt.eq3}
||u||_{_{H^3}}+||u_t||_{_{H^2}}\le \epsilon^3
\end{equation}
on the interval $[0,\frac{\kappa}{\epsilon})$ such that $v+u$ solves the AHM equation. Property \ref{pmt.eq3} and \eqref{p1.eq0}, \eqref{prop.ans.eq1}, \eqref{prop.ans.eq2} in theorem \ref{pmt.eq1} imply that 
\begin{equation}
\label{pmt.eq4}
v(t,y,z)=\begin{pmatrix}
(\phi,\alpha)(y;q(\epsilon t, \epsilon z))
\\
0
\end{pmatrix}+
\begin{pmatrix}
\bar{\psi}
\\
\bar{a}_0,\bar{a}_3
\end{pmatrix}
\end{equation}
where
\begin{equation}
\label{pmt.eq5}
||\psi||_{_{H^3}}+||\psi_t||_{_{H^2}}\le C\epsilon^{\frac{3}{2}} 
\end{equation}
and 
\begin{equation}
\label{pmt.eq6}
\sum_{j=0,3}||\bar{a}_j||_{_{L^{\infty}}}\le C\epsilon
\end{equation}
and also
\begin{equation}
\label{pmt.eq7}
\sum_{j=0,3}||\nabla \bar{a}_j||_{_{H^2}}\le C \epsilon^{\frac{1}{2}}
\end{equation}
Properties \eqref{pmt.eq3}, \eqref{pmt.eq4}, \eqref{pmt.eq5}, \eqref{pmt.eq6} and \eqref{pmt.eq7} imply the desired estimates in theorem \ref{thm2.main} for a time interval of the form $[0,\frac{C}{\epsilon})$ where $C$ depends only in the wave map $q$.
\chapter{Appendix}
In this appendix, we are going to mention some analytic results and estimates which have been used in the chapters. 
\begin{lemma}
\label{al0}
Corresponding to the vortex centers $z_1,z_2,\cdots,z_N$ contained inside the ball $B(0,{R})\subset{\mathbb{R}}^2$, consider the function $\Theta:\mathbb{C}\to \mathbb{R}$ as
\begin{equation}
\label{al0.eq1}
\Theta(z)=2\sum_{i=1}^N \arg(z-z_{i})
\end{equation}
Suppose that $z=z^1+iz^2$ and 
\[
z_j=z_{j}^1+iz_{j}^2
\]
where $z_1,z_2,z_{j}^1,z_{j}^2\in \mathbb{R}$. Then, for every multi-index $r>0$, there exists $A>0$ such that
\[
|D^r\Theta(z)|\le A|z|^{-|r|+1}
\]
for every $z\in \mathbb{R}^2$ with $|z|>2R$, where $D^r$ denotes a combination of the differentials $\frac{\partial}{\partial z^i}$ and $\frac{\partial}{\partial z^{i}_j}$.
\end{lemma}
\subfile{Lemma_B}

\subfile{Lemma_A}

\subfile{Regularity_in_Lemma_A}

\subfile{References}
\end{document}

%% file: Perturbation_part.tex
Suppose that $v=\big(\varphi,a\big)$ is an approximate solution for the Abelian Higgs model constructed in the first part of the project. Suppose that it is of the form
\begin{equation}
\label{eqp1}
\begin{split}
&\varphi (t,y,z)=\phi (y;q(\epsilon t, \epsilon z))+\epsilon^2{\hat{\varphi}}\\
&a_j(t,y,z)=\alpha_j (y;q(\epsilon t, \epsilon z))+\epsilon^2 {\hat{a}}_j\hspace{1cm}j=1,2\\
&a_0(t,y,z)=\epsilon \hat{a}_0(t,y,z)\\
&a_3(t,y,z)=\epsilon \hat{a}_3(t,y,z)
\end{split}
\end{equation}
defined on 
\[
\Big[0,\frac{T_1}{\epsilon}\Big)
\]
for some number $T_1$. We look for an honest solution $(\varphi+\tilde{\varphi},a+\tilde{a})$ for the Abelian Higgs equations. Let
\begin{equation}
\label{eq2}
u=(\tilde{\varphi},\tilde{a})
\end{equation}
and 
\begin{equation}
\label{eq3}
\tilde{a}=\sum_{j=0}^3\tilde{a}_jdx^j
\end{equation}
and 
\begin{equation}
\label{eq4}
\psi=(\tilde{\varphi}, \tilde{a}_1dx^1+\tilde{a}_2dx^2)
\end{equation}
Also, suppose that the error terms of the approximate solutions are denoted by
\begin{equation}
\label{eq5}
E=\big(E_{_{\varphi}},E_0,E_1.E_2,E_3\big)
\end{equation}
where $E_{\varphi}$ correspond to the Higgs section and $E_j$ correspond to the gauge field $A_j$. If the number of iterations is large enough, we can ensure that 
\begin{equation}
\label{er.ans}
||E(t,.)||_{H^4}+||E_t(t,.)||_{H^3}\le C\epsilon ^n
\end{equation}
for some $C$. The number $n$ depends on the number of iterations in the ansatz which will be figured out later.  
\subsection{Gauge Choice}
To find $u$, we need to assume a choice for gauge. Also, we need to justify, why we can assume such a choice:
\begin{proposition}
\label{sq.p1}
Suppose that 
\[
\big(\tilde{\varphi},\tilde{a}\big)\in C^0\big([0,T^{\prime});H^3\big)\cap \big(C^1\big([0,T^{\prime});H^2\big)\cap \big(C^2\big([0,T^{\prime});H^1\big)
\]
is such that
\[
w=\big(\varphi,a\big)+\big(\tilde{\varphi},\tilde{a}\big)
\]
is a solution for the system of equations
\begin{equation}
\label{sg.eq4}
\begin{split}
&S_{_{\varphi}}[w]=iG(\varphi+\tilde{\varphi})\\
&S_{_{a_j}}[w]=\partial_j G\hspace{.5cm}j=0,1,2,3\\
\end{split}
\end{equation}
where
\[
G=-\partial_0\tilde{a}_0+\partial_1\tilde{a}_1+\partial_2\tilde{a}_2+\partial_3\tilde{a}_3-\big(i\varphi,\tilde{\varphi}\big)
\]
Suppose that
\begin{equation}
\label{sg.eq4.01}
\tilde{\varphi}(0)=\tilde{a}_1(0)=\tilde{a}_2(0)=\tilde{a}_3(0)=(\partial_t \tilde{\varphi})(0)=(\partial_t \tilde{a})(0)=0
\end{equation}
and
\begin{equation}
\label{sg.eq4.1}
\Delta \tilde{a}_0(0)-|\varphi|^2(0)\tilde{a}_0(0)+E_{0}(0)=0
\end{equation}
Then, $v$ solves the Abelian Higgs model.
\end{proposition}
\begin{proof}
We have:
\begin{equation}
\label{eq5}
\Big(S_{_{\varphi}}[w],i(\varphi+\tilde{\varphi})\Big)=-\partial_0 S_{a_0}[w]+\sum_{_{k=1}}^3\partial_k S_{_{a_k}}[w]
\end{equation}
Therefore, by \eqref{sg.eq4}, we have:
\begin{equation}
\label{eq6}
\big(\partial_{tt}-\Delta\big)G=|\varphi+\tilde{\varphi}|^2G
\end{equation}
According to the initial conditions, we have:
\[
G(0)=0
\]
Note that
\begin{equation}
    \label{eq7}
    \begin{split}
    S_{a_0}[w]&=E_0+\Delta\tilde{a}_0-\sum_{j=1}^3\partial_t\partial_j \tilde{a}_j+\Big(i\varphi,\partial_{t}\tilde{\varphi}\Big)\\
    &-\tilde{a}_0|\varphi|^2-2(\varphi,a_0\tilde{\varphi})+(i\tilde{\varphi},\partial_t \varphi)\\
    &+(i\tilde{\varphi},\partial_t \tilde{\varphi})-2(\varphi,\tilde{a}_0\tilde{\varphi})-a_0|\tilde{\varphi}|^2-\tilde{a}_0|\tilde{\varphi}|^2
    \end{split}
\end{equation}
Therefore, together with equation \eqref{sg.eq4.1} implies that
\[
S_{a_0}[w](0)=0
\]
and therefore, by equation \eqref{sg.eq4}
\[
\partial_tG(0)=0
\]
Therefore, \eqref{eq6} implies that $G=0$ for all times and $v$ solves the Abelian Higgs model.
\end{proof}
\subsection{Rewriting the equations}
According to proposition \ref{sq.p1}, it suffices to find $u=(\tilde{\varphi},\tilde{a})$ such that the following is satisfied:
\begin{equation}
\label{se.eq4}
\begin{split}
&S_{_{\varphi}}[(\varphi+\tilde{\varphi},a+\tilde{a})]=iG(\varphi+\tilde{\varphi})\\
&S_{_{a_j}}[\varphi+\tilde{\varphi},a+\tilde{a}]=\partial_j G\hspace{.5cm}j=0,1,2,3\\
\end{split}
\end{equation}
where
\begin{equation}
\label{se.eq4.1001}
G=-\partial_0 \tilde{a}_0+\sum_{j=1}^3\partial_j \tilde{a}_j -(i\varphi,\tilde{\varphi})
\end{equation}
where the initial conditions are:
\begin{equation}
\label{sg.eq4.02}
\tilde{\varphi}(0)=\tilde{a}_1(0)=\tilde{a}_2(0)=\tilde{a}_3(0)=(\partial_t \tilde{\varphi})(0)=(\partial_t \tilde{a})(0)=0
\end{equation}
and
\begin{equation}
\label{sg.eq4.1002}
\Delta\tilde{a}_0(0)-|\varphi|^2(0)\tilde{a}_0(0)+E_{0}(0)=0
\end{equation}
The existence of $\tilde{a}_0$ satisfying equation \eqref{sg.eq4.1002} is provided by lemma \ref{3d.l1} in the appendix.
\\
To write down equation \eqref{se.eq4}, first we study the linearization of the equations. To linearize the equations, we use the following calculations:
\begin{lemma}
\label{2.3.l1}
We have:
\begin{equation}
\label{2.3.l1.eq1}
(D_{_{A+\tilde{A}}})^2(\Phi+\tilde{\Phi})=D_A^2\Phi+D_A^{2}\tilde{\Phi}-2i\tilde{A}D_A\Phi-i(\nabla.\tilde{A})\Phi-D_A(i\tilde{A}\tilde{\Phi})-i\tilde{A}D_{A}\tilde{\Phi}-(\tilde{A})^2(\Phi+\tilde{\Phi})
\end{equation}
and
\begin{equation}
\label{2.3.l1.eq2}
(|\Phi+\tilde{\Phi}|^2-1)(\Phi+\tilde{\Phi})=(|\Phi|^2-1)\Phi+\Big(2(\Phi,\tilde{\Phi})\Phi+(|\Phi|^2-1)\tilde{\Phi}\Big)\\
+2(\Phi,\tilde{\Phi})\tilde{\Phi}+|\tilde{\Phi}|^2\tilde{\Phi}
\end{equation}
and
\begin{equation}
\label{2.3.l1.eq3}
\begin{split}
(i(\Phi+\tilde{\Phi}),D_{A+\tilde{A}}(\Phi+\tilde{\Phi}))&=(i\Phi,D_{A}\Phi)+(i\tilde{\Phi},D_A\Phi)+(i\Phi,D_{A}\tilde{\Phi})-|\Phi|^2\tilde{A}\\
&+(i\tilde{\phi},D_A\tilde{\phi})-2(\phi,\tilde{\phi})\tilde{A}-|\tilde{\phi}|^2\tilde{A}
\end{split}
\end{equation}
\end{lemma}
\begin{proof}
The proof is by straightforward calculations.
\begin{equation}
\label{p1.l1.eq1}
\begin{split}
(D_{A+\tilde{A}})^2(\Phi+\tilde{\Phi})&=(D_{A}-i\tilde{A})^2(\Phi+\tilde{\Phi})
\\
&=(D_A^2-i\tilde{A}D_{A}-iD_{A}.\tilde{A}-(\tilde{A})^2)(\Phi+\tilde{\Phi})
\\
&=D_A^2\Phi+D_A^2\tilde{\Phi}-i\tilde{A}D_{A}\Phi-i\tilde{A}D_A\tilde{\Phi}\\
&-iD_{A}(\tilde{A}\Phi)-iD_{A}(\tilde{A}\tilde{\Phi})-(\tilde{A})^2(\Phi+\tilde{\Phi})
\\
&=D_A^2\Phi+D_A^2\tilde{\Phi}-2i\tilde{A}D_{A}\Phi-i(\nabla. \tilde{A})\Phi
\\
&-iD_{A}(\tilde{A}\tilde{\Phi})-i\tilde{A}D_{A}\tilde{\Phi}-(\tilde{A})^2(\Phi+\tilde{\Phi})
\end{split}
\end{equation}
Identity \eqref{2.3.l1.eq2} is simple. To check \eqref{2.3.l1.eq3}, we use the fact that $(iv,iw)=(v,w)$ and proceed as follows:
\begin{equation}
\label{p1.l1.eq3.1}
\begin{split}
(i(\Phi+\tilde{\Phi}),D_{A+\tilde{A}}(\Phi+\tilde{\Phi}))&=(i(\Phi+\tilde{\Phi}),(D_A-i\tilde{A})(\Phi+\tilde{\Phi}))
\\&=(i(\Phi+\tilde{\Phi}), D_{A}\Phi+D_{A}\tilde{\Phi}-i\tilde{A}\Phi-i\tilde{A}\tilde{\Phi})
\\&=
(i\Phi,D_{A}\Phi)+(i\tilde{\Phi},D_A\Phi)+(i\Phi,D_{A}\tilde{\Phi})-|\Phi|^2\tilde{A}\\
&+(i\tilde{\Phi},D_{A}\tilde{\Phi})-2(\Phi,\tilde{\Phi})\tilde{A}-|\tilde{\Phi}|^2\tilde{A}
\end{split}
\end{equation} 
\end{proof}
According to lemma \ref{2.3.l1}, one can observe that
\begin{equation}
\label{2.3.eq1.9}
S(\Phi+\tilde{\Phi},A+\tilde{A})=S(\Phi,A)+\mathcal{H}[\Phi,A](\tilde{\Phi},\tilde{A})+N(\Phi,\tilde{\Phi},A,\tilde{A})
\end{equation}
where
\begin{equation}
\label{2.3.eq1.91}
\mathcal{H}[\Phi,A](\tilde{\Phi},\tilde{A})=\mathcal{G}[\Phi+\tilde{\Phi},A][\tilde{\Phi},\tilde{A}]
+\begin{pmatrix}
i\Phi\big(\nabla.\tilde{A}-\big(i\Phi,\tilde{\Phi}\big)\big)\\
\partial_i\big(\nabla.\tilde{A}-\big(i\Phi,\tilde{\Phi}\big)\big)
\end{pmatrix}
\end{equation}
where
\begin{equation}
\label{2.3.eq1.91.01}
\mathcal{G}[\Phi,A][\tilde{\Phi},\tilde{A}]=
\begin{pmatrix}
&\Box_{A}\tilde{\Phi}-2i(\tilde{A}_0D_0{\Phi}-\sum_{i=1}^3\tilde{A}_iD_i{\Phi})+\frac{1}{2}\big(3|\Phi|^2-1\big)\tilde{\Phi}\\
&\Box\tilde{A}_i+|\Phi|^2\tilde{A}_i-2\big(i\tilde{\Phi},D_i\Phi\big)
\end{pmatrix}
\end{equation}
where
\[
\Box_A=D_0D_0-\sum_{j=1}^3D_jD_j
\]
and $N$ consists of nonlinear terms with respect to $\tilde{\Phi}$ and $\tilde{A}$.
\\
To write equations \eqref{se.eq4}, we use equations \eqref{2.3.eq1.9}, \eqref{2.3.eq1.91} and \eqref{2.3.eq1.91.01}, and obtain that
\begin{equation}
\label{2.3.eq1.93}
\begin{split}
S_{\varphi}(\varphi+\tilde{\varphi},a+\tilde{a}))&=S_{\varphi}\big((\phi+\epsilon^2 \hat{\varphi})+\tilde{\varphi},a+\tilde{a}\big)
\\
&=S_{\varphi}(\varphi,a)+(\partial_t-i{a}_0)^2\tilde{\varphi}-(\partial_z-i{a}_3)^2\tilde{\varphi}\\
&-\sum_{j=1}^2(\partial_j-i\alpha_j-i\epsilon^2\hat{a}_j)^2\tilde{\varphi}
\\
&-2i\tilde{a}_0(\partial_t-i{a}_0)(\phi+\epsilon^2 \hat{\varphi})+2i\tilde{a}_3(\partial_z-i{a}_3)(\phi+\epsilon^2 \hat{\varphi})
\\
&+2i\sum_{j=1}^2(\tilde{a}_j(\partial_j-i\alpha_j-i\epsilon^2\hat{a}_j)(\phi+\epsilon^2 \hat{\varphi}))
\\
&+\frac{1}{2}(3|(\phi+\epsilon^2 \hat{\varphi})|^2-1)\tilde{\varphi}+{N}_0+i\varphi G
\\
&=\partial_t^2 \tilde{\varphi}-2ia_0\partial_t\tilde{\varphi}-\partial_z^2 \tilde{\varphi}+2ia_3\partial_z\tilde{\varphi}
\\
&+L_{\varphi}[\phi,\alpha][\psi]+E_0+N_0
\\&-2i\tilde{a}_0D_0\varphi+2i\tilde{a}_3D_3\varphi
\\&+R_0u+i\varphi G
\end{split}
\end{equation}
where
\begin{equation}
\label{new.2.3.neq1}
\begin{split}
R_0u&=-i(\partial_t a_0)\tilde{\varphi}-a_0^2\tilde{\varphi}+i(\partial_z a_3)
\tilde{\varphi}+a_3^2\tilde{\varphi}\\
&+\sum_{j=1}^2\Big[2i\epsilon^2\hat{a}_j\partial_j\tilde{\varphi}+i\epsilon^2(\partial_j\hat{a}_j)\tilde{\varphi}-2\epsilon^2\alpha_j\hat{a}_j\tilde{\varphi}+\epsilon^4\hat{a}_j^2\tilde{\varphi}\Big]+3\epsilon^2(\phi,\hat{\varphi})\tilde{\varphi}+\frac{3}{2}\epsilon^4|\hat{\varphi}|^2\tilde{\varphi}\\
&+2i\sum_{j=1}^2\epsilon^2\tilde{a}_j\Big(\partial_j \hat{\varphi}-i\hat{a}_j\phi-i\alpha_j\hat{\varphi}-i\epsilon^2\hat{a}_j\hat{\varphi}\Big)
\end{split}
\end{equation}
and $N_0$ consists of nonlinear terms and $E_0$ corresponds to the error of the approximate solution.
\\
Now, for $j=1,2$, we have
\begin{equation}
\label{new.eq2.3.neq2}
\begin{split}
S_{a_j}[\varphi+\tilde{\varphi},a+\tilde{a}]&=S_{a_j}[\varphi,a]+\partial_jG+\big(\partial_t^2-\partial_{z}^2-\Delta+|\phi+\epsilon^2\hat{\varphi}|^2\big)\tilde{a}_j\\
&-2(i\tilde{\varphi},(\partial_j-i\alpha_j-i\epsilon^2\hat{a}_j)(\phi+\epsilon^2\hat{\varphi}))+N_j\\
&=\big(\partial_t^2-\partial_{z}^2-\Delta+|\phi|^2\big)\tilde{a}_j-2(i\tilde{\varphi},D_j\phi)\\
&+S_ju+N_j+E_j+\partial_j G\\
&=(\partial_t^2-\partial_z^2)\tilde{a}_j-L_{j}[\phi,\alpha](\psi)+S_ju+N_j+E_j+\partial_j G\
\end{split}
\end{equation}
where
\begin{equation}
\label{new.eq7.5}
S_ju=-2\epsilon^2\Big(i\tilde{\varphi},\partial_j \hat{\varphi}-i\hat{a}_j\phi-i\alpha_j\hat{\varphi}-\epsilon^2\hat{a}_j\hat{\varphi}\Big)+\epsilon^2\tilde{a}_j\Big(2(\phi,\hat{\varphi})+\epsilon^2|
\hat{\varphi}|^2\Big)
\end{equation}
and $N_j,E_j$ correspond to the nonlinear terms and error of the approximate solution, respectively.
\\
For $j=0,3$, we have:
\begin{equation}
\label{new.eq2.3.neq2.1}
\begin{split}
S_{a_j}[\varphi+\tilde{\varphi},a+\tilde{a}]&=S_{a_j}[\varphi,a]+\partial_jG+\big(\partial_t^2-\partial_{z}^2-\Delta+|\phi+\epsilon^2\hat{\varphi}|^2\big)\tilde{a}_j\\
&-2(i\tilde{\varphi},(\partial_j-i{a}_j)(\phi+\epsilon^2\hat{\varphi}))+N_j\\
&=\big(\partial_t^2-\partial_{z}^2-\Delta+|\phi|^2\big)\tilde{a}_j-2(i\tilde{\varphi},D_j\phi)\\
&+S_ju+N_j+E_j+\partial_j G\\
\end{split}
\end{equation}
where
\begin{equation}
\label{neq.eq1001}
S_ju=2\epsilon^2 (\phi,\hat{\varphi})\tilde{a}_j-2\epsilon^2(i\tilde{\varphi},D_j\hat{\varphi})
\end{equation}
for $j=0,3$ and $N_j,E_j$ are the nonlinearity and the error of the approximate solution. Therefore, one can write down the equations \eqref{se.eq4} for $u=(\tilde{\varphi},\tilde{a})$ in the following form:
\begin{equation}
\label{eq7.1}
u_{tt}+Mu+Pu+N+E=0
\end{equation}
where
\begin{equation}
\label{eq1.main1}
Mu=
\begin{pmatrix}
-\partial_z^2\psi+L[\phi,\alpha]\psi\\
\\
-\Delta \tilde{a}_0+|\phi|^2\tilde{a}_0\\
\\
-\Delta \tilde{a}_3+|\phi|^2\tilde{a}_3
\end{pmatrix}
\end{equation}
and
\[
Pu=\begin{pmatrix}
\begin{pmatrix}
Ru\\
0
\end{pmatrix}\\
\\
2(i\tilde{\varphi},D_0\phi)\\
\\
2(i\tilde{\varphi},D_3\phi)
\end{pmatrix}+
\begin{pmatrix}
\begin{pmatrix}
R_0u
\\
Su
\end{pmatrix}
\\
\\
S_0u
\\
\\
S_3u
\end{pmatrix}
\]
where
\begin{equation}
\label{eq7.3}
\begin{split}
Ru&=-2ia_0\partial_t \tilde{\varphi}+2ia_3\partial_z \tilde{\varphi}-2i\tilde{a}_0D_0\varphi+2i\tilde{a}_3D_3\varphi
\end{split}
\end{equation}
where 
\begin{equation}
\label{new.eq7.4}
Su=(S_1u,S_2u)
\end{equation}
where $R_0u$ and $S_ju$ for $j=0,1,2,3$ are as before and $N,E$. consist of nonlinear terms and the error of the approximate solution respectively. The quantity $Ru$ is of the form $\epsilon F(u,Du)$ for some linear function $F$ and the
quantities $R_0u, Su, S_0u$ and $S_3u$ are of the form $\epsilon^2 G(u,Du)$ for some linear function $G$
\subsection{local existence theorems and a priori estimates}
\begin{theorem}
\label{thm2.loc}
Consider the approximate solution $v=(\tilde{\varphi},\tilde{a})$ described in the beginning of section \ref{per} with error $E$ satisfying conditions \eqref{er.ans}. Suppose that
\[
0\le s\le \frac{T_1}{\epsilon}
\] 
We consider equation \eqref{eq7.1} for the approximate solution . Consider the starting point to be at $t=s$. For any $b>0$, there exists $\delta>0$ such that if $\epsilon<\delta$ and the initial data for equations for $u$ satisfies
\[
(u(s,.),\partial_t u(s,.))\in H^3(\mathbb{R}^3)\oplus H^2(\mathbb{R}^3)
\]
and
\begin{equation}
\label{b.cs1}
a=||u(s,.)||_{H^3}+||\partial_t u(s,.)||_{_{H^2}}\le b
\end{equation}
there exists $T_2$ depending only on the function
\[
\tilde{v}(t,x,z)=v(\epsilon t,x,\epsilon z)
\]
for the approximate solution $v$ and in particular independent of $\epsilon$, and a solution
\[
u\in E_1=C^0([s,s+T_2];H^3)\cap C^1([s,s+T_2];H^2)\cap C^2([s,s+T_2];H^1)
\]
to equation \eqref{eq7.1}. Furthermore, it satisfies the estimate
\[
||u(t,.)||_{H^3}+||\partial_t u||_{H^2}\le Ca +{\epsilon}^n
\]
for some $C>0$.
\end{theorem}
\begin{proof}
Suppose that $\epsilon<\frac{1}{2}$.
Let
\[
w=\begin{pmatrix}
u\\
\partial_t u\\
\partial_{x_1}u\\
\partial_{x_2}u\\
\partial_{x_3}u
\end{pmatrix}
\]
Then, the equations for $w$ can be written in the form 
\begin{equation}
\label{thm1eq1}
\partial_{t}w+\sum_{j=1}^3A_j\partial_{x_j}w+F(t,x,w)=f(t,x)
\end{equation}
for some symmetric matrices $B_j$ with constant entries so that the operator
\[
K=\frac{\partial}{\partial_t}+\sum_{j=1}^3B_j\frac{\partial}{\partial x_j}
\]
is symmetric hyperbolic. Also $F$ involves the first order and nonlinear terms of equation \eqref{eq7.1} and $f$ is associated with the error $E$ of the approximate solution and by the property \eqref{er.ans} of error, we have
\begin{equation}
\label{f.error}
f(.)\in L^1_{_{loc}}\Big{[}\big{[}0,\frac{T_1}{\epsilon});H^2(\mathbb{R}^3)\Big{]}
\end{equation}
for some $C>0$. Now, we use the so called Friedrichs theorem from page 36, chapter 2 of \cite{Rauch}.
\begin{theorem}[Friedrichs]
\label{p1.thm1}
Suppose that $L$ is a symmetric hyperbolic operator and $g\in H^r(\mathbb{R}^d)$ and $f\in L^1_{loc}(\mathbb{R};H^r(\mathbb{R}^d))$ for some $r>0$. Then, there is one and only one solution $u\in C(\mathbb{R};H^r(\mathbb{R}^d))$ to the initial value problem
\begin{equation}
\label{eq1.p1.thm1}
\begin{split}
&Lu=f
\\
&u(0)=g
\end{split}
\end{equation}
In addition, there is a constant $C=C(L,r)$ independent of $f,g$ so that for all $f,g>0$, there holds:
\begin{equation}
\label{eq2.p1.thm1}
||u(t)||_{_{H^r(\mathbb{R}^d)}}\le Ce^{Ct}||u(0)||_{_{H^r(\mathbb{R}^d)}}+\int_{0}^t Ce^{C(t-\sigma)}||f(\sigma)||_{_{H^r(\mathbb{R}^d)}}d\sigma
\end{equation}
\end{theorem}
Let $r=2$ and $L=K$ in the setting of the above theorem. By using the Sobolev embedding and the fact that the coefficients of $F$ have bounded derivatives of any order and by \eqref{f.error}, we can use the above theorem to find a sequence 
\[
w_{m}\in C\big[[s,\frac{T_1}{\epsilon});H^2\big]
\]
such that
\[
\partial_{t}w_{m+1}+\sum_{j=1}^3B_j\partial_{x_j}w_{m+1}+F(t,x,w_m)=f(t,x)
\]
and $w_0=0$. Suppose that $C_1$ is the constant provided by the Friedrich's theorem. Suppose that
\[
R=2aC_1+\epsilon^{n}
\]
There exists a polynomial $p$ of degree $3$ such that if 
\[
||w(t,.)||_{H^2}\le R
\]
then
\begin{equation}
\label{eq2.p1.thm.n0}
||F(w(t,.))||_{H^2}\le p(R)
\end{equation}
This is because the nonlinearity in the equation \eqref{eq7.1} is of order 3. By assuming that $w_0=0$, we will find $T_2$ such that for any $t\in [s,s+T_2]$, we have
\begin{equation}
\label{eq2.p1.thm1.n1}
||w_m(t,.)||_{H^2}\le R
\end{equation}
Note that since $R=2aC_1+\epsilon^n$ and $a\le b$, there exists $T_2$ such that if $(t-s)\le T_2$, then
\begin{equation}
\label{eq2.p1.thm1.n3}
C_1 e^{C_1(t-s)}a +\frac{t-s}{C_1}(e^{C_1(t-s)}-1)(p(R)+C^{\prime}\epsilon^n) \le R
\end{equation}
Suppose that for some $m$, we know that \eqref{eq2.p1.thm1.n1} holds for any $t\in (s,s+T_2)$. By equation \eqref{eq2.p1.thm1} and using \eqref{eq2.p1.thm.n0}, we have: 
\begin{equation}
\label{eq2.p1.thm1.n2}
\begin{split}
||w_{m+1}(t,.)||_{H^2}\le &C_1 \Big(e^{C_1(t-s)}a+\int_{s}^te^{C_1(t-\sigma)}(P(R)+C\epsilon^n)d\sigma\Big)
\\
& \le C_1 e^{C_1(t-s)}a +\frac{t-s}{C_1}(e^{C_1(t-s)}-1)(p(R)+C^{\prime}\epsilon^n)
\end{split}
\end{equation}
\eqref{eq2.p1.thm1.n2} and \eqref{eq2.p1.thm1.n3} imply that if $(t-s)\le T_2$, then we have
\begin{equation}
\label{eq2.p1.thm1.n4}
||w_{m+1}(t,.)||_{H^2}\le R
\end{equation}
and by induction, the above bound holds for any $m$.
Now, we claim that the sequence $\{w_m\}_{m=1}^{\infty}$ converges in $C\big([s,s+T_2);H^2\big)$. Note that
\begin{equation}
\label{eq2.p1.thm1.n5}
K(w_{m+1}-w_{m})+F(t,x,w_m)-F(t,x,w_{m-1})=0
\end{equation}
But, since $||w_m(t,.)||_{H^2}\le R$, there exists a number $\lambda>0$ such that for any $m$,
\begin{equation}
\label{eq2.p1.thm1.n6}
||F(t,.,w_m)-F(t,.,w_{m-1})||_{H^2} \le \lambda ||w_m(t,.)-w_{m-1}(t,.)||_{H^2}
\end{equation}
Therefore, by the Friedrich's theorem, we have:
\begin{equation}
\label{eq2.p1.thm1.n7}
||w_{m+1}(t,.)-w_{m}(t,.)||_{H^2}\le C_1\int_{s}^te^{C_1(t-\sigma)}||w_{m}(t,.)-w_{m-1}(t,.)||_{H^2}d\sigma
\end{equation}
Let
\begin{equation}
\label{e2.p1.thm1.n7.1}
\begin{split}
&M_1=\sup_{[s,s+T_1]}||w_1(t,.)-w_{0}(t,.)||_{H^2}
\\
&M_2=C_1\lambda e^{C_1T_2}
\end{split}
\end{equation}
By induction on $m$, we have:
\begin{equation}
\label{e2.p1.thm1.n8}
||w_{m+1}(t,.)-w_{m}(t,.)||_{H^2}\le M_1\frac{(M_2t)^{m-1}}{(m-1)!}
\end{equation}
Therefore, the sequence $\{w_m\}_{m=0}^{\infty}$ is Cauchy in $C\big([s;s+T_2);H^2\big)$ and the limit satisfies the equation.
\end{proof}
\subsection{Some quantities and estimates}
\label{soq.bo}
In this section, we are going to introduce some quantities and estimate their evolution over the course of time. In subsections \ref{sub.so1} and \ref{sec.bo}, we will explain how these estimates allow us to do a bootstrap argument.
\subsubsection{Some quantities}
\label{bs.eq1}
The new quantities that we want to consider are as follows:
\begin{itemize}
\item 
Let:
\begin{equation}
S=Mu
\end{equation}
and
\begin{equation}
\label{eq1.defq1}
Q_1(t)=\int_{\mathbb{R}^3}|u_t|^2+(Mu,u)
\end{equation}
and
\begin{equation}
\label{eq1.defq2}
Q_2(t)=\int_{_{\mathbb{R}^3}}|S_t|^2+(MS,S)
\end{equation}
\item 
For each $(t,z)$, suppose that
\begin{equation}
\label{dec.in.eq1}
\psi(t,y,z)=\psi_1(t,y,z)+\sum_{\mu}c_{\mu}(t,z)n_{\mu}(y;t,z)
\end{equation}
where 
\begin{equation}
\label{dec.in.eq2}
\psi_1(t,.,z)\perp T_{(q(\epsilon t, \epsilon z))} M_N
\end{equation}
for any $(t,z)$ where $q$ is the wave map in theorem \ref{thm2.main}.
\item 
Let:
\begin{equation}
\label{eq5.5}
\tilde{a}_j=f_j+\partial_0 \chi_j
\end{equation}
for $j=0,3$, where $\chi_j$ solves the following equation:
\begin{equation}
\label{eq5.6}
\big(\Delta-|\phi|^2\big)\chi_j=\partial_0 \tilde{a}_j
\end{equation}
The existence of such $\chi_j\in H^3(\mathbb{R}^3)$ is guarantied by lemma \ref{3d.l1} and remark \ref{reg.3d} in the appendix.
\end{itemize}
\subsubsection{Estimates for the above quantities}
In this section, we will describe how the quantities $Q_1$, $Q_2$, $||c_{\mu}||_{_{H^3}}$ and $||\partial_t c_{\mu}||_{_{H^2}}$ change over time, approximately.
\begin{itemize}
\item 
\begin{proposition}
\label{eq.prop1}
There holds:
\begin{equation}
\label{eq.prop1.eqb5}
\begin{split}
(Q_1+Q_2)^{\prime}(t)\le C \big(\epsilon||u||_{H^3}+||N||_{H^2}+||E||_{H^2}+\epsilon ||u_t||_{H^2})\big(\epsilon ||u||_{H^3}+||u_t||_{H^2}\big)+\epsilon ||u||_{H^3}^2
\end{split}
\end{equation}
\end{proposition}
\begin{proof}
\begin{equation}
\label{eqb2}
Q_1^{\prime}=2\int_{\mathbb{R}^3}\big((u_{tt},u_t)+(Mu,u_t)\big)+(M_tu,u)
\end{equation}
So
\begin{equation}
\label{eqb2.1}
Q_1^{\prime}(t)=2(Pu+N+E,u_t)+(M_tu,u)
\end{equation}
and
\begin{equation}
\label{eqb3}
Q_2^{\prime}=2\Big(\int_{\mathbb{R}^3}(S_{tt},S_t)+(MS,S_t)\Big)+(M_tS,S)
\end{equation}
so
\begin{equation}
\label{eqb4}
Q_2^{\prime}(t)=2\big(MPu+MN+ME+2M_tu_t+M_{tt}u, (Mu)_t\big)+(M_tMu,Mu)
\end{equation}
So
\begin{equation}
\label{eqb5}
\begin{split}
(Q_1+Q_2)^{\prime}(t)\le &C ||Pu||_{L^2}|| ||u_t||_{L^2}+||N+E||_{L^2}||u_t||_{L^2}+C\epsilon ||u||_{H^1}||u||_{L^2}\\
+&C\big(||Pu+N+E||_{H^2}+\epsilon||u_t||_{H^1}+{\epsilon}^2||u||_{H^1}\big)\big(\epsilon ||u||_{H^1}+||u_t||_{H^2}\big)+\epsilon||u||_{H^3}^2
\\ 
\le & C \big(\epsilon||u||_{H^3}+||N||_{H^2}+||E||_{H^2}+\epsilon ||u_t||_{H^2})\big(\epsilon ||u||_{H^3}+||u_t||_{H^2}\big)+\epsilon ||u||_{H^3}^2
\end{split}
\end{equation}
\end{proof}
\item 
\begin{proposition}
\label{bs.cf.prop}
For $j=0,3$, we have:
\begin{equation}
\label{eq5.601}
-\Delta f_j+|\phi|^2f_j= (\partial_0|\phi|^2)\chi_j+2(i\tilde{\varphi},D_j\phi)+N_j+E_j
\end{equation}
for $j=0,3$, where $N_j$ and $E_j$ are the corresponding components of $N$ and $E$. 
\end{proposition}
\begin{proof}
According to the gauge field section of equation \eqref{eq7.1} and equations \eqref{eq5.5}, \eqref{eq5.6}, we have:
\begin{equation}
\label{eq5.61}
\begin{split}
\big(-\Delta+|\phi|^2\big)f_j&=\big(-\Delta+|\phi|^2\big)(\tilde{a}_j-\partial_0 \chi_j)\\
&
=\big(-\Delta+|\phi|^2\big)\tilde{a}_j
\\
&-\big(-\Delta+|\phi|^2\big)(\partial_0\chi_j)
\\
&=-\partial_t^2\tilde{a}_j+2(i\tilde{\varphi},D_j\phi)+N_j+E_j
\\
&-\big(-\Delta+|\phi|^2\big)(\partial_0\chi_j)
\\
&=-\partial_t^2\tilde{a}_j+2(i\tilde{\varphi},D_j\phi)+N_j+E_j\\
&-\partial_0\Big(\big(-\Delta+|\phi|^2\big)(\chi_j)\Big)+(\partial_0|\phi|^2)(\chi_j)
\\
&=-\partial_t^2\tilde{a}_j+2(i\tilde{\varphi},D_j\phi)+N_j+E_j+\partial_0^2\tilde{a}_j+(\partial_0|\phi|^2)\chi_j
\\&=(\partial_0|\phi|^2)(\chi_j)+2(i\tilde{\varphi},D_j\phi)+N_j+E_j
\end{split}
\end{equation}
\end{proof}
\item 
\begin{proposition}
\label{int.prop1}
There holds:
\begin{equation}
\label{eqb13.1.1}
\partial_{tt}c_{\mu}-\partial_{zz}c_{\mu}=\partial_0h_{0}+\partial_{3}h_3+h
\end{equation}
where
\begin{equation}
\label{eqb11}
h_0=\big(2ia_0\tilde{\varphi}+2i\chi_0D_0\varphi-2i\chi_3D_3\varphi,n_{_{\mu,\varphi}}\big)+2\big(\psi,(n_{_{\mu}})_{_t}\big)
\end{equation}
and
\begin{equation}
\label{eqb12}
h_3=-\big(2ia_3\tilde{\varphi},n_{_{\mu,\varphi}}\big)-2\big(\psi,(n_{\mu})_{_z}\big)
\end{equation}
and
\begin{equation}
\label{eqb13}
\begin{split}
h=&-(R_0 u,n_{\mu,\varphi})-\sum_{j=1}^2(S_ju, n_{\mu,a_j})\\
&-(N,n_{\mu})-(E,n_{\mu})-(\psi,(n_{\mu})_{tt})+(\psi,(n_{\mu})_{zz})
\\
&+\big(2if_0D_0\varphi-2if_3D_3\varphi,n_{_{\mu,\varphi}}\big)\\
&-\big(2i(\partial_0 a_0)\tilde{\varphi},n_{_{\mu,\varphi}}\big)
-\big(2ia_0\tilde{\varphi},\partial_0 n_{_{\mu,\varphi}}\big)
\\
&+\big(2i(\partial_3 a_3)\tilde{\varphi},n_{_{\mu,\varphi}}\big)+\big(2ia_3\tilde{\varphi},\partial_3n_{_{\mu,\varphi}}\big)
\\
&-\big(2i\chi_0(\partial_0D_0{\varphi}),n_{\mu,\varphi}\big)-\big(2i\chi_0D_0\varphi,\partial_0n_{\mu,\varphi}\big)
\\
&+\big(2i\chi_3\partial_0D_3{\varphi},n_{\mu,\varphi}\big)+\big(2i\chi_3D_3\varphi,\partial_0n_{\mu,\varphi}\big)
\end{split}
\end{equation}
and therefore for any $t,t^{\prime}$, we have
\begin{equation}
\label{eqb13.1.2}
\begin{split}
c_{_{\mu}}(t+t^{\prime},z) &=\frac{1}{2}\Big[c_{\mu}(t^{\prime},z-t)+c_{\mu}(t^{\prime},z+t)\Big]+\frac{1}{2}\big[\int_{z-t}^{z+t}\big(\partial_0c_{\mu}(t^{\prime},s)ds\big)\big]\\
&+\frac{1}{2}\int_{0}^{t}\int_{z-(t-\tau)}^{z+t-\tau}h(t^{\prime}+\tau,s)ds d\tau\\&
+\frac{1}{2}\int_{0}^t \Big[h_3(t^{\prime}+s,z+t-s)-h_3(t^{\prime}+s, z-(t-s))\Big]ds
\\&
-\frac{1}{2}\int_{z-t}^{z+t}h_0(t^{\prime},s)ds+\frac{1}{2}\int_{0}^t h_0(t^{\prime}+s,z-t+s) ds +\frac{1}{2}\int_{0}^th_0(t^{\prime}+s,z+t-s)ds
\end{split}
\end{equation}
\end{proposition}
\begin{proof}
We Take the inner product of the equations \eqref{eq7.1} with the zero modes $n_{\mu}$'s:
\begin{equation}
\label{eqb6}
\big(\psi_{tt}-\psi_{zz},n_{\mu}\big)+(L\psi,n_{\mu})=-(Hu,n_{\mu})-(N,n_{\mu})-(E,n_{\mu})
\end{equation}
where 
\begin{equation}
\label{eqb6.1001}
Hu=\begin{pmatrix}
Ru
\\
0
\end{pmatrix}+
\begin{pmatrix}
R_0u
\\
Su
\end{pmatrix}
\end{equation}
. We have:
\begin{equation}
\label{eqb6.1}
(L\psi,n_{\mu})=0
\end{equation} 
Let
\[
c_{\mu}=(\psi,n_{\mu})
\]
By \eqref{eqb6} and \eqref{eqb6.1}, we have:
\begin{equation}
\label{eq6.2}
\begin{split}
(c_{\mu})_{tt}-(c_{\mu})_{zz}&=2(\psi_t, (n_{\mu})_t)-2(\psi_z, (n_{\mu})_z)-(Hu,n_{\mu})-(N,n_{\mu})-(E,n_{\mu})\\
&+(\psi,(n_{\mu})_{tt}-(n_{\mu})_{zz})
\\
&=2\partial_t(\psi,(n_{\mu})_t)-2\partial_z (\psi,(n_{\mu})_{z})-(Hu,n_{\mu})
\\&-(N,n_{\mu})-(E,n_{\mu})-(\psi,(n_{\mu})_{tt})+(\psi,(n_{\mu})_{zz})
\end{split}
\end{equation}
Now, we want to write the term $(Hu,n_{\mu})$ in another form. 
\begin{equation}
\label{eq6.21}
\begin{split}
(Hu,n_{\mu})&=(Ru,n_{\mu,\varphi})+\big((R_0u,Su),n_{\mu}\big)\\
&=\big(-2ia_0\partial_t \tilde{\varphi}+2ia_3\partial_z \tilde{\varphi}-2i\tilde{a}_0D_0\varphi+2i\tilde{a}_3D_3\varphi,n_{\mu,\varphi}\big)
\\&+\big((R_0u,Su),n_{\mu}\big)
\end{split}
\end{equation}
Now, by equation \eqref{eq6.21}, we have:
\begin{equation}
\label{eqb8.1}
\begin{split}
(Hu,n_{\mu})&=\big(-2ia_0\partial_t \tilde{\varphi}+2ia_3\partial_z \tilde{\varphi}-2i\tilde{a}_0D_0\varphi+2i\tilde{a}_3D_3\varphi,n_{\mu,\varphi}\big)
\\
&+(R_0 u,n_{\mu,\varphi})+\sum_{j=1}^2(S_ju, n_{\mu,a_j})
\end{split}
\end{equation}
Therefore, according to \eqref{eq5.5}, we have:
\begin{equation}
\label{eqb9}
\begin{split}
(Hu,n_{\mu})&=-\partial_0\big(2ia_0\tilde{\varphi},n_{_{\mu,\varphi}})
+\big(2i(\partial_0 a_0)\tilde{\varphi},n_{_{\mu,\varphi}}\big)
+\big(2ia_0\tilde{\varphi},\partial_0 n_{_{\mu,\varphi}}\big)
\\&+\partial_3\big(2ia_3\tilde{\varphi},n_{_{\mu,\varphi}})-\big(2i(\partial_3 a_3)\tilde{\varphi},n_{_{\mu,\varphi}}\big)-\big(2ia_3\tilde{\varphi},\partial_3n_{_{\mu,\varphi}}\big)\\&
-\big(2if_0D_0\varphi-2if_3D_3\varphi,n_{_{\mu,\varphi}}\big)\\
&-\partial_{0}\big(2i\chi_0D_{0}\varphi,n_{\mu,\varphi}\big)+\big(2i\chi_0(\partial_0D_0{\varphi}),n_{\mu,\varphi}\big)+\big(2i\chi_0D_0\varphi,\partial_0n_{\mu,\varphi}\big)\\
&+\partial_{0}\big(2i\chi_3D_{3}\varphi,n_{\mu,\varphi}\big)-\big(2i\chi_3\partial_0D_3{\varphi},n_{\mu,\varphi}\big)-\big(2i\chi_3D_3\varphi,\partial_0n_{\mu,\varphi}\big)\\
&+\big(R_0u,n_{_{\mu,\varphi}}\big)+\sum_{j=1}^2(S_ju, n_{\mu,a_j})
\end{split}
\end{equation}
Combining equations \eqref{eq6.2} and \eqref{eqb9}, we have:
\begin{equation}
\label{eqb10}
\partial_{tt}c_{\mu}-\partial_{zz}c_{\mu}=\partial_0h_{0}+\partial_{3}h_3+h
\end{equation}
To solve this equation, we use the explicit formulas about the linear wave equations. The crucial point is that the term $\partial_0h_0$ in \eqref{eqb10} can be handled differently and the time derivative can be dropped in the solution, as in the following lemma. 
\begin{lemma}
\label{ldal}
The solution to the equation 
\begin{equation}
\label{ldal.eq1}
\partial_t^2f-\partial_z^2f=\partial_t w
\end{equation}
with the initial data $f(t,0)=\partial_t f(t,0)=0$
can be written as
\begin{equation}
\label{ldal.eq1.1}
f(t,z)=-\frac{1}{2}\int_{z-t}^{z+t}w(s,0)ds+\int_{0}^t w(s,z-t+s) ds +\int_{0}^tw(s,z+t-s)ds
\end{equation}
\end{lemma}
\begin{proof}
We have:
\begin{equation}
\label{ldal.p.eq1}
(\partial_t-\partial_z)(\partial_t+\partial_z)f=\frac{1}{2}(\partial_t-\partial_z)w+\frac{1}{2}(\partial_t+\partial_z)w
\end{equation}
Therefore, by linearity of the equations, the solution to the equations can be written as $f_1+f_2$ where $f_1$ solves
\begin{equation}
\label{ldal.p.eq1.1}
\begin{split}
&(\partial_t+\partial_z)f_1=\frac{1}{2}w
\\
&f_1(0,z)=\frac{1}{2}\int_{0}^zw(0,s)ds
\end{split}
\end{equation}
and 
\begin{equation}
\label{ldal.p.eq1.11}
\begin{split}
&(\partial_t-\partial_z)f_2=\frac{1}{2}w
\\
&f_2(0,z)=-\frac{1}{2}\int_{0}^z w(0,s)ds
\end{split}
\end{equation}
Therefore, we have:
\begin{equation}
\label{ldal.p.eq2}
f(t,z)=-\frac{1}{2}\int_{z-t}^{z+t}w(0,s)ds+\frac{1}{2}\int_{0}^t w(s,z-t+s) ds +\frac{1}{2}\int_{0}^tw(s,z+t-s)ds
\end{equation}
\end{proof}
Using lemma \ref{ldal} and the d'Alembert's formula, we have:
\begin{equation}
\label{eqb13.1}
\begin{split}
c_{_{\mu}}(t+t^{\prime},z)
&= 
\frac{1}{2}\Big[c_{\mu}(t^{\prime},z-t)+c_{\mu}(t^{\prime},z+t)\big]+\frac{1}{2}\big[\int_{z-t}^{z+t}\big(\partial_0c_{\mu}(t^{\prime},s)ds\big)\big]
\\
&+\frac{1}{2}\int_{0}^{t}\int_{z-(t-\tau)}^{z+t-\tau}(\partial_3h_3+h)(t^{\prime}+\tau,s)ds d\tau\\&
-\frac{1}{2}\int_{z-t}^{z+t}h_0(t^{\prime},s)ds+\int_{0}^t h_0(t^{\prime}+s,z-t+s) ds +\int_{0}^th_0(t^{\prime}+s,z+t-s)ds
\\
&
=\frac{1}{2}\Big[c_{\mu}(t^{\prime},z-t)+c_{\mu}(t^{\prime},z+t)\Big]+\frac{1}{2}\big[\int_{z-t}^{z+t}\big(\partial_0c_{\mu}(t^{\prime},s)ds\big)\big]
\\
&+\frac{1}{2}\int_{0}^{t}\int_{z-(t-\tau)}^{z+t-\tau}h(t^{\prime}+\tau,s)ds d\tau\\&
+\frac{1}{2}\int_{0}^t \Big[h_3(t^{\prime}+s,z+t-s)-h_3(t^{\prime}+s, z-(t-s))\Big]ds
\\&
-\frac{1}{2}\int_{z-t}^{z+t}h_0(t^{\prime},s)ds+\frac{1}{2}\int_{0}^t h_0(t^{\prime}+s,z-t+s) ds +\frac{1}{2}\int_{0}^th_0(t^{\prime}+s,z+t-s)ds
\end{split}
\end{equation}
\end{proof}
\item 
\begin{proposition}
\label{eq.prop2}
For any $t,t^{\prime}$, we have:
\begin{equation}
\label{eqb20.new}
\begin{split}
||c_{\mu}(t+t^{\prime},.)||_{H^3}\le & \big[||c_{\mu}(t^{\prime},.)||_{H^3}+C(t+1) ||\partial_0 c_{\mu}(t^{\prime})||_{H^2}\big]
\\
+&C\epsilon(t+1)\Big(||\psi(t^{\prime},.)||_{H^2}+||\partial_0 \tilde{a}_0(t^{\prime},.)||_{L^2}+||\partial_0\tilde{a}_3(t^{\prime},.)||_{L^2}\Big)
\\
+&Ct\epsilon\big{[} \sup_{(t^{\prime},t+t^{\prime})}||\psi(\tau,.)||_{H^2}+\sup_{(t^{\prime},t+t^{\prime})}||\partial_0 \tilde{a}_0(\tau,.)||_{H^1}+\sup_{(t^{\prime},t+t^{\prime})}||\partial_0 \tilde{a}_3(\tau,.)||_{H^1}\big{)}\big{]}\\
+&Ct(t+1)\epsilon^2\big{[} \sup_{(t^{\prime},t+t^{\prime})}||\psi(\tau,.)||_{H^2}+\sum_{j\in\{0,3\}}\sup_{(t^{\prime},t+t^{\prime})}||\partial_t\tilde{a}_j||_{H^1}]
\\
+&Ct(t+1)\big{[}\sum_{j\in\{0,3\}}\sup_{(t^{\prime},t+t^{\prime})} ||N_j(\tau,.)||_{H^2}+\sum_{j\in\{0,3\}}\sup_{(t^{\prime},t+t^{\prime})}||E_j(\tau,.)||_{H^2}\big{]}
\end{split}
\end{equation}
and
\begin{equation}
\label{eq.prop3.eqb24}
\begin{split}
||\partial_0 c_{\mu}(t+t^{\prime},.)||_{H^2}
&\le || c_{\mu}(t^{\prime},.)||_{H^2}+||\partial_0 c_{\mu}(t^{\prime},.)||_{H^2}\\
&+ C\epsilon(t+1)\Big[\sup_{(t^{\prime},t+t^{\prime})}||\psi(\tau,.)||_{H^2}+\sup_{(t^{\prime},t+t^{\prime})}||\partial_0 \tilde{a}_0(\tau,.)||_{H^1}+\sup_{(t^{\prime},t+t^{\prime})}||\partial_0 \tilde{a}_3(\tau,.)||_{H^1}\Big]\\
&+C\epsilon (t+1)\Big[\sup_{(t^{\prime},t+t^{\prime})}||N(\tau,.)||_{H^3}+ \sup_{(t^{\prime},t+t^{\prime})}||E(\tau,.)||_{H^3}\Big)\Big]
\end{split}
\end{equation}
\end{proposition}
\begin{proof}
We use proposition \ref{int.prop1} and therefore, we estimate the terms on the right hand side of equation \eqref{eqb13.1.2}. There holds:
\begin{equation}
\label{eqb13.1001}
||c_{\mu}(t^{\prime},z-t)+c_{\mu}(t^{\prime},z+t)||_{H^3}\le 2||c_{\mu}(t^{\prime},.)||_{H^3}
\end{equation}
To estimate 
\[
\int_{z-t}^{z+t} (\partial_0c_{\mu})(t^{\prime},s)ds
\]
we use the following lemma
\begin{lemma}
\label{les1}
Suppose that
\[
w(z)=\int_{z-t}^{z+t}f(s)ds
\]
\begin{equation}
\label{les1.eq1}
||w||_{H^3}\le C\big(t||f||_{L^2}+||f||_{H^2}\big)
\end{equation}
\end{lemma}
\begin{proof}
We have:
\begin{equation}
\label{les.eq2}
||w||_{L^2}\le 2t ||f||_{L^2}
\end{equation}
Also,
\begin{equation}
\label{les.eq3}
w_z(z)=f(z-t)-f(z+t)
\end{equation}
Therefore,
\begin{equation}
\label{les.eq4}
||w_z||_{H^2}\le C||f||_{H^2}
\end{equation}
Equations \eqref{les.eq2} and \eqref{les.eq4} finish the proof.
\end{proof}
Therefore, by lemma \ref{les1}, we have:
\begin{equation}
\label{eqb13.89}
||\int_{z-t}^{z+t}\big(\partial_0c_{\mu}(t^{\prime},s)ds\big)||_{H^3}\le C\big{(}t||\partial_0 c_{\mu}(t^{\prime},.)||_{L^2}+||\partial_0 c_{\mu}(t^{\prime},.)||_{H^2}\big{)}
\end{equation}
Similarly, we have:
\begin{equation}
\label{eqb13.891}
||\int_{z-t}^{z+t}h_0(t^{\prime},s)ds||_{H^3}\le C\big{(}t||\partial_0 c_{\mu}(t^{\prime},.)||_{L^2}+||\partial_0 c_{\mu}(t^{\prime},.)||_{H^2}\big{)}
\end{equation}
Now, we want to estimate
\begin{equation}
\label{eqb13.892}
\Big{|}\Big{|}\int_{0}^{t}\int_{z-(t-\tau)}^{z+t-\tau}h(t^{\prime}+\tau,s)ds \Big{|}\Big{|}_{H^3}
\end{equation}
We have:
\begin{claim}
\label{c1.es1}
\begin{equation}
\label{eqb13.893}
\Big{|}\Big{|}\int_{0}^{t}\int_{z-(t-\tau)}^{z+t-\tau}h(t^{\prime}+\tau,s)ds \Big{|}\Big{|}_{H^3}\le Ct\sup_{(t^{\prime},t^{\prime}+t)}\Big(t||h(t^{\prime}+\tau,.)||_{L^2}+||h(t^{\prime}+\tau,.)||_{H^2}\Big)
\end{equation}
\end{claim}
Now, we want to estimate the above quantities. 
\begin{proof}
\begin{equation}
\label{eqb13.894}
\begin{split}
\Big{|}\Big{|}\int_{0}^{t}\int_{z-(t-\tau)}^{z+t-\tau}h(t^{\prime}+\tau,s)ds \Big{|}\Big{|}_{H^3}&
\le t\sup_{\tau}\Big{|}\Big{|}\int_{z-(t-\tau)}^{z+t-\tau}h(t^{\prime}+\tau,s)ds\Big{|}\Big{|}_{H^3}
\end{split}
\end{equation}
By using lemma \ref{les1}, we have:
\begin{equation}
\label{eqb13.895}
\begin{split}
\Big{|}\Big{|}\int_{z-(t-\tau)}^{z+(t-\tau)}h(t^{\prime}+\tau,s)ds\Big{|}\Big{|}_{H^3}\le 
C\Big(t||h(t^{\prime}+\tau,.)||_{L^2}+||h(t^{\prime}+\tau,.)||_{H^2}\Big)
\end{split}
\end{equation}
Equations \eqref{eqb13.894} and \eqref{eqb13.895} imply \eqref{eqb13.893}.
\end{proof}
By differentiating with respect to $z$, we have:
\begin{equation}
\label{eqb13.896}
\Big{|}\Big{|}\frac{1}{2}\int_{0}^t \big[h_3(t^{\prime}+s,z+t-s)-h_3(t^{\prime}+s, z-(t-s))\big]ds\Big{|}\Big{|}_{H^3}\le Ct\sup_{\tau\in (t^{\prime},t^{\prime}+t)}||h_3(\tau,.)||_{H^3}
\end{equation}
Similarly, we have:
\begin{equation}
\label{eqb13.897}
\Big{|}\Big{|}\int_{0}^t h_0(t^{\prime}+s,z-t+s) ds +\int_{0}^th_0(t^{\prime}+s,z+t-s)ds\Big{|}\Big{|}_{H^3}\le Ct\sup_{\tau\in (t^{\prime},t^{\prime}+t)}||h_0(\tau,.)||_{H^3}
\end{equation}
By using \eqref{eqb13.1} and combining estimates \eqref{eqb13.1001}, \eqref{eqb13.89}, \eqref{eqb13.891},\eqref{eqb13.893}, \eqref{eqb13.896} and \eqref{eqb13.897}, we have:
\begin{equation}
\label{eqb14}
\begin{split}
||c_{\mu}(t+t^{\prime},.)||_{H^3}\le &||c_{\mu}(t^{\prime},.)||_{_{H^3}}+C\big{(}t||\partial_0 c_{\mu}(t^{\prime},.)||_{L^2}+||\partial_0 c_{\mu}(t^{\prime},.)||_{H^2}\big{)}
\\
+&C\big{(}t||h_0(t^{\prime},.)||_{L^2}+||h_0(t^{\prime},.)||_{H^2}\big{)}
\\+&
C\big{[}t\sup_{(t^{\prime},t+t^{\prime})}||h_3(\tau,.)||_{H^3}+t\sup_{(t^{\prime},t+t^{\prime})}||h_0(\tau,.)||_{H^3}\big{]}
\\
+&Ct\sup_{(t^{\prime},t^{\prime}+t)}\Big(t||h(\tau,.)||_{L^2}+||h(\tau,.)||_{H^2}\Big)
\end{split}
\end{equation}
for some $C>0$. But, by \eqref{eqb11}, we have: 
\begin{equation}
\label{eqb15}
||h_0(\tau,.)||_{H^3}\le C(||a_0\tilde{\varphi}(\tau,.)||_{H^3}+||\chi_0D_0\varphi(\tau,.)||_{H^3}+||\chi_3D_3\varphi(\tau,.)||_{H^3}+\epsilon ||\psi(\tau,.)||_{H^3})
\end{equation}
According to equations \eqref{eq5.6} and lemma \ref{3d.l1} and remark \ref{reg.3d} in the appendix, we have:
\begin{equation}
\label{eqb16}
||\chi_j(\tau,.)||_{H^3} \le C ||\partial_t \tilde{a}_j(\tau,.)||_{H^1}
\end{equation}
for some $C>0$ and any $\tau$ and $j=0,3$. According to the estimates that we have for $a_0$ and $D_0\varphi$ and $D_3\varphi$, by \eqref{eqb15} and \eqref{eqb16}, we have:
\begin{equation}
\label{eqb17}
||h_0(\tau,.)||_{H^3}\le C\epsilon (||\psi(\tau,.)||_{H^3}+||\partial_0 \tilde{a}_0(\tau,.)||_{H^1}+||\partial_0 \tilde{a}_3(\tau,.)||_{H^1})
\end{equation}
Similarly, by \eqref{eqb12}, we have:
\begin{equation}
\label{eqb18}
||h_3(\tau,.)||_{H^3}\le C\epsilon ||\psi(\tau,.)||_{H^3}
\end{equation}
To estimate $||h(\tau,.)||_{H^2}$ defined in \eqref{eqb13}, we first estimate $||f_j(\tau,.)||_{H^2}$. By equation \eqref{eq5.601} and \eqref{eqb16}, we have:
\begin{equation}
\label{eqb18.1}
||f_j(\tau,.)||_{H^2} \le C\big{[}\epsilon ||\psi(\tau,.)||_{L^2}+\epsilon ||\partial_t \tilde{a}_j||_{H^1}+||N_j(\tau,.)||_{L^2}+||E_j(\tau,.)||_{L^2}\big{]}
\end{equation}
Also, by equation \eqref{new.2.3.neq1}, we have:
\begin{equation}
\label{eqb20}
||\big{(}R_0 u(\tau,.),n_{\mu,\varphi}(\tau,.)\big{)}||_{H^2} \le C\epsilon^2 ||\psi(\tau,.)||_{H^2}
\end{equation}
and by equation \eqref{new.eq7.5},
\begin{equation}
\label{eqb21}
||(S_j u(\tau,.),n_{\mu,\varphi}(\tau,.)\big{)}||_{H^2} \le C\epsilon^2 ||\psi(\tau,.)||_{H^2}
\end{equation}
Therefore, by equation \eqref{eqb13}, we have:
\begin{equation}
\label{eqb19}
||h(\tau,.)||_{H^2}\le C\Big{[}\epsilon^2 ||\psi(\tau,.)||_{H^2}+\sum_{j\in\{0,3\}}\big{[}\epsilon^2||\partial_t\tilde{a}_j||_{H^1}+||N_j(\tau,.)||_{H^2}+||E_j(\tau,.)||_{H^2}\big{]}\Big{]}
\end{equation}
Therefore, by equations \eqref{eqb14}, \eqref{eqb17}, \eqref{eqb18} and \eqref{eqb19}, we have:
\begin{equation}
\label{eqb20}
\begin{split}
||c_{\mu}(t+t^{\prime},.)||_{H^3}\le & \big[||c_{\mu}(t^{\prime},.)||_{H^3}+C(t+1) ||\partial_0 c_{\mu}(t^{\prime})||_{H^2}\big]
\\
+&C\epsilon(t+1)\Big(||\psi(t^{\prime},.)||_{H^2}+||\partial_0 \tilde{a}_0(t^{\prime},.)||_{L^2}+||\partial_0\tilde{a}_3(t^{\prime},.)||_{L^2}\Big)
\\
+&Ct\epsilon\big{[} \sup_{(t^{\prime},t+t^{\prime})}||\psi(\tau,.)||_{H^2}+\sup_{(t^{\prime},t+t^{\prime})}||\partial_0 \tilde{a}_0(\tau,.)||_{H^1}+\sup_{(t^{\prime},t+t^{\prime})}||\partial_0 \tilde{a}_3(\tau,.)||_{H^1}\big{)}\big{]}\\
+&Ct(t+1)\epsilon^2\big{[} \sup_{(t^{\prime},t+t^{\prime})}||\psi(\tau,.)||_{H^2}+\sum_{j\in\{0,3\}}\sup_{(t^{\prime},t+t^{\prime})}||\partial_t\tilde{a}_j||_{H^1}]
\\
+&Ct(t+1)\big{[}\sum_{j\in\{0,3\}}\sup_{(t^{\prime},t+t^{\prime})} ||N_j(\tau,.)||_{H^2}+\sum_{j\in\{0,3\}}\sup_{(t^{\prime},t+t^{\prime})}||E_j(\tau,.)||_{H^2}\big{]}
\end{split}
\end{equation}
Now, we want to control $||\partial_0 c_{\mu}||_{H^2}$.
\begin{equation}
\label{eqb23}
\begin{split}
\partial_0 c_{_{\mu}}(t+t^{\prime},z)&=\frac{1}{2}\partial_t\Big[c_{\mu}(t^{\prime},z-t)+c_{\mu}(t^{\prime},z+t)\big]+\frac{1}{2}\partial_t\big[\int_{z-t}^{z+t}\big(\partial_0c_{\mu}(t^{\prime},s)ds\big)\big]
\\
&+\frac{1}{2}\partial_t\Big[\int_{0}^{t}\big(h_3(t^{\prime}+\tau,z+t-\tau)-h_3(t^{\prime}+\tau,z-(t-\tau))\big)d\tau\Big]
\\
&+\frac{1}{2}\partial_t\Big [\int_{0}^{t}\int_{z-(t-\tau)}^{z+t-\tau}h(t^{\prime}+\tau,s)ds d\tau\Big]\\&
-\frac{1}{2}\frac{\partial}{\partial t}\int_{z-t}^{z+t}h_0(t^{\prime},s)ds+\frac{1}{2}\frac{\partial}{\partial t}\int_{0}^t h_0(t^{\prime}+s,z-t+s) ds\\
&+\frac{1}{2}\frac{\partial}{\partial t}\int_{0}^th_0(t^{\prime}+s,z+t-s)ds
\\
&=-\frac{1}{2}(\partial_z c_{\mu})(t^{\prime},z-t)+\frac{1}{2}(\partial_zc_{\mu})(t^{\prime},z+t)+\frac{1}{2}(\partial_0 c_{\mu})(t^{\prime},z+t)+\frac{1}{2}(\partial_0 c_{\mu})(t^{\prime},z-t)\\
&+\frac{1}{2}\Big[\int_{0}^{t}\big((\partial_3 h_3)(t^{\prime}+\tau,z+t-\tau)-(\partial_3 h_3)(t^{\prime}+\tau,z-(t-\tau)\big)d\tau\Big]
\\
&+\frac{1}{2}\int_{0}^t\Big [\big(h(t^{\prime}+\tau,z+t-\tau)+h(t^{\prime}+\tau,z-(t-\tau)\big)d\tau \Big]
\\
&
-\frac{1}{2}\big(h_0(t^{\prime},z+t)+ h_0(t^{\prime},z-t)\big)
+h_0(t^{\prime}+t,z)
\\
&-\frac{1}{2}\int_{0}^t(\partial_3h_0)(t^{\prime}+s,z-t+s)ds+\frac{1}{2}\int_{0}^t(\partial_3h_0)(t^{\prime}+s,z+t-s)ds
\end{split}
\end{equation}
Therefore, by \eqref{eqb17}, \eqref{eqb18} and \eqref{eqb19}, we have:
\begin{equation}
\label{eqb24}
\begin{split}
||\partial_0 c_{\mu}(t+t^{\prime},.)||_{H^2}&\le ||c_{\mu}(t^{\prime},.)||_{H^3}+||\partial_0 c_{\mu}(t^{\prime},.)||_{H^2}+t\Big[\sup_{(t^{\prime},t+t^{\prime})}||h_3(\tau,.)||_{H^3}+\sup_{(t^{\prime},t+t^{\prime})}\big||h(\tau,.)||_{H^2}\Big]
\\
&+C(t+1)\sup_{(t^{\prime},t+t^{\prime})}||h_0(\tau,.)||_{H^3}
\\
&\le || c_{\mu}(t^{\prime},.)||_{H^3}+||\partial_0 c_{\mu}(t^{\prime},.)||_{H^2}\\
&+ C\epsilon(t+1)\Bigg[\sup_{(t^{\prime},t+t^{\prime})}||\psi(\tau,.)||_{H^3}+\sup_{(t^{\prime},t+t^{\prime})}||\partial_0 \tilde{a}_0(\tau,.)||_{H^1}+\sup_{(t^{\prime},t+t^{\prime})}||\partial_0 \tilde{a}_3(\tau,.)||_{H^1}\Bigg]\\
&+C\epsilon (t+1)\Bigg[||N(\tau,.)||_{H^3}+ ||E(\tau,.)||_{H^3}\Big)\Bigg]
\end{split}
\end{equation}
\end{proof}
\end{itemize}
\subsection{Returning to estimates for the perturbation}
\label{sub.so1}
To move from the estimates that we have for the quantities $Q_1, Q_2, \{c_{\mu}\}_{\mu}$ to the estimates that we want to have for $u$, we use propositions \ref{m1l1} and \ref{b1.l91}.
\begin{proposition}
\label{m1l1}
There exists $C>0$ such that for any $t$,
\begin{equation}
\label{eq117}
\sum_{_{\mu}}||c_{\mu}(.,t)||_{_{H^3}}^2+Q_1(t)+Q_2(t)\ge C ||u(.,t)||_{_{H^3}}^2
\end{equation}
\end{proposition}
\begin{proposition}
\label{b1.l91}
There exists $C>0$ such that for any $t$,
\begin{equation}
\label{b1.l92}
||\partial_tu(.,t)||_{H^2}^2\le C\Big(||\partial_t c_{\mu}(.,t)||_{H^2}^2+||c_{\mu}(.,t)||_{H^2}^2+Q_1(t)+Q_2(t)\Big)
\end{equation}
\end{proposition}
The main ingredient in the proof of the above estimates is the coercivity result \ref{eq35.1} which implies that
\begin{proposition}
\label{corepr}
Suppose that $\psi\in H^1(\mathbb{R}^2)$ and let $p=(\phi,\alpha)\in M_N$ is an $N$-vortex configuration. Let
\begin{equation}
\label{corepr.eq1}
c_{\mu}=\langle\psi,n_{\mu}\rangle_{L^2}
\end{equation}
where $\{n_{\mu}\}_{\mu}$ is the introduced orthonormal basis for $T_p M_N$. Then,
\begin{equation}
\label{corepr.eq2}
\int_{\mathbb{R}^2}(L\psi,\psi)+\sum_{\mu}|c_{\mu}|^2\ge C||\psi||_{H^1}^2
\end{equation}
\end{proposition}
We will also use the following estimate:
\begin{proposition}
\label{l2.new}
Suppose that $f\in H^1(\mathbb{R}^2)$ and $p=(\phi,\alpha)\in K\subset M_N$ and $K\subset M_N$ is compact.  Let
\[
A(f)=-\Delta f+|\phi|^2f
\]
There exist a constant $C=C(K)$ such that
\[
Q(f)=\int_{\mathbb{R}^3}\big(Af,f\big)\ge C ||f||_{_{H^1}}^2
\]
\end{proposition}
The proof of this estimate can be found in the proof of lemma ??? in the appendix.
\subsubsection{Proof of proposition \ref{m1l1}}
In the proof, we consider the time $t$ fixed. Then, let
\[
\gamma(z)=q(t,z)
\]
We use the coordinate system $(y,z)=\big((x_1,x_2),z\big)$ for the subspace $\{t\}\times \mathbb{R}^3$. For any
\[
w\in H^1\big(\mathbb{R}^3,(\mathbb{C}\oplus \mathbb{R}^2)\oplus\mathbb{R}\oplus \mathbb{R}\big)
\] with
\[
w=(\psi,w_0,w_3)
\]
Let
\begin{equation}
\label{l4.eq1}
Kw=-\partial_z^2w+\begin{pmatrix}
L\psi
\\
\\
(-\Delta_y+|\phi|^2)w_0
\\
\\
(-\Delta_y+|\phi|^2)w_3
\end{pmatrix}
\end{equation}
where $L$ and $\phi$ correspond to the base point vortex configuration in $\gamma(z)$, for any $z\in \mathbb{R}$.
\\
Let
\begin{equation}
\label{l4.eq1.1}
c_{\mu}(z)=\langle \psi(.,z), n_{\mu}(.,z)\rangle_{L^2}
\end{equation}
and
\begin{equation}
\label{l4.eq2}
E_1(w)=\int_{\mathbb{R}^3}(Kw,w)+\sum_{\mu}|c_{\mu}|_{L^2}^2
\end{equation}
\begin{proposition}
\label{rep.l1.0.1}
There exists a constant $C>0$ such that for any $w\in H^1(\mathbb{R}^3)$, we have
\begin{equation}
\label{rep.l1.eq1}
E_1(w)\ge C||w||_{H^1}^2
\end{equation}
\end{proposition}
\begin{proof}
Note that by proposition \ref{corepr}, for any $z$, we have:
\begin{equation}
\label{rep.l1.eq1.01}
\int_{\mathbb{R}^2}(L\psi,\psi)(y,z)dy+\sum_{\mu}|c_{\mu}|^2(z) \ge C||\psi(.,z)||_{H^1}^2
\end{equation}
By integrating this over $z$, we deduce that
\begin{equation}
\label{rep.l1.eq1.02}
\begin{split}
E_1(w)&\ge ||\partial_z \psi||_{L^2}^2+\int_{\mathbb{R}}\int_{\mathbb{R}^2}(L\psi,\psi)(y,z) dzdy+\sum_{\mu}\int_{\mathbb{R}}|c_{\mu}|^2dz
\\
&\ge C\Big(||\partial_z \psi||_{L^2}^2+\int_{\mathbb{R}}||\psi(.,z)||_{H^1}^2dz\Big)
\\
&\ge C||\psi||_{H^1}^2
\end{split}
\end{equation}
Similarly, by proposition \ref{l2.new}, we have:
\begin{equation}
\label{rep.l1.eq1.03}
\begin{split}
E_1(w)&\ge  C\Big[\sum_{j\in \{0,3\}}||\partial_z w_j||_{L^2}^2+\int_{\mathbb{R}}\int_{\mathbb{R}^2}\big((-\Delta_y w_j+|\phi|^2w_j)w_j\big)dydz\Big]
\\
&\ge C\sum_{j\in \{0,3\}}\Big[||\partial_z w_j||_{L^2}^2+\int_{\mathbb{R}}\int_{\mathbb{R}^2}(|\nabla_y w_j |^2+|\phi|^2w_j^2)dy dz\Big]
\\
&\ge C\sum_{j\in \{0,3\}}\Big[||\partial_z w_j||_{L^2}^2+\int_{\mathbb{R}}||w_j(.,z)||_{_{H^1}}^2 dz\Big]
\\
&\ge C\sum_{j\in\{0,3\}}||w_j||_{_{H^1}}^2
\end{split}
\end{equation}
Inequalities \eqref{rep.l1.eq1.02} and \eqref{rep.l1.eq1.03} finish the proof of proposition \ref{rep.l1.0.1}.
\end{proof}
\begin{proposition}
\label{rep.l2}
There exists a constant $C>0$ such that for any $w\in H^3(\mathbb{R}^3)$, we have
\begin{equation}
\label{rep.l2.eq1}
\int_{\mathbb{R}^3}(K^2w,Kw)+\sum_{\mu}||c_{\mu}||_{H^2}^2+C\epsilon ||w||_{H^1}^2\ge C||Kw||_{H^1}^2
\end{equation}
\end{proposition}
\begin{proof}
We apply proposition \ref{rep.l1.0.1} to $Kw$. Let
\begin{equation}
\label{rep.l1.eq1.9.1}
d_{\mu}(z)=\langle -\partial_z^2 \psi +L\psi,n_{\mu}(z)\rangle_{_{L^2}}
\end{equation}
Then by proposition \ref{rep.l1.0.1}, we obtain that
\begin{equation}
\label{rep.l1.eq1.9.2.new}
\int_{\mathbb{R}^3}(K^2w,Kw)+\sum_{\mu}||d_{\mu}||_{L^2}^2\ge C||Kw||_{H^1}^2
\end{equation}
But, note that 
\begin{equation}
\label{rep.l1.eq1.9.3}
\begin{split}
d_{\mu}(t,z)&=\langle -\partial_z^2w+Lw,n_{\mu} \rangle _{_{L^2}}
\\
&=-\langle \partial_z^2w,n_{\mu} \rangle _{_{L^2}}
\\
&=-\partial_z^2c_{\mu}+2\langle w_{z},(n_{\mu})_z\rangle_{_{L^2}}+\langle w,(n_{\mu})_{zz}\rangle_{_{L^2}}
\end{split}
\end{equation}
Therefore,
\begin{equation}
\label{rep.l1.eq1.9.4}
||d_{\mu}||_{L^2}^2\le ||c_{\mu}||_{H^2}^2+C\epsilon ||w||_{_{H^1}}^2
\end{equation}
Therefore, by \eqref{rep.l1.eq1.9.2.new}, we have:
\begin{equation}
\label{rep.l1.eq1.9.5}
\int_{\mathbb{R}^3}(K^2w,Kw)+\sum_{\mu}||c_{\mu}||_{H^2}^2+C\epsilon||w||_{H^1}^2\ge C||Kw||_{H^1}^2
\end{equation}
\end{proof}
\begin{proposition}
\label{rep.l1.eq2}
There exists a constant $C>0$ such that 
\begin{equation}
\label{rep.l1.eq3}
||Mu||_{H^1}^2+||u||_{H^1}^2 \ge C||u||_{H^3}^2
\end{equation}
\end{proposition}
\begin{proof}
It suffices to show that
\begin{equation}
\label{rep.l1.eq3.1}
||\Delta u||_{_{H^1}}^2\le C\Big(||Mu||_{H^1}^2+||u||_{H^1}^2\Big)
\end{equation}
We have:
\begin{equation}
\label{rep.l1.eq3.1.1}
||\Delta u||_{L^2}^2\le C\Big(||Mu||_{L^2}^2+||u||_{H^1}^2\Big)
\end{equation}
Therefore,
\begin{equation}
\label{rep.l1.eq3.2}
||u||_{H^2}^2\le C\Big(||Mu||_{L^2}^2+||u||_{H^1}^2\Big)
\end{equation}
By replacing $u$ with $\nabla u$ in \eqref{rep.l1.eq3.1.1}, we have:
\begin{equation}
\label{rep.l1.eq3.3}
\begin{split}
||\nabla \Delta u||_{L^2}^2&\le C\big(||M\nabla u||_{L^2}^2+||\nabla u||_{H^1}^2\big)
\\
&\le C\big(||\nabla Mu||_{L^2}^2+||u||_{H^2}^2\big)
\\
&\le C\big(||Mu||_{H^1}^2+||u||_{H^1}^2\big)
\end{split}
\end{equation}
where the last line follows by \eqref{rep.l1.eq3.2}. This implies \eqref{rep.l1.eq3}.
\end{proof}
According to propositions \ref{rep.l1.0.1} and \ref{rep.l2}, we have:
\begin{equation}
\label{coer.pr.eq1}
Q_1+Q_2+\sum_{\mu}||c_{\mu}||_{_{H^2}}^2+C\epsilon||u||_{H^1}^2\ge C\big[||u||_{_{H^1}}^2+||Mu||_{_{H^1}}^2\big]
\end{equation}
This and \eqref{rep.l1.eq3} imply \eqref{eq117}.
\subsubsection{Proof of proposition \ref{b1.l91}}
Using proposition \ref{m1l1}, we have:
\begin{equation}
\label{b1.192.eq1}
\begin{split}
Q_2(t)+Q_1(t)+\sum_{\mu}||\partial_tc_{\mu}||_{H^2}^2+\sum_{\mu}||c_{\mu}||_{H^3}^2&\ge C\Big[||u_t||_{L^2}^2+||(Mu)_t||_{L^2}^2+||u||_{H^3}^2\\
&+\sum_{\mu}||\partial_tc_{\mu}||_{H^2}^2+\sum_{\mu}||c_{\mu}||_{H^3}^2\Big]\\
&\ge C\Big[||u_t||_{L^2}^2+||Mu_t||_{L^2}^2+\sum_{\mu}||(\partial_t\psi,n_{\mu})_{_{L^2}}||_{_{H^2}}^2\Big]\\
&\ge C\Big[||Mu_t||_{_{L^2}}^2+\int_{\mathbb{R}^3}(Mu_t,u_t)\\
&+\sum_{\mu}||(\partial_t\psi,n_{\mu})_{_{L^2}}||_{_{H^2}}^2\Big]
\end{split}
\end{equation}
But by proposition \ref{rep.l1.0.1}, we have:
\begin{equation}
\label{b1.192.eq1.1}
\int_{\mathbb{R}^3}(Mu_t,u_t)+\sum_{\mu}||(\partial_t\psi,n_{\mu})_{_{L^2}}||_{_{H^2}}^2 \ge C||u_t||_{_{H^1}}^2
\end{equation}
According to \eqref{b1.192.eq1} and \eqref{b1.192.eq1.1}, we have:
\begin{equation}
\label{b1.192.eq1.2}
\begin{split}
Q_1(t)+Q_2(t)+\sum_{\mu}||\partial_tc_{\mu}||_{H^2}^2+\sum_{\mu}||c_{\mu}||_{H^3}^2& \ge C \big[||Mu_t||_{_{L^2}}^2+||u_t||_{_{H^1}}^2\big]\\
& \ge C||u_t||_{_{H^2}}^2
\end{split}
\end{equation}
\subsection{Bootstrap}
\label{sec.bo}
\begin{theorem}
\label{thm3bs}
Suppose that the number $n$ in \eqref{er.ans} satisfies $n\ge 6$. There exists $\epsilon_1>0$ such that if $\epsilon<\epsilon_1$ and $u$ solves the equation \eqref{eq7.1} on
\[
[0,T^{\prime})\times \mathbb{R}^3
\]
for some $T^{\prime}<\frac{1}{2\epsilon}$ and
\begin{equation}
\label{eq8bsprf}
||u(t,.)||_{_{H^3}}+||\partial_t u(t,.)||_{_{H^2}}\le K\le \epsilon^3
\end{equation}
and the initial conditions are as mentioned before,
\begin{equation}
\label{eq9bspf}
||u(t,.)||_{_{H^3}}+||\partial_t u(t)||_{_{H^2}}\le C\big(\epsilon t\big)^{\frac{1}{2}}K+\epsilon^4
\end{equation}
for some $C>0$ which depends only on the wave map and the number of iterations in the ansatz.
\end{theorem}
\begin{proof}
We have
\begin{equation}
\label{b1.p.eq1}
||N||_{_{H^2}}\le CK^2
\end{equation}
and 
\begin{equation}
\label{b1.p.eq2}
||E||_{_
{H^3}}\le C\epsilon ^6
\end{equation}
Now, we want to control $Q_1+Q_2$. According to \eqref{eq.prop1.eqb5}, \eqref{eq8bsprf}, \eqref{b1.p.eq1} and \eqref{b1.p.eq2}, we have:
\begin{equation}
\label{b1.p.eq3}
\begin{split}
Q_1(t)+Q_2(t)&\le Q_1(0)+Q_2(0)+ Ct\Big[\big(\epsilon K+K^2+\epsilon^6\big)K+\epsilon K^2\Big]
\\
&\le Ct\epsilon K+C\epsilon^{8}
\end{split}
\end{equation}
and by \eqref{eqb20.new}, we have:
\begin{equation}
\label{b1.p.eq4}
\begin{split}
||c_{\mu}(t)||_{_{H^3}}&\le C\epsilon (t+1)K+ Ct\epsilon K+Ct(t+1)\epsilon^2 K+Ct(t+1)(K^2+\epsilon^6)
\\
& \le Ct\epsilon K+C\epsilon^{4}
\end{split}
\end{equation}
and by \eqref{eq.prop3.eqb24}, we have:
\begin{equation}
\label{b1.p.eq5}
\begin{split}
||\partial_0 c_{\mu}||_{_{H^2}}&\le C\epsilon(t+1)K+C\epsilon(t+1)(K^2+\epsilon^n)\\
& \le C\epsilon t K+C\epsilon^{4}
\end{split}
\end{equation}
Therefore,
\begin{equation}
\label{b1.p.eq6}
\begin{split}
Q_1(t)+Q_2(t)+||c_{\mu}||_{H^3}^2+||\partial_t c_{\mu}||_{_{H^2}}^2&\le Ct\epsilon K^2+ C\epsilon^8
\end{split}
\end{equation}
According to \eqref{b1.p.eq6} and propositions \ref{m1l1}
and \ref{b1.l91}, the inequality \eqref{eq9bspf} holds.
\end{proof}
\subsection{Last Step: Moving forward in time}
To prove theorem \ref{thm2.main}, we do the iterations in the ansatz $3$ times so that by propositions \ref{2.3.p1}, we can ensure that the error $E$ of the approximate solution satisfies
\begin{equation}
\label{ls.eq1}
||E||_{_{H^4}}+||E_t||_{_{H^3}}\le C\epsilon^6
\end{equation}
and choose the initial data as in \eqref{sg.eq4.02} and \eqref{sg.eq4.1002}. We use the local existence statement, theorem \ref{thm2.loc}, with $b=1$ in \eqref{b.cs1} to find a solution $u$ defined over the interval $[0,T_1)$, for some $T_1$. Then, we use the bootstrap theorem \ref{thm3bs} to refine the estimates. Then, we will look at the time $t=T_1$ and if it satisfies condition  in the local existence result, we repeat the same process and then look at the time $t=2T_1$ and continue likewise. We claim that there exists a number $\kappa$ such that one can use the mentioned procedure and find solutions to \eqref{eq7.1} in a time interval of $[0,\frac{\kappa}{\epsilon})$ such that $u$ satisfies
\begin{equation}
\label{lsmf}
||u||_{_{H^3}}+||u_t||_{_{H^2}}\le \epsilon^3
\end{equation}
Suppose that \eqref{lsmf} holds up to the step $l$; interval $[0,lT_1)$, but not for step $(l+1)$, the interval $[lT_1,(l+1)T_1)$. Therefore by the bootstrap statement, theorem \ref{thm3bs}, we have:
\begin{equation}
\label{mft.eq3.1}
||u||_{_{H^3}}+||u_t||_{_{H^2}} \le C(\epsilon t)^{\frac{1}{2}}\epsilon^3+\epsilon^4
\end{equation}
for $0\le t\le lT$.
Therefore, by the local existence theorem, we have:
\begin{equation}
\label{mft.eq3.2}
||u||_{_{H^3}}+||u_t||_{_{H^2}} \le C \Big[C(\epsilon t)^{\frac{1}{2}}\epsilon^3+\epsilon^4\Big]+ \epsilon^5
\end{equation}
over the interval $[lT,(l+1)T)$. Since the condition \eqref{lsmf}
fails to be satisfied on this interval, then by \eqref{mft.eq3.2}
\begin{equation}
\label{mft.eq3.3}
 C(\epsilon (l+1)T)^{\frac{1}{2}} \ge \frac{1}{2}
\end{equation}
if $\epsilon$ is small enough, and therefore,
\begin{equation}
\label{mft.eq3.4}
l+1\ge \frac{1}{4 C^2 T^2 \epsilon}
\end{equation}
and this finishes the proof of the fact that the aforementioned process can be carried on over a time interval of time interval $[0,\frac{\kappa}{\epsilon}]$ with the desired estimate \eqref{lsmf}. Therefore, we have proved that
\begin{proposition}
\label{prop.fin.per}
Provided that the approximate solution $v=(\varphi,a)$ constructed in theorem \ref{2.3.p1} satisfies the error condition \eqref{er.ans} with $n=6$, then one can find $u=(\tilde{\varphi},\tilde{a})$ such that $(v+u)$ solves AHM on a time interval of the form $[0,\frac{\kappa}{\epsilon})$ and
\begin{equation}
\label{prop.fin.per.eq1}
||u(.,t)||_{_{H^3}}+||u_t(.,t)||_{_{H^2}}\le \epsilon ^3
\end{equation}
for every $t$
\end{proposition}

%% file: Lemma_B.tex
\begin{lemma}
\label{al1}
Consider a compact subset $K\subset M_1$. Let $p=(\phi,\alpha)\in K$ . Then for any $\zeta\in L^2(\mathbb{R}^2)$, there exists $u\in H^2(\mathbb{R}^2)$ such that:
\begin{equation}
    \label{eqp0}
    \Delta u-|\phi|^2u=\eta
\end{equation}
Furthermore, it satisfies

\begin{equation}
    \label{eqq0}
    ||u||_{H^2}\le C||\eta||_{L^2}
\end{equation}
for a constant $C=C(K)$. Moreover, if $\eta\in H^2(\mathbb{R}^2)$ and 
\begin{equation}
\label{eqp01}
    |\eta|(x)\le Ae^{-\gamma|x|}
\end{equation}
for some constants $A,\gamma>0$ with $0\le \gamma <1$, then
\begin{equation}
    \label{eqp02}
    |u|(x)\le Be^{-\gamma|x|}
\end{equation}
for some $B=B(K)$.
\end{lemma}
\begin{proof}
First, for any $\eta\in L^2$, consider the following functional:
\begin{equation}
\label{eq1}
\begin{split}
    &I_{{\eta}}:H^1(\mathbb{R}^2)\to \mathbb{R}\\
    &I_{{\eta}}[v]=\int_{\mathbb{R}^2}\Big(\frac{1}{2}|\nabla v|^2+\frac{1}{2}|\phi|^2v^2-v\eta\Big)dx
\end{split}
\end{equation}
If $v\in H^1$, then $|I_{{\eta}}[v]|<\infty$, since $|\phi|^2\le 1$. 
\begin{proposition}
\label{p1.n1}
For $\eta\in L^2(\mathbb{R}^2)$, the functional $I_{\eta}$ has a unique minimizer $u\in H^1$. Furthermore:
\begin{equation}
    \label{eqp04}
    ||u||_{H^1}\le C ||\eta||_{L^2}
\end{equation}
for a constant $C$.
\end{proposition}
\begin{proof}
In the following proof, all of the constants can be chosen in a way to depend uniformly on $p$.
\\
First, we prove that there exist constants $a\ge 0$ and $b\ge 0$ such that:
\begin{equation}
    \label{eq6.new1}
    I_{\eta}[v]\ge a||v||_{_{H^1}}^2-b||\eta||_{L^2}^2
\end{equation}
To prove this, first consider the following functional:
\begin{equation}
    \label{eq7}
    \begin{split}
        &J:H^1(\mathbb{R}^2)\to \mathbb{R}\\
        &J[v]=\int_{\mathbb{R}^2}\Big(\frac{1}{2}|\nabla v|^2+\frac{1}{2}|\phi|^2v^2\Big)dx
    \end{split}
\end{equation}
\begin{proposition}
\label{c1.n1}
There exists a constant $C\ge 0$ such that
\begin{equation}
    \label{eq8}
    J[v]\ge C||v||_{_{H^1}}^2
\end{equation}
for any $v\in H^1\big(\mathbb{R}^2\big)$.
\end{proposition}
\begin{proof}
Suppose not. Then we can find a sequence $\{v_n\}\subset H^1(\mathbb{R}^2)$ such that $J[v_n]\to 0$ and $||v_n||_{_{H^1}}=1$ for any $n$. We know that $|\phi|(x)\to 1$ as $|x|\to\infty$. So, we can find $r>0$ such that $|\phi|(x)>\frac{1}{2}$ if $|x|>r$. Now, for any $n$ consider the functions:
\begin{equation}
    \label{eq9}
    \begin{split}
    &f_n= v_n\Big|_{_{B_{2r}(0)}}\\
    &g_n=v_n\Big|_{_{\mathbb{R}^2\backslash B_{r}(0)}} 
    \end{split}
\end{equation}
Also, consider the functional
\begin{equation}
    \label{eq91}
    \begin{split}
        &J_1: H^1\big(B_{2r}(0)\big)\to \mathbb{R}\\
        &J_1[v]=\int_{B_{2r}(0)}\Big(\frac{1}{2}|\nabla v|^2+\frac{1}{2}|\phi|^2v^2\Big)dx
    \end{split}
\end{equation}
and
\begin{equation}
\label{eq92}
    \begin{split}
        &J_2: H^1\big(\mathbb{R}^2\backslash B_{r}(0)\big)\to \mathbb{R}\\
        &J_2[v]=\int_{\mathbb{R}^2\backslash B_{r}(0)}\Big(\frac{1}{2}|\nabla v|^2+\frac{1}{2}|\phi|^2v^2\Big)dx
\end{split}
\end{equation}
Then, we have:
\begin{align}
    &0\le J_1[f_n]\le J[v_n]\label{eq101}\\
    &0\le J_2[g_n]\le J[v_n]\label{eq102}
    \end{align}
Since $J[v_n]\to 0$ as $n\to \infty$, then by the above inequalities , we have:
\begin{align}
&\lim_{n\to\infty}J_1[f_n]=0\label{eq103}\\
&\lim_{n\to\infty}J_2[g_n]=0\label{eq104}
\end{align}
Now, note that since $|\phi|>\frac{1}{2}$ for $|x|> r$, then using the definition of $J_2$, we can say:
\begin{equation}
    \label{eq11}
    J_2[g_n]\ge \frac{1}{8}||g_n||_{_{H^1}}^2
\end{equation}
Using \eqref{eq104} and \eqref{eq11}, we have:
\begin{equation}
    \label{eq12}
    \lim_{n\to\infty}||g_n||_{_{H^1}}=0 
\end{equation}
Now, we want to prove that:
\begin{equation}
    \label{eq13}
    \lim_{n\to\infty}||f_n||_{_{H^1}}=0 
\end{equation}
Using the definition of $J_1$, and \eqref{eq103}, we can say:
\begin{equation}
    \label{eq131}
    \lim_{n\to\infty}\int_{B_{2r}(0)}|\nabla f_n|^2dx=0
\end{equation}
Therefore, in order to prove \eqref{eq13}, it suffices to show that:
\begin{equation}
    \label{eq132}
    \lim_{n\to\infty}\int_{B_{2r}(0)}f_n^2=0
\end{equation}
Suppose that:
\begin{equation}
    \label{eq14}
    c_n=\frac{1}{|B_{2r}(0)|}\int_{B_{2r}(0)}f_n
\end{equation}
and:
\begin{equation}
\label{eq15}
    h_n=f_n-c_n
\end{equation}
Then, according to one of the Poincare inequalities for a ball, we have:
\begin{equation}
    \label{eq16}
    ||h_n||_{_{L^2(B_{2r}(0))}}\le C ||\nabla f_n||_{_{L^2(B_{2r}(0))}}
\end{equation}
for some constant $C$. Now, according to \eqref{eq131}, and \eqref{eq16}, we have:
\begin{equation}
    \label{eq17}
    \lim_{n\to\infty} ||h_n||_{_{L^2(B_{2r}(0))}}=0
\end{equation}
According to \eqref{eq15}, \eqref{eq17}, in order to prove \eqref{eq132}, it suffices to show that
\begin{equation}
    \label{eq18.new}
    \lim_{n\to\infty}c_n=0
\end{equation}
Suppose not. Then there exists $m>0$, and a subsequence $\{c_{i_n}\}$ of $\{c_n\}$ such that $|c_{i_n}|>m$ for any $n$. Without loss of generality, we can assume that $i_n=n$. So for any $n$, we have:
\begin{equation}
    \label{eq19}
    c_n>m>0
\end{equation}
Now, according to \eqref{eq103}, and the definition of $J_1$, we can say:
\begin{equation}
    \label{eq20}
    \lim_{n\to\infty}\int_{B_{2r}(0)}|\phi|^2\big(h_n+c_n\big)^2dx=0
\end{equation}
Therefore:
\begin{equation}
    \label{eq21}
     \lim_{n\to\infty}\int_{B_{2r}(0)}|\phi|^2\big(h_n^2+c_n^2+2h_nc_n\big)dx=0
\end{equation}
But, according to \eqref{eq17}, and the fact that $\phi$ is continuous, we have:
\begin{equation}
    \label{eq22}
    \lim_{n\to\infty}\int_{B_{2r}(0)}|\phi|^2h_n^2dx=0
\end{equation}
So, by \eqref{eq21}, \eqref{eq22}, we have:
\begin{equation}
    \label{eq23}
    \lim_{n\to\infty}\int_{B_{2r}(0)}|\phi|^2\big(c_n^2+2h_nc_n\big)dx=0
\end{equation}
According to \eqref{eq19}, \eqref{eq23}, we have:
\begin{equation}
    \label{eq24}
    \lim_{n\to\infty}\int_{B_{2r}(0)}|\phi|^2\big(c_n+2h_n\big)dx=0
\end{equation}
But according to \eqref{eq17}, we have:
\begin{equation}
    \label{eq25}
    \lim_{n\to\infty}\int_{B_{2r}(0)}|\phi|^2h_n dx=0
\end{equation}
Now, using \eqref{eq24}, \eqref{eq25}, we can say:
\begin{equation}
    \label{eq26}
    \lim_{n\to\infty}c_n=0
\end{equation}
which is contradiction with \eqref{eq19}. This shows that \eqref{eq18.new} holds, and this implies \eqref{eq13}, as mentioned before. Now, note that:
\begin{equation}
    \label{eq27}
    ||v_n||_{_{H^1(\mathbb{R}^2)}}^2\le ||f_n||_{_{H^1}}^2+||g_n||_{_{H^1}}^2
\end{equation}
Using \eqref{eq12}, \eqref{eq13}, and \eqref{eq27}, we can say that:
\begin{equation}
    \label{eq28}
    \lim_{n\to\infty}||v_n||_{_{H^1(\mathbb{R}^2)}}=0
\end{equation}
But this is a contradiction, since in the beginning we assumed that $||v_n||_{_{H^1}}=1$ for any $n$. This finishes the proof of proposition \ref{c1.n1}.
\end{proof}
Now, we want to find $a,b\ge 0$ such that \eqref{eq6.new1} holds. We have:
\begin{equation}
    \label{eq281}
    I_{\eta}[v]=J[v]-\int_{\mathbb{R}^2}v\eta  dx
\end{equation}
Consider the constant $C>0$ which satisfies \eqref{eq8}. Now, note that:
\begin{equation}
    \label{eq29}
    \int_{\mathbb{R}^2}v\eta  dx \le \frac{C}{4}||v||_{{H^1}}^2+\frac{1}{C}||\eta||_{_{L^2}}^2
\end{equation}
Using \eqref{eq8}, \eqref{eq281}, and \eqref{eq29}, we have:
\begin{equation}
    \label{eq30}
    I_{\eta}[v]\ge \frac{3C}{4} ||v||_{_{H^1}}^2-\frac{1}{C}||\eta||_{_{L^2}}^2
\end{equation}
Therefore \eqref{eq6.new1} holds for some constants $a,b>0$. 
\\
\\
Now, we want to show that the functional $I_{\eta}$ has a minimizer in $H^1$. Note that \eqref{eq6.new1} tells us that $\inf(I_{\eta})\neq (-\infty)$. Now, consider a minimizing sequence $\{v_n\}$. According to \eqref{eq6.new1}, we can say that $\{v_n\}$ is bounded in $H^1$. Therefore, there exists a subsequence $\{v_{n_k}\}_{_{k=1}}^{\infty}$ of $\{v_n\}$ and $v_0\in H^1$ such that 
\begin{equation}
    \label{eq31}
    v_{n_{k}}\to v_0 \hspace{.2cm} \textbf{weakly in $H^1$}
\end{equation}
Now, note that the functional $I_{\eta}$ is lower semicontinuous with respect to the weak topology of $H^1$. Therefore, $v_0$ is a minimizer for $I_{\eta}$.
To show the uniqueness of the minimizer, suppose that $u_1$, $u_2$ are two minimizers of $I_{\eta}$. Then, we have:
\begin{equation}
    \label{eq3121}
    I_{\eta}\big[\frac{u_1+u_2}{2}\big]=\frac{1}{2}\big(I_{\eta}[u_1]+I_{\eta}[u_2]\big)-\int_{\mathbb{R}^2}\Big[\frac{1}{4}|\nabla(u_1-u_2)|^2+|\phi|^2\frac{1}{4}(u_1-u_2)^2\Big]
\end{equation}
Therefore, we have $\nabla u_1=\nabla u_2$ in $L^2$, and since $u_1,u_2\in {L}^2$, then we deduce that $u_1=u_2$. So the minimizer is unique.
\\
\\
Now, to prove \eqref{eqp04}, note that according to \eqref{eq30}, we have:
\begin{equation}
    \label{eq311}
   \frac{3C}{4}||u||_{_{H^1}}^2-\frac{1}{C}||\eta||_{_{L^2}}^2\le I[u]\le I[0]=0
\end{equation}
Therefore,
\begin{equation}
    \label{eq312}
    ||u||_{_{H^1}}^2\le \frac{4}{3C^2}||\eta||_{_{L^2}}^2
\end{equation}
Note that the constant $C$ above depends only on $\phi$, since it was the constant we obtained in proposition \eqref{c1.n1}. This proves \eqref{eqp04}.
\end{proof}
\begin{proposition}
\label{prop2}
Suppose that $\eta\in L^2{R}^2)$. If $u$ is the minimizer of $I_{\eta}$ in $H^1(\mathbb{R}^2)$, then it satisfies the equation
\begin{equation}
\label{eqp05}
-\Delta u+|\phi|^2u=\eta
\end{equation}
in the weak sense.
\end{proposition}
\begin{proof}
Suppose that $v\in C^{\infty}_c(\mathbb{R}^2)$. Then, for any $c\in \mathbb{R}$, we have:
\[
I_{\eta}(u)\le I_{\eta}(u+cv)
\]
Therefore, 
\[
\int_{\mathbb{R}^2}\nabla u.\nabla v+|\varphi|^2 uv-v\eta=0
\]
\end{proof}
\begin{proposition}
\label{prop3}
Suppose that $\eta\in C^{\infty}_{c}(\mathbb{R}^2)$, and $u\in{H}^1(\mathbb{R}^2)$ satisfies the equation
\begin{equation}
\label{eqp051}
-\Delta u+|\phi|^2u=\eta
\end{equation}
in the weak sense. Then, $u\in{H}^2(\mathbb{R}^2)$, and:
\begin{equation}
    \label{eq511}
    ||u||_{H^2}\le C||\eta||_{L^2}
\end{equation}
where $C$ is a constant.
\end{proposition}
\begin{proof}
First, we use the following theorem from \cite{Evans}:
 \begin{theorem}[Theorem 1 in Page 329 of \cite{Evans}]
    \label{thm1.evans}
    Suppose that $U$ is a bounded and open subset of $\mathbb{R}^n$.
        Consider the elliptic operator
        \begin{equation}
            \label{eqa1}
            Lv=-\sum_{i,j=1}^n\big(a^{ij}v_{x_i}\big)_{x_j}+\sum_{i=1}^n b^{i}(x)v_{x_i}+c(x)v
        \end{equation}
        Assume that
        \begin{equation}
            \label{eqa2}
            a^{ij}\in C^{1}(U), \hspace{.2cm} b^i,c \in L^{\infty}(U)\hspace{.2cm} \big(i,j=1,2,\cdots,n\big) 
        \end{equation}
    Suppose that $v\in H^1(U)$ is a weak solution of the equation
    \begin{equation}
        \label{eqa3}
        Lv=f \hspace{.2cm}\textbf{in}\hspace{.2cm} U
    \end{equation}
    for some $f\in L^2(U)$. Then
    \begin{equation}
        \label{eqa3}
        v\in H^2_{loc}(U)
    \end{equation}
    \end{theorem}
According to theorem \eqref{thm1.evans}, we can say that $u\in H^2_{loc}(\mathbb{R}^2)$. Therefore, $\Delta u \in {L}^2_{loc}(\mathbb{R}^2)$. Now, note that:
\begin{equation}
\label{eq35}
    \Delta u= |\phi|^2u-\eta
\end{equation}
Therefore $\Delta u\in L^2(\mathbb{R}^2)$. Now, consider a smooth cutoff function $\eta:\mathbb{R}^2\to\mathbb{R}$ such that $\eta=1$ for $|x|\le 1$, and $\eta=0$ for $|x|\ge 2$. Consider the functions:
\begin{equation}
    \label{eq36}
    \eta_r(x)=\eta\Big(\frac{x}{r}\Big)
\end{equation}
Now, define the functions:
\begin{equation}
    \label{eq37}
    u_r=u.\eta_r
\end{equation}
Then, note that $u_r\in H^2(\mathbb{R}^2)$, because $u\in H^2_{loc}(\mathbb{R}^2)$. We have:
\begin{equation}
    \label{eq38}
    \Delta u_r=\big(\Delta u\big)\eta_r+2\nabla u. \nabla \eta_r+ u \Delta \eta_r
\end{equation}
Therefore, 
\begin{equation}
    \label{eq39}
    ||\Delta u_r||_{L^2}\le ||\Delta u||_{L^2} + 2||\nabla u. \nabla \eta_r||_{L^2}+||u\Delta \eta_r||_{L^2}
\end{equation}
But note we can find $C$ such that:
\begin{equation}
    \label{eq40}
    \begin{split}
      &|\nabla \eta_r|\le C\hspace{.2cm} \forall \hspace{.1cm} r>1\\
      &|\Delta \eta_r|\le C\hspace{.2cm} \forall \hspace{.1cm} r>1
    \end{split}
\end{equation}
Therefore, by \eqref{eq39}, \eqref{eq40}, we have
\begin{equation}
    \label{eq41}
    ||\Delta u_r||_{L^2}\le ||\Delta u||_{L^2}+ 3C||u||_{H^1}\hspace{.2cm} \forall \hspace{.1cm} r>1
\end{equation}
But, using \eqref{eq35} and the fact that $|\phi|\le 1$, we have:
\begin{equation}
    \label{eq411}
    ||\Delta u||_{L^2}\le ||u||_{L^2}+||\eta||_{L^2} 
\end{equation}
So, by \eqref{eq41} and \eqref{eq411}, we have:
\begin{equation}
    \label{eqg1}
    ||\Delta u_r||_{L^2}\le C ||u||_{H^1}+||\eta||_{L^2}\hspace{.2cm} \forall \hspace{.1cm} r>1
\end{equation}
for a constant $C$.
\\
\\
Now, note that according to propositions \ref{p1.n1} and \ref{prop2}, the functional $I_{\eta}$ has a unique minimizer $u_{\eta}$ in $H^1$ which satisfies the equation:
\begin{equation}
    \label{eq4121}
    -\Delta u_{\eta}+|\phi|^2u_{\eta}=\eta
\end{equation}
in the weak sense. Using this and \eqref{eqp051}, we have:
\begin{equation}
    \label{eq4122}
    -\Delta \big(u_{\eta}-u\big)+|\phi|^2(u_{\eta}-u)=0
\end{equation}
in the weak sense. Since $(u_{\eta}-u)\in H^1(\mathbb{R}^2)$, then \eqref{eq4122} implies that:
\begin{equation}
    \label{eq4123}
    \int_{\mathbb{R}^2}\big|\nabla(u_{\eta}-u)\big|^2+|\phi|^2|u_{\eta}-u|^2=0
\end{equation}
which implies that $u=u_{\eta}$. Therefore, $u$ is the minimizer of $I_{\eta}$ in $H^1$. Therefore, according to proposition \ref{p1.n1}, we have:
\begin{equation}
    \label{eq4124}
    ||u||_{H^1}\le C||\eta||_{L^2}
\end{equation}
where $C$ is a constant. Using \eqref{eqg1} and \eqref{eq4124}, we have:
\begin{equation}
    \label{eq42}
    ||\Delta u_r||_{L^2}\le K ||\eta||_{L^2}\hspace{.2cm} \forall \hspace{.1cm} r>1
\end{equation}
for some constant $K$. Now, we use the following theorem:
\begin{theorem}
\label{thm2}
Suppose that $v\in H^2(\mathbb{R}^n)$. Then, we have:
\begin{equation}
\label{eqa4}
\int_{\mathbb{R}^n}|\Delta v|^2 dx= \sum_{i,j=1}^n\int_{\mathbb{R}^n}|v_{x_ix_j}|^2dx
\end{equation}
\end{theorem}
Using \eqref{eq42}, theorem \eqref{thm2}, and the fact that $u_r\in H^2(\mathbb{R}^2)$, we have:
\begin{equation}
    \label{eq43}
    \int_{\mathbb{R}^2}|(u_r)_{x_ix_j}|^2\le K||\eta||_{L^2}\hspace{.2cm} \forall \hspace{.1cm} r>1
\end{equation}
But, note that:
\begin{equation}
    \label{eq44}
    \int_{B_{_{\frac{r}{2}}}(0)}|u_{x_ix_j}|^2\le \int_{\mathbb{R}^2}|(u_r)_{x_ix_j}|^2
\end{equation}
Now, using \eqref{eq43}, and \eqref{eq44}, we have:
\begin{equation}
    \label{eq45}
    \int_{B_{_{\frac{r}{2}}}(0)}|u_{x_ix_j}|^2\le K||\eta||_{L^2}\hspace{.2cm} \forall \hspace{.1cm} r>1
\end{equation}
This implies that $u_{x_ix_j}\in L^2(\mathbb{R}^2)$, and therefore $u\in H^2(\mathbb{R}^2)$. Furthermore, \eqref{eq4124}, \eqref{eq45} imply \eqref{eq511}. This finishes the proof of proposition \ref{prop3}.
\end{proof}
Now, suppose that $\{\eta_n\}_{_{n=1}}^{\infty}\subset C^{\infty}_{c}(\mathbb{R}^2)$ is such that $\eta_n\to \eta$ in $L^2(\mathbb{R}^2)$. According to propositions \eqref{p1.n1},\eqref{prop2}, and \eqref{prop3}, there exist $\{u_n\}_{_{n=1}}^{\infty}\subset H^2(\mathbb{R}^2)$ solving the equations:
\begin{equation}
    \label{h1}
    -\Delta u_n+|\phi|^2 u_n=\eta_n
\end{equation}
Furthermore, they satisfy:
\begin{equation}
    \label{h2}
    ||u_n||_{H^2}\le C ||\eta_n||_{L^2}
\end{equation}
where $C$ is a constant. According to \eqref{h2}, the sequence $\{u_n\}_{_{n=1}}^{\infty}$ is bounded in $H^2$. Therefore, without loss of generality, we can assume that it is Cauchy in the weak topology of $H^2(\mathbb{R}^2)$. So, suppose that $u_n\to u$ in the weak topology of $H^2(\mathbb{R}^2)$, for some $u\in H^2(\mathbb{R}^2)$. Then, since \eqref{h1} holds in the weak sense for every $n$ and $\eta_n\to\eta$ in $L^2$, then we deduce that $u$ satisfies equation \eqref{eqp0} in the weak sense. Furthermore, according to \eqref{h2}, we have:
\begin{equation}
    \label{h3}
    ||u||_{H^2}\le C||\eta||_{L^2}
\end{equation}
for a constant $C$. This proves \eqref{eqq0}.
\\
\\
Furthermore, note that if $u_1,u_2\in H^2$ satisfy equation \eqref{eqp0}, then we have
\begin{equation}
    \label{h4}
    \int_{\mathbb{R}^2}\big|\nabla(u_1-u_2)\big|^2+|\phi|^2|u_1-u_2|^2=0
\end{equation}
which implies that $u_1=u_2$. Therefore, the solution to equation \eqref{eqp0} is unique.
\\
\\
Now, suppose that $\zeta\in H^2(\mathbb{R}^2)$ and
\begin{equation}
    \label{eq5111}
    |\eta|(x)\le Me^{-\gamma|x|}
\end{equation}
for some $M, \gamma>0$. First, we prove that:
\begin{proposition}
\label{prop4}
\begin{equation}
    \label{eq46}
    \lim_{r\to\infty}\sup_{|x|=r}|u(x)|=0
\end{equation}
\end{proposition}
\begin{proof}
According to the Sobolev embedding theorem, we know that $C^{0,\gamma}\big(B_r(x)\big)\subset H^2\big(B_r(x)\big)$ for any ball $B_r(x)\subset \mathbb{R}^2$, and any $0<\gamma<1$. Furthermore, there exists a constant $C=C(\gamma, r)$ such that
\begin{equation}
     \label{eqq1}
     ||u||_{_{C^{0,\gamma}(B_r(x))}}\le C ||u||_{_{H^2(B_r(x))}}
 \end{equation}
Therefore, since $u\in H^2(\mathbb{R}^2)$, we can say that $u\in C^{0,\gamma}(B_r(x))$ for any ball $B_r(x)\subset \mathbb{R}^2$, and any $0<\gamma<1$. Furthermore, there exists a constant $C=C(\gamma,r)$ such that
\begin{equation}
     \label{eqq2}
     ||u||_{_{C^{0,\gamma}(B_r(x))}}\le {C} ||u||_{_{H^2(\mathbb{R}^2)}}
 \end{equation}
for any $x\in \mathbb{R}^2$. Now, take arbitrary $m>0$. According to \eqref{eqq2}, there exists $\delta>0$ independent of $x$ such that if $|u(x)|>m$, then $|u(y)|>\frac{m}{2}$ for any $y$ with $|x-y|<{\delta}$. Therefore, if $|u(x)|>m$, there exists $\delta>0$ independent of $x$ such that:
\begin{equation}
    \label{eq47}
    \int_{B_{\delta}(x)}|u|^2\ge \pi\delta\frac{m^2}{4}
\end{equation}
On the other hand, since $u\in L^2(\mathbb{R}^2)$, then there exists $R>0$ such that:
\begin{equation}
    \label{eq48}
    \int_{\mathbb{R}^2\backslash B_R(0)}|u|^2 dx < \pi\delta\frac{m^2}{4}
\end{equation}
According to \eqref{eq48}, we can say that \eqref{eq47} does not hold for $|x|>R+\delta$. This implies that for $|x|> R+\delta$, we have $|u(x)|< m$. (Note that \eqref{eq47} holds under the assumption $|u(x)|>m$.) This proves \eqref{eq46}.
\end{proof}
Now, we prove that:
\begin{proposition}
   \label{prop5}
   There exists a constant $C$ such that:
   \begin{equation}
       \label{eq49}
       ||u||_{_{L^{\infty}}}< CM
   \end{equation}
\end{proposition}
\begin{proof}
 Similar to the proof of proposition \eqref{prop4}, note that for any ball $B_r(x)\subset \mathbb{R}^2$, and any $0<\gamma<1$, according to the Sobolev embedding theorem, we have:
 \begin{equation}
     \label{eq50}
     ||u||_{_{C^{0,\gamma}(B_r(x))}}\le {C} ||u||_{_{H^2(\mathbb{R}^2)}}
 \end{equation}
 where $C=C\big(r,\gamma\big)$ is a constant. This implies that:
 \begin{equation}
     \label{eq51}
     ||u||_{_{L^{\infty}(B_r(x))}}\le {C} ||u||_{_{H^2(\mathbb{R}^2)}}
 \end{equation}
 for some universal constant ${C}$. Therefore:
   \begin{equation}
       \label{eq52}
       ||u||_{_{L^{\infty}(\mathbb{R}^2)}}\le {C} ||u||_{_{H^2(\mathbb{R}^2)}}
   \end{equation}
   for the same universal constant ${C}>0$. Now, according to \eqref{h3}, and \eqref{eq52}, we get:
   \begin{equation}
       \label{eq53}
       ||u||_{_{L^{\infty}(\mathbb{R}^2)}}\le C ||\zeta||_{L^2}
   \end{equation}
   where $C$ is a constant. But note that
   \begin{equation}
       \label{eq54}
       ||\eta||_{L^2}\le M||e^{-\gamma |x|}||_{_{L^2(\mathbb{R}^2)}}
   \end{equation}
   Using \eqref{eq53} and \eqref{eq54}, we get the statement of the proposition.
\end{proof}
   Now, to prove the exponential decay of $u$, consider the function:
   \begin{equation}
       \label{eq55}
       s(x)=Ne^{-\gamma|x|}
   \end{equation}
   where $N=N_{\gamma}(M,\phi)$ is a constant which we will find it later. Then, we have:
   \begin{equation}
       \label{eq56}
       \Delta s(x)= \big(-\gamma |x|^{-1}+ \gamma^2\big)s(x)
   \end{equation}
Therefore:
\begin{equation}
    \label{eq57}
    \Delta s(x)- |\phi|^2s(x)=\big(-\gamma |x|^{-1}+\gamma^2-|\phi|^2\big)s(x)
\end{equation}
Using \eqref{eq57}, and the fact that $u$ satisfies equation \eqref{eqp0}, we have:
\begin{equation}
    \label{eq58}
    \begin{split}
        \Delta(s\pm u)(x)-|\phi|^2 (s\pm u)(x)&=\big(-\gamma |x|^{-1}+\gamma^2-|\phi|^2\big)s(x)\mp \zeta (x)
        \\&\le N\big(-\gamma |x|^{-1}+\gamma^2-|\phi|^2\big)e^{-\gamma|x|}+Me^{-\gamma|x|}
    \end{split}
\end{equation}
Therefore:
\begin{equation}
    \label{eq59}
    \Delta(s\pm u)(x)-|\phi|^2 (s\pm u)(x)\le N\big(-\gamma |x|^{-1}+\gamma^2-|\phi|^2\big)e^{-\gamma|x|}+Me^{-\gamma|x|}
\end{equation}
\begin{proposition}
\label{prop6}
There exist constants $R=R(\gamma)$, and $N=N(M,\gamma)$ such that
\begin{equation}
    \label{eq60}
    N\big(-\gamma |x|^{-1}+\gamma^2-|\phi|^2\big)e^{-\gamma|x|}+Me^{-\gamma|x|} <0 \hspace{.3cm}\textbf{if}\hspace{.2cm}|x|>R
\end{equation}
\end{proposition}
and
\begin{equation}
    \label{eq61}
    \big(s\pm u\big)\Big|_{|x|=R} >0
\end{equation}
\begin{proof}
According to theorem 8.1 of \cite{JT80}, we have:
\begin{theorem}
\label{thm3}
For every coupling constant $\lambda>0$, given $\epsilon>0$, we can find $K=K(\lambda,\phi,\epsilon)>0$  such that:
\begin{equation}
    \label{eq62}
    1-|\phi|^2\le Ke^{-(1-\epsilon)m_L|X|}
\end{equation}
where $m_L=\min\{\lambda^{\frac{1}{2}},2\}$.
\end{theorem}
Now, using \eqref{eq62}, and the fact that $\gamma<1$, we can find $R=R(p,\gamma)$ such that:
\begin{equation}
    \label{eq63}
    -\gamma |x|^{-1}+\gamma^2-|\phi|^2 < \frac{\gamma^2-1}{2} \hspace{.3cm}\textbf{if}\hspace{.2cm}|x|>R
\end{equation}
Therefore, if we choose $N=N(M,\gamma)$ such that:
\begin{equation}
    \label{eq64}
    N> \frac{2M}{1-\gamma^2}
\end{equation}
then, condition \eqref{eq60} is satisfied.
\\
\\
Now, note that according to proposition \eqref{prop5}, we have:
\begin{equation}
    \label{eq65}
    \sup_{|x|=R}u< L
\end{equation}
where $L={L}\big(M,\gamma\big)$ is a constant. Therefore, if we choose $N=N(M,\gamma)$ such that:
\begin{equation}
    \label{eq66}
    N>e^{\gamma R} L
\end{equation}
then condition \eqref{eq61} is satisfied. Therefore, if we choose $R=R(\gamma)$ as described, and $N=N(M,\gamma)$ such that both conditions \eqref{eq64}, \eqref{eq66} are satisfied, then both conditions \eqref{eq60}, \eqref{eq61} will be satisfied.
Now, using \eqref{eq59}, and proposition \eqref{prop6}, we see that for the constants $R$, and $N=N(M,\gamma)$, we have
\begin{align}
    &\Delta\big(s\pm u\big)(x)-|\phi|^2\big(s\pm u\big)(x)<0 \hspace{.3cm}\textbf{if}\hspace{.2cm}|x|>R\label{eq67}
    \\& \big(s\pm u\big)\Big|_{|x|=R}>0\label{eq68}
\end{align}
But according to proposition \eqref{prop4}, we deduce that:
\begin{equation}
    \label{eq69}
    \lim_{r\to\infty}\sup_{|x|=r}\big(s\pm u\big)\Big|_{|x|=r}=0
\end{equation}
Using equations \eqref{eq67}, \eqref{eq68} and \eqref{eq69}, and the maximum principle, we have:
\begin{equation}
    \label{eq70}
    (s\pm u)(x)>0\hspace{.3cm}\textbf{if}\hspace{.2cm}|x|>R
\end{equation}
\end{proof}
Equation \eqref{eqq0} and proposition \ref{prop6} implies \eqref{eqp02} and this finishes the proof of lemma \ref{al1}.
\end{proof}

\begin{proposition}
\label{p1}
Suppose that $U\subset \mathbb{R}^N$ and $f:U\to {L}^2\big(\mathbb{R}^2,\mathbb{R}\big)$ and $h:U\to M_N$ are differentiable. Suppose that $u:U\to H^2\big(\mathbb{R}^2,\mathbb{R}\big)$ solves the equation
\begin{equation}
    \label{eq1ln41}
    \Delta u(p)-|\phi|^2(h(p))u(p)=f(p)
\end{equation}
for every $p\in U$. Then, $u$ is differentiable and for every $p\in U$, we have:
\begin{equation}
    \label{eq1}
    \Delta Du(p)-|\phi|^2\big(h(p)\big)Du(p)=Df(p)+\big(D|\phi|^2(h(p))\big)u(p)
\end{equation}
where $D$ represents differentiation with respect to any of the spatial variables in $U$.
\end{proposition}
\begin{proof}
\begin{claim}
\label{c1}
For every $p\in U$, there exists $C>0$ and $\delta>0$ such that for every $q\in U$ with $|p-q|<\delta$, we have
\begin{equation}
    \label{eq2}
    \begin{split}
     \big{|}\big{|}u(p)-u(q)\big{|}\big{|}_{_{H^2}}< C|p-q|
    \end{split}
\end{equation}
\end{claim}
\begin{proof}
Let 
\[
w_{p,q}=u(q)-u(p)
\]
We have:
\begin{equation}
    \label{eq3}
    \begin{split}
        \Delta w_{p,q}-|\phi|^2\big(h(p)\big)w_{p,q}&=\Big(f(q)-f(p)\Big)\\
        &+\Big(|\phi|^2\big(h(q)\big)-|\varphi|^2\big(h(p)\big)\Big)u(q)
    \end{split}
\end{equation}
There exists $\delta_1=\delta_1(p)>0$ and $C_1>0$ such that if $|p-q|<\delta_1$, then
\begin{equation}
    \label{eq4}
    \big{|}\big{|}|\phi|^2\big(h(q)\big)-|\phi|^2\big(h(p)\big{|}\big{|}_{_{L^{\infty}}}< C_1|p-q|
\end{equation}
Therefore, since $f$ is differentiable, for every $p\in U$, there exists $\delta_2, C_2>0$ such that for every $q\in U$ with $|p-q|<\delta_2$,
\begin{equation}
    \label{eq5}
    \Big{|}\Big{|}\big(f(q)-f(p)\big)+\Big(|\phi|^2\big(h(q)\big)-|\phi|^2\big(h(p)\big)\Big)u(q)\Big{|}\Big{|}_{L^2(\mathbb{R}^2)}\le C_2|p-q|
\end{equation}
According to lemma \ref{al1}, \eqref{eq3}, and \eqref{eq5}, we deduce the statement of claim.
\end{proof}
Using the fact that $f$ is differentiable, for every $p\in U$ and $\tau\in T_p U$, there exists a function $v_{_{p,\tau}}\in H^2$ which satisfies:
\begin{equation}
    \label{eq35.25}
    \Delta v_{_{p,\tau}}-|\phi|^2\big(h(p)\big)v_{_{p,\tau}}=Df(p)(\tau)+D{ |\phi|^2}(p)(\tau)u(p)
\end{equation}
Suppose that $p+\tau\in U$. Let: 
\begin{equation}
    \label{eq35.25}
    w_{p,\tau}={u(p+\tau)-u(p)-v_{_{p,\tau}}}
\end{equation}
We have:
\begin{equation}
    \label{eq35.26}
    \begin{split}
    \Delta w_{p,\tau}-|\phi|^2w_{p,\tau}&=\Big(f(p+\tau)-f(p)-Df(p)(\tau)\Big)\\
    &+\Big(|\phi|^2(p+\tau)-|\varphi|^2(p)-D|\phi|^2(p)(\tau)\Big)u(p)\\
    &+\Big(|\phi|^2(p+\tau)-|\phi|^2(p)\Big)\Big(u(p+\tau)-u(p)\Big)
    \end{split}
\end{equation}
According to the facts that $f$ and $h$ are differentiable, and claim \ref{c1}, for any $\epsilon>0$, there exists $\delta>0$ such that if $|\tau|<\delta$, then
\begin{equation}
    \label{eq35.27}
    \begin{split}
       &||f(p+\tau)-f(p)-Df(p)(\tau)||_{L^2(\mathbb{R}^2)}\\
       +&||\Big(|\phi|^2(p+\tau)-|\phi|^2(p)-D|\phi|^2(p)(\tau)\Big)u(p)||_{L^2(\mathbb{R}^2)}\\
       +&||\Big(|\phi|^2(p+\tau)-|\phi|^2(p)\Big)\Big(u(p+\tau)-u(p)\Big)||_{L^2(\mathbb{R}^2)}\\
       \le& \epsilon |\tau|
    \end{split}
\end{equation}
 According to \eqref{eq35.26} and \eqref{eq35.27} and lemma \ref{al1}, we deduce that for every $\epsilon>0$, there exists $\delta_1>0$ such that if $|\tau|<\delta_1$, then
 \begin{equation}
     \label{eq35.38}
     ||w_{p,\tau}||_{H^2(\mathbb{R}^2)}\le \epsilon|\tau|
 \end{equation}
 Therefore, $u$ is differentiable.
\end{proof}
 \begin{lemma}
\label{ln42}
Suppose that $U\subset \mathbb{R}^N$ and $f:U\to {L}^2\big(\mathbb{R}^2,\mathbb{R}\big)$ is $m$-times differentiable for $m\in \mathbb{N}$ and $h:U\to M_N$ is $(m+2)$-times differentiable. Suppose that $u:U\to H^2\big(\mathbb{R}^2,\mathbb{R}\big)$ solves the equation
\begin{equation}
    \label{eq1ln41}
    \Delta u(p)-|\phi|^2(h(p))u(p)=f(p)
\end{equation}
for every $p\in U$. Then, $u$ is $m$-times differentiable 
\end{lemma}
\begin{proof}
The base case holds by the previous statement. Suppose that the statement holds for $m$. According to the base case and proposition \ref{p1}, we have:
\begin{equation}
    \label{eq3ln41}
    \Delta Du(p)-|\phi|^2\big(h(p)\big)Du(p)=Df(p)+\big(D|\phi|^2\big(h(p)\big)u(p)
\end{equation}
where $D$ represents differentiation with respect to any spatial direction in $U$. According to the base case, we have:
\begin{equation}
    \label{eq3ln41}
    \Delta Du(p)-|\phi|^2\big(h(p)\big)Du(p)=Df(p)+\big(D|\phi|^2\big(h(p)\big)u(p)
\end{equation}
where $D$ represents differentiation with respect to any spatial direction in $U$. The function $g:U\to H^2\big(\mathbb{R}^2,\mathbb{R}\big)$ is defined by: 
\[
g(p)=\Big(D|\phi|^2(h(p))\Big)u(p)
\]
\begin{claim}
   \label{cl1ln41}
   The function $g$ is $m$-times differentiable. 
\end{claim}
\begin{proof}
According to the induction hypothesis, $u$ is $m$-times differentiable.The function
\[
e:U\to H^2(\mathbb{R}^2,\mathbb{R}) 
\]
defined by
\[
e(p)=\big(D|\phi|^2(h(p))\big)
\]
is $m$-times differentiable. Therefore, by the Sobolev embedding theorem, $g$ is $m$-times differentiable.
\end{proof}
Claim \ref{cl1ln41}, \eqref{eq3ln41} and the induction hypothesis imply that $Du$ is $m$-times differentiable. Therefore, $u$ is $(m+1)$-times differentiable.
\end{proof}
\begin{lemma}
\label{l3}
Suppose that $U$ is an open subset of $\mathbb{R}^2$, $f\in \mathcal{E}_m(\mathbb{R}^2,U,\mathbb{R}^2)$ and $h:U\to M_N$ is $m$-times differentiable and $h(U)$ is a precompact subset of $M_N$ and $D^s(h)$ is bounded for every $s$ with $|s|\le m$. Suppose that $u:\mathbb{R}^2\times U\to \mathbb{R}$ satisfies the equation
\begin{equation}
    \label{eq35.2}
    \Delta u(,p)-|\phi|^2\big(h(p)\big)u(,p)=f(,p)
\end{equation}
for every $p\in U$ and $u(,p)\in H^2$ for every $p\in U$.
Then, $u\in \mathcal{E}_{m-3}(\mathbb{R}^2,U,\mathbb{R}^2)$.
\end{lemma}
\begin{proof}
Suppose that $V=\mathbb{R}^2\times U$. Consider the functions
\begin{equation}
    \label{35.22}
\begin{split}
    &\hat{f}:V\to L^2\big(\mathbb{R}^2,\mathbb{R}\big)\\
    &\hat{f}(x,p)(y)=f\big(y-x,p\big)\hspace{.4cm}\forall x,y\in \mathbb{R}^2, p\in U
\end{split}
\end{equation}

\begin{equation}
    \label{35.221}
\begin{split}
    &\hat{u}:V\to H^2\big(\mathbb{R}^2,\mathbb{R}\big)\\
    &\hat{u}(x,p)(y)=u\big(y-x,p\big)\hspace{.4cm}\forall x,y\in \mathbb{R}^2, p\in U
\end{split}
\end{equation}
\begin{equation}
    \label{35.222}
    \begin{split}
        &\hat{h}:V\to M_2\\
        &\hat{h}(x,p)(y)=h\big(y-x,p\big)\hspace{.4cm}\forall x,y\in \mathbb{R}^2, p\in U
    \end{split}
\end{equation}
The function $\hat{h}$ is $m$-times differentiable. For every $q\in V$, we have
\begin{equation}
    \begin{split}
        \Delta \hat{u}(q)-|\varphi|^2\big(\hat{h}(q)\big)\hat{u}(q)=\hat{f}(q)
    \end{split}
\end{equation}

Suppose that $q\in V$. For any multi-index $r$ with $|r|\le m$, the function $x\to \big(D^r\hat{f}(q)\big)(x)$ has exponential decay as $|x|\to\infty$. Therefore, $\hat{f}$ is $m$-times differentiable. Therefore, according to lemma \ref{ln42}, the function $\hat{u}$ is $m$-times differentiable. Therefore, the function $u$ is $m$-times differentiable and $D^ru(.,p)\in H^2$ for any multi-index $r$ with $|r|\le m$, and for any multi-index $s$ with $|s|\le (m-3)$, we have:
\begin{equation}
    \label{35.223}
    \Delta \Big(D^su(q)\Big)(x)-\Big(D^s\big(|\varphi|^2\big(h(q)\big){u}(q)\big)\Big)(x)=\big(D^sf(q)\big)(x)
\end{equation}
Therefore, by induction on $|s|$, the facts that $h(U)$ is a precompact subset and $f\in \mathcal{E}_m(\mathbb{R}^2,U,\mathbb{R}^2)$, we deduce that $f\in \mathcal{E}_{m-3}(\mathbb{R}^2,U,\mathbb{R}^2)$.
\end{proof}
\begin{lemma}
\label{3d.l1}
Consider a smooth curve $\gamma:\mathbb{R}\to K$ where $K$ is a compact subset of $M_N$. Suppose that
\begin{equation}
\label{3d.l1.eq0}
\phi(y,z)=\phi(y;\gamma(z))
\end{equation}
Also, suppose that $|D\gamma|\le m$ for some $m>0$. Then, there exists a number $\delta>0$ such that if $f\in H^2(\mathbb{R}^3)$ with $||f||_{_{L^\infty}}\le \delta$, then the equation
\begin{equation}
\label{3d.l1.eq1}
(-\Delta+|\phi+f|^2)u=g
\end{equation}
for $g\in L^2(\mathbb{R}^3)$, has a solution $u\in H^2(\mathbb{R}^3)$ with
\begin{equation}
\label{3d.l1.eq2}
||u||_{_{H^2}}\le C||g||_
{_{L^2}}
\end{equation}
where $C=C(K,m,\delta)$.
\end{lemma}
\begin{proof}
The proof goes by the same arguments as in lemma \ref{al1}. But, here we consider the energy quantity:
\begin{equation}
\label{3d.l1.eq3}
Tu=\int_{\mathbb{R}^3}|\nabla u|^2+|\phi+f|^2u^2
\end{equation}
and we prove a coercivity for that:
\begin{claim}
\label{3d.l1.cl1}
If $\delta$ is small enough, then there exists $C=C(K,m,\delta)>0$ such that
\begin{equation}
\label{3d.l1.eq4}
Tu\ge C ||u||_{_{H^1}}^2
\end{equation}
\end{claim}
\begin{proof}
We have:
\begin{equation}
\label{3d.l1.eq5}
\begin{split}
Tu&\ge \int_{\mathbb{R}}\int_{\mathbb{R}^2}(|\nabla_y u|^2+|\phi+f|^2u)
\\
&\ge  \int_{\mathbb{R}}\int_{\mathbb{R}^2}\big(|\nabla_y u|^2+|\phi|^2u^2-\delta^2u^2\big)dydz
\\
&\ge \int_{\mathbb{R}}\int_{\mathbb{R}^2} (C-\delta^2)u^2dydz
\\
&\ge (C-\delta^2)||u||_{_{L^2}}^2
\end{split}
\end{equation}
where $C$ is the constant provided by claim \ref{c1.n1} in the proof of lemma \ref{al1}. 
\\
On the other hand, we have
\begin{equation}
\label{3d.l1.eq6}
Tu\ge ||\nabla u||_{_{L^2}}^2
\end{equation}
Therefore, by \eqref{3d.l1.eq5} and \eqref{3d.l1.eq6}, we have $Tu\ge C||u||_{_{H^1}}^2$.
\end{proof}
Now, the rest of the proof goes in the same way as in lemma \ref{al1}. We look for the minimizers of the functional 
\begin{equation}
\label{3d.l1.eq7}
\begin{split}
&T:H^1(\mathbb{R}^3)\to \mathbb{R}
\\
&T_gu=Tu-\int_{\mathbb{R}^3}gu
\end{split}
\end{equation}
Following claim \ref{3d.l1.cl1}, one can prove that if one considers a minimizing sequence for $T_g$, then it would be bounded in $H^1$ and one can pass to a weakly convergent subsequence whose limit is a minimizer of $T_{g}$. But minimizers of $T_g$ satisfy equations \eqref{3d.l1.eq1}, and then by elliptic regularity one obtains that the minimizer $u\in H^2(\mathbb{R}^3)$.
\end{proof}
\begin{Remark}:
The weakly lower semicontinuity of the functional $T$ follows from its convexity and differentiability of $T$.
\end{Remark}
\begin{Remark}
\label{reg.3d}
By using standard elliptic regularity results, if in the statement of lemma \ref{3d.l1}, $f\in H^m(\mathbb{R}^3)$, then $u\in H^(m+2)(\mathbb{R}^3)$ and 
\begin{equation}
\label{reg.3d.eq1}
||u||_{_{H^{m+2}}}\le C ||f||_{_{H^m}}
\end{equation}
\end{Remark}

%% file: Lemma_A.tex
\begin{lemma}
\label{al2}
Consider a compact set $K\subset M_N$. Let $p=(\phi,\alpha)\in K$. Then, for any 
\[
\zeta=(\zeta_0,\zeta_1,\zeta_2)\in L^2(\mathbb{R}^2,\mathbb{C})\oplus L^2(\mathbb{R}^2,\mathbb{R})\oplus L^2(\mathbb{R}^2,\mathbb{R})
\]
which is orthogonal to the zero modes at $p$, there exists a unique 
\[\psi=(\tilde{\varphi},\tilde{a}_1,\tilde{a}_2)\in  H^2(\mathbb{R}^2,\mathbb{C})\oplus H^2(\mathbb{R}^2,\mathbb{R})\oplus H^2(\mathbb{R}^2,\mathbb{R})\] 
which is orthogonal to zero modes at $p$ and solves the equation
\begin{equation}
\label{eq1.new}
L[\phi,\alpha]\psi=\zeta
\end{equation}
and
\begin{equation}
\label{eq2.new}
||\psi||_{_{H^2}}\le C ||\zeta||_{_{L^2}}
\end{equation}
Furthermore, if the vector $\zeta$ satisfies the gauge orthogonality condition
\begin{equation}
\label{geq1}
\partial_{x_1}\zeta_1+\partial_{x_2}\zeta_2=(i\phi,\zeta_0)
\end{equation} 
 then $\psi$ satisfies the gauge orthogonality condition:
\begin{equation}
\label{geq2}
\partial_{x_1}\tilde{a}_1+\partial_{x_2}\tilde{a}_2=(i\phi,\tilde{\varphi})
\end{equation}
Furthermore, if the vector $\zeta$ satisfies the exponential decay
\begin{equation}
\label{eqexd.1}
|\zeta(x)|\le Ae^{-\beta{|x|}}
\end{equation}
for some $\gamma$ with $0<\gamma<1$, then there exists numbers $B$ and $R=R(K)$ such that
\begin{equation}
\label{al2.eq1}
\begin{split}
&|\tilde{\varphi}(x)|\le Be^{-\frac{\beta}{2}|x|}\\
&|\tilde{a_j}(x)|\le Be^{-\beta|x|}
\end{split}
\end{equation}
and if $|x|>R$, then
\begin{equation}
\label{al2.eq1.1}
\begin{split}
&|\partial_j\tilde{\varphi}(x)|\le Be^{-\frac{\beta}{2}|x|}\\
&|\partial_j\tilde{a_j}(x)|\le Be^{-\beta|x|}
\end{split}
\end{equation}
for $j=1,2$.
\end{lemma}
\begin{notation}
For any $r>0$, let:
\begin{equation}
    (H^{r})^{\perp}=H^{r}\cap \big(T_{(\phi,\alpha)}M_N\big)^{\perp}
\end{equation}
\end{notation}
\begin{proof}
Consider the functional $I_{_{\zeta}}:(H^{1})^{\perp}\to\mathbb{R}$ defined by:
\begin{equation}
    \label{e1}
    \begin{split}
        I_{_\zeta}[\eta,b]=&\int_{\mathbb{R}^2}\Big[\frac{1}{2}|\nabla b|^2+\frac{1}{2}|D_{\alpha}\eta|^2+\frac{1}{4}(3|\phi|^2-1)|\eta|^2 +\frac{1}{2}|b|^2|\phi|^2\\
        &-2\sum_{j=1}^2\big(i\eta,D_{\alpha_j}\phi\big)b_j-\sum_{j=1}^2\zeta_{_{j}}b_j-(\zeta_{_{0}},\eta)\Big]
    \end{split}
\end{equation}
Note that if $(\eta,b)\in (H^{1})^{\perp}$, then using the fact that $|\phi|$, $D_{\alpha_j}\phi$ are bounded, and $\zeta_{_{j}}\in L^2$, and $\zeta_{_0}\in L^2$, we see that $\big|I_{_{\zeta}}[\eta,b]\big|<\infty$.
\begin{claim}
\label{c1}
The functional $I_{_\zeta}$ has a unique minimizer $(\tilde{\phi},\tilde{a})$ in $(H^{1})^{\perp}$. Furthermore,
\begin{equation}
    \label{m0}
    ||(\tilde{\phi},\tilde{a})||_{_{H^1}}\le C ||\zeta||_{L^2}
\end{equation}
where $C$ is a constant.
\end{claim}
\begin{proof}
Consider the functional $J:H^{1}\to\mathbb{R}$ defined by:
\begin{equation}
    \label{m1}
    \begin{split}
        J[\eta,b]=&\int_{\mathbb{R}^2}\Big[\frac{1}{2}|\nabla b|^2+\frac{1}{2}|D_{\alpha}\eta|^2+\frac{1}{4}(3|\phi|^2-1)|\eta|^2 +\frac{1}{2}|b|^2|\phi|^2\\
        &-2\sum_{j=1}^2\big(i\eta,D_{\alpha_j}\phi\big)b_j\Big]
    \end{split}
\end{equation}

Then, according to (theorem 3.1 of the Stuart's paper)(put the statement somewhere), we have:
\begin{equation}
    \label{m2}
    J[\eta,b]\ge \gamma \big|(\eta,b)\big|_{_{H^1}}^2 \hspace{.5cm}\forall\hspace{.2cm}(\eta,b)\in (H^{1})^{\perp}
\end{equation}
for some constant $\gamma$.
Using \eqref{m2}, we have:
\begin{equation}
    \label{m3}
    \begin{split}
    I_{_\zeta}[\eta,b]&\ge \gamma |(\eta,b)|_{_{H^1}}^2-\int_{\mathbb{R}^2}\big[\sum_{j=1}^2\zeta_{_{j}}b_j+(\zeta_{_{0}},\eta)\big]\\
    &\ge \frac{\gamma}{2}|(\eta,b)|_{_{H^1}}^2-\frac{1}{2\gamma}||\zeta_{_{j}}||_{L^2}^2-\frac{1}{2\gamma}||\zeta_{_{0}}||_{L^2}^2
    \end{split}
\end{equation}
for the constant $\gamma>0$ described above.\\
\\
Now, take a minimizing sequence $\{(\eta_i,b_i)\}_{i=1}^{\infty}\subset H^{1}$ for the functional $I_{_\eta}$. According to \eqref{m3}, this sequence is bounded in $H^{1}$. Therefore, there exists $(\eta_0, b_0)\in H^{1}$ such that $(\eta_i,b_i)\rightharpoonup (\eta_0,b_0)$ in $H^{1}$. Furthermore, since $\{(\eta_i, b_i)\}\subset \big(T_{p}M_N\big)^{\perp}$, then for any $n\in T_{p}M_N$, we have:
\begin{equation}
    \label{m4}
    \big((\eta_i,b_i),n\big)_{L^2}=0
\end{equation}
Now, since $(\eta_i,b_i)\rightharpoonup (\eta_0,b_0)$ in $H^{1}$, by \eqref{m4} we have:
\begin{equation}
    \label{m5}
    \big((\eta_0,b_0),n\big)_{L^2}=0
\end{equation}
Therefore, $(\eta_0,b_0)\in (H^{1})^{\perp}$. Now, using the fact that $I_{\zeta}$ is lower semicontinuous with respect to the weak topology of $H^{1}$, we deduce that $(\eta_0,b_0)$ is a minimizer for $I_{_\zeta}$ on $(H^{1})^{\perp}$.
\\
\\
Now, suppose that $(\eta_1,b^1)$, $(\eta_2,b^2)$ are both minimizers of the functional $I_{\zeta}$ in $(H^{1})^{\perp}$. Then, we have:
\begin{equation}
    \label{m511}
    I_{\zeta}\big[\frac{\eta^1+\eta^2}{2},\frac{b^1+b^2}{2}\big]-\frac{1}{2}\big(I_{_{\zeta}}[\eta^1,b^1]+I_{_{\zeta}}[\eta^2,b^2]\big)=-J\Big[\frac{\eta^1-\eta^2}{2},\frac{b^1-b^2}{2}\Big]
\end{equation}
Now, according to \eqref{m2}, \eqref{m511}, and the fact that $(\eta^1,b^1)$, $(\eta^2,b^2)$ are both minimizers of the functional $I_{\zeta}$ in $(H^{1})^{\perp}$, we deduce that $b^1=b^2$ and $\eta^1=\eta^2$.
\\
\\
Now, suppose that $(\tilde{\phi},\tilde{a},)$ is the minimizer of $I_{_{\zeta}}$ in $(H^{1})^{\perp}$. Then, by \eqref{m3}, we have:
\begin{equation}
\label{m5111}
\begin{split}
    \big|(\tilde{\varphi},\tilde{a})\big|_{H^1}^2&\le \frac{2}{\gamma}\Big[I_{_{\zeta}}[\tilde{\phi},\tilde{a}]+\frac{1}{2\gamma}\sum_{j=1}^2||\zeta_{_{_j}}||_{L^2}^2+\frac{1}{2\gamma}||\eta_{_{0}}||_{L^2}^2\Big]\\
    &\le \frac{2}{\gamma}\Big[I_{_{\zeta}}[0,0]+\frac{1}{2\gamma}\sum_{j=1}^2||\zeta_{_{j}}||_{L^2}^2+\frac{1}{2\gamma}||\zeta_{_{0}}||_{L^2}^2\Big]\\
    &=\frac{1}{\gamma^2}\Big(\sum_{j=1}^2||\zeta_{_{_j}}||_{L^2}^2+||\zeta_{_{0}}||_{L^2}^2\Big)
\end{split}
\end{equation}
This implies \eqref{m0}.
\end{proof}
\begin{definition}
\label{d1}
Suppose that for for $i,j,k\in \{1,2\}$
\[
m_{ij}^k, n_i^k, p^k, \hat{q}_i\in L^{\infty}(\mathbb{R}^2, \mathbb{R})
\]
and
\[
\hat{m}_{ij}, \hat{n}_i^k, \hat{p}^k, q^k\in L^{\infty}(\mathbb{R}^2, \mathbb{C})
\]
and
\[
\zeta_1^k\in L^2(\mathbb{R}^2)
\]
and $\zeta_2\in L^2(\mathbb{R}^2,\mathbb{C})$. Then, we say that  $(u,v)\in H^{1}$ satisfies the equations:
\begin{align}
    &-\sum_{i,j}m^k_{ij}(u_k)_{_{x_ix_j}}+\sum_{i}n^k_i(u_k)_{_{x_i}}+p^ku_k+(q^k,v)=\zeta_1^k\hspace{.3cm}k=1,2\label{m51112}\\
    &-\sum_{i,j}\hat{m}_{ij}D_{\alpha_i}D_{\alpha_j}v+\sum_{i}\hat{n}_iD_{\alpha_i}v+\hat{p}v+\sum_{i}\hat{q}_iu^i=\zeta_2\label{m51113}
\end{align}
in the weak sense if for any $(f,g)\in H^{1}$, we have:
\begin{align}
    &\int_{\mathbb{R}^2}\sum_{i,j}(u_k)_{_{x_i}}(m^k_{ij}f_k)_{_{x_j}}+\Big(\sum_{i}n^k_i(u_k)_{_{x_i}}+p^ku_k+(q^k,v)\Big)f_k=\int_{\mathbb{R}^2}\zeta_1^kf_k \hspace{.3cm}{k=1,2}\label{m51114}\\
    &\int_{\mathbb{R}^2}\sum_{i,j}\big(D_{\alpha_j}v,D_{\alpha_i}(\hat{m}_{ij}g)\big)+\Big(\sum_{i}\hat{n}_iD_{\alpha_i}v+\hat{p}v+\sum_{i}\hat{q}_iu^i,g\Big)=\int_{\mathbb{R}^2}(\zeta_2,g)\label{m51115}
\end{align}
\end{definition}
\begin{claim}
\label{c11}
The equation
\begin{equation}
\label{m601}
L[\phi,\alpha](\eta,b)=\zeta
\end{equation}
has at most one weak solution in the space $(H^{1})^{\perp}$.
\end{claim}
\begin{proof}
Suppose that $(\eta_1,b_1), (\eta_2,b_2)\in (H^{1})^{\perp}$ are two solutions for the equation \eqref{m601} in the weak sense. Let
\[
(\eta,b)=(\eta_1,b_1)-(\eta_2,b_2)
\]
Then, $(\eta,b)$ satisfies the equations
\begin{equation}
\label{m602}
L[\phi,\alpha](\eta,b)=0
\end{equation}
in the weak sense. Therefore, we have:
\begin{align}
    &-\Delta {b}_j+|\phi|^2{b}_j-2\big(i\eta,D_{\alpha_j}\phi\big)=0\label{m603}\\
    &-D_{\alpha_j}D_{\alpha_j}{\eta}+\frac{1}{2}\big(3|\phi|^2-1\big)|\eta|^2+2i\sum_{j=1}^2(D_{\alpha_j}\phi){b}_j=0\label{m703}
\end{align}
in the weak sense. Now, if we integrate the above equations against the test function $(\eta,b)\in H^{1}$, we deduce that:
\begin{align}
    &\int_{\mathbb{R}^2}\big|\nabla {b}_j\big|^2+|\phi|^2{b}_j^2-2\big(i{\eta},D_{\alpha_j}\phi\big){b}_j=0\label{m604}\\
    &\int_{\mathbb{R}^2}\big|D_{\alpha}\eta\big|^2+\frac{1}{2}\big(3|\phi|^2-1\big)|\eta|^2-2\sum_{j=1}^2\big(i\eta,D_{\alpha_j}\phi\big){b}_j=0\label{m704}
\end{align}
Now, adding equation \eqref{m604} to the equation \eqref{m704}, we deduce that $J[\eta,b]=0$. ( The functional $J$ was defined in \eqref{m1}.) Now, according to \eqref{m2}, we deduce that $(\eta,b)=0$. Therefore, $(\eta_1,b_1)=\eta_2,b_2)$.
\end{proof}

\begin{claim}
\label{c22}
Suppose that $(\eta,b)\in (H^1)^{\perp}$ is the minimizer of the functional $I_{_{\zeta}}$ in $(H^1)^{\perp}$. Then $(\eta,b)$ satisfies the equations
\begin{equation}
\label{t11}
L[\phi,\alpha](\eta,b)=\zeta
\end{equation}
in the weak sense. 
\end{claim}
\begin{proof}
Consider the functionals $T_{j}: H^{1}\to \mathbb{R}$ for $j=1,2$ and $T_{0}: H^{1}\to \mathbb{R}$ defined by:
\begin{align}
    &T_{_{j}}[n_0,n_1,n_2]=\int_{\mathbb{R}^2}\big( \nabla n_{_{j}}.\nabla {b}_j+\Big(|\phi|^2{b}_j+\big(2iD_{\alpha_j}\phi ,\eta\big)-\zeta_{_{j}}\Big)n_{_{j}}\label{m71}\\
    &T_{_{0}}[n_0,n_1,n_2]=\int_{\mathbb{R}^2} \sum_{j=1}^2\big(2iD_{\alpha_j}\phi,n_{0}\big){b}_j+\sum_{j=1}^2\big(D_{\alpha_j}\eta,D_{\alpha_j}n_{0}\big)+\Big(\frac{1}{2}\big(3|\phi|^2-1\big)\eta-\zeta_{0},n_{0}\Big)\label{m72}
\end{align}
Suppose that $u=\big(u_0,u_1,u_2\big)\in H^{1}$. We can write:
\begin{equation}
\label{eqm73}
u=v+w
\end{equation}
where 
\[v=\big(v_0,v_1,v_2\big)\in (H^1)^{\perp}\]
and
\[w=\big(w_0,w_1,w_2\big)\in H^{1}\cap T_{p}M_N\]
\\

Since $(\eta,b)$ is the minimizer of $I_{_{\zeta}}$ in $(H^{1})^{\perp}$, we have:
\begin{equation}
    \label{eq727}
    \frac{d}{dt}\Big |_{_{t=0}}I_{_\zeta}\big((\eta,b)+tv\big)=0
\end{equation}
On the other hand, we have:
\begin{equation}
    \label{eq7271}
    \frac{d}{dt}\Big|_{_{t=0}}I_{_\zeta}\big((\eta,b)+tv\big)=T_{_{A_1}}(v)+T_{_{A_2}}(v)+T_{_{\phi}}(v)
\end{equation}
According to \eqref{eq727} and \eqref{eq7271}, we have:
\begin{equation}
\label{eqp1}
    T_{_{A_1}}(v)+T_{_{A_2}}(v)+T_{_{\phi}}(v)=0
\end{equation}
We know $w\in H^{2}$. Using this and the fact that $(\eta,b)\in H^{1}$, we can do some integration by parts to obtain:
\begin{equation}
\label{eqp2}
\begin{split}
T_{_{1}}(w)+T_{_{2}}(w)+T_{_{0}}(w)&=\int_{\mathbb{R}^2}{L}_{_{1}}(w){b}_1+{L}_{_{2}}(w){b}_2+\big({L}_{_{0}}(w),\eta\big)
\end{split}
\end{equation}
But since $w\in T_{_{p}}M_N$, we have
\begin{equation}
    \label{eqp3}
    {L}_{_{1}}(w)={L}_{_{2}}(w)={L}_{_{0}}(w)=0
\end{equation}
Using \eqref{eqp2} and \eqref{eqp3}, we have:
\begin{equation}
    \label{eqp4}
    T_{_{0}}(w)+T_{_{1}}(w)+T_{_{2}}(w)=0
\end{equation}
Now, according to \eqref{eqm73}, \eqref{eqp1} and \eqref{eqp4}, we have:
\begin{equation}
    \label{eqp5}
    T_{_{0}}(u)+T_{_{1}}(u)+T_{_{2}}(u)=0
\end{equation}
Since $u\in H^{1}$ is arbitrary, \eqref{eqp4} and \eqref{eqp5} imply that:
\begin{equation}
    \label{new.eqp5}
    T_{_{0}}(u)=T_{_{1}}(u)=T_{_{2}}(u)=0
\end{equation}
Therefore, equation \eqref{t11} holds in the $H^{1}$-weak sense. 
\end{proof}

\begin{claim}
\label{c4}
Suppose that $(\tilde{\varphi},\tilde{a})\in (H^{1})^{\perp}$ satisfies the equations
\begin{equation}
\label{new.eq511}
L[\phi,\alpha](\tilde{\varphi},\tilde{a})=\zeta
\end{equation}
in the weak sense. Then, $(\tilde{\varphi},\tilde{a})\in{H}^{2}$, and:
\begin{equation}
    \label{eq511.new}
    ||(\tilde{\varphi},\tilde{a})||_{_{H^2}}\le C||\zeta||_{L^2}
\end{equation}
where $C$ is a constant.
\end{claim}
\begin{proof}
First, note that according to the claims \eqref{c1}, \eqref{c11} and \eqref{c22}, we can say that $(\tilde{\varphi},\tilde{a})$ is the unique minimizer of $I_{\zeta}$ in $(H^{1})^{\perp}$. Therefore, by equation \eqref{m0} in claim \eqref{c1}, we have:
\begin{equation}
    \label{eq513}
    ||\tilde{\varphi}||_{L^2}\le C||\zeta||_{L^2}
\end{equation}
for some constant $C$. Set
\begin{equation}
    \label{eqnnn1}
    h_j=2\big(i\tilde{\varphi}, D_{\alpha_j}\phi\big)+\zeta_{_{j}}
\end{equation}
Now, note that $D_{\alpha_j}\phi\in L^2$, since it has exponential decay at infinity. Therefore, according to \eqref{eq513}, we can say \begin{equation}
    \label{eq515}
||h||_{L^2}\le C ||\zeta||_{L^2}   
\end{equation} 
for some constant $C$. Now, note that according to claim \eqref{c22}, we have:
\begin{equation}
    \label{eqnnn2}
    -\Delta \tilde{a}_j+|\phi|^2\tilde{a}_j=h\hspace{.3cm}\textbf{in the weak sense}
\end{equation}
On the other hand, according to lemma \ref{al1}, we can say that there exists unique $u_j\in H^2$ such that: 
\begin{equation}
    \label{eq512}
    -\Delta u_j+|\phi|^2u_j=h_j \hspace{.3cm}\textbf{in the weak sense}
\end{equation}
Now, by \eqref{eqnnn2} and \eqref{eq512}, we can say:
\begin{equation}
    \label{eq514}
    -\Delta\big(u_j-\tilde{a}_j\big)+|\phi|^2\big(u_j-a_j\big)=0 \hspace{.3cm}\textbf{in the weak sense}
\end{equation}
in the weak sense. This implies that:
\begin{equation}
    \label{eq516}
    \int_{\mathbb{R}^2}\big|\nabla(u_j-\tilde{a}_j)\big|^2+|\phi|^2\big(u_j-\tilde{a}_j\big)^2=0
\end{equation}
Therefore, $\tilde{a}_j=u_j$. Therefore, $\tilde{a_j}\in H^2$, and according to lemma \ref{al1} and \eqref{eq515}, we have:
\begin{equation}
    \label{eq517}
    ||\tilde{a}_j||_{H^2}\le C||\zeta||_{L^2}
\end{equation}
for a constant $C$.
\\
\\
Now, note that according to equation \eqref{new.eq511}, and using the standard elliptic regularity, we can say that:
\begin{equation}
    \label{eq5171}
    \tilde{\varphi}\in H^{2}_{loc}(\mathbb{R}^2)
\end{equation}
According to equation \eqref{new.eq511} and claim \eqref{c22}, we have:
\begin{equation}
    \label{eq518}
   D_{\alpha_j}D_{\alpha_j}\tilde{\varphi}=\frac{1}{2}\big(3|\phi|^2-1\big)\tilde{\varphi}+2i(D_{\alpha_j}{\phi})\tilde{a}_j-\zeta_{_0}
\end{equation}
in the weak sense. Therefore,
\begin{equation}
    \label{eq519}
   \Big( \sum_{j=1}^2D_{\alpha_j}D_{\alpha_j}\tilde{\varphi}\Big)\in L^2(\mathbb{R}^2)
\end{equation}
We already know that $\tilde{\varphi}\in H^1$ and 
\begin{equation}
\label{new.eq519.1}
||\tilde{\varphi}||\le C ||\zeta||_{L^2}
\end{equation}
Therefore, \eqref{eq519} implies that $\varphi\in H^2$ and provides a constant $C$ such that
\begin{equation}
\label{new.eq519.2}
||\varphi||_{H^2}\le C||\zeta||_{L^2}
\end{equation}
\end{proof}
Now, according to claims \eqref{c1}, \eqref{c11}, \eqref{c22} and \eqref{c4}, we deduce that there exists unique $\psi=(\tilde{\varphi},\tilde{a})\in (H^{2})^{\perp}$ which satisfies \eqref{eq1.new}. Furthermore, \eqref{eq2.new} holds for a constant $C$.
\\
\\
Now, suppose that $\zeta\in \big(T_{_{(p)}}M_N\big)^{\perp}\cap H^2_{_{loc}}$, and the inequalities in \eqref{eqexd.1} hold for the two constants $A,\gamma$. Note that using standard elliptic regularity and the fact that $(\tilde{\varphi},\tilde{a})\in H^2_{loc}$, we deduce that $(\tilde{\varphi},\tilde{a})\in H^4_{_{loc}}$. 
According to equations \eqref{eq1.new}, and the fact that $\nabla.\alpha=0$, we have:
\begin{align}
     -\Delta\tilde{\varphi}+2i\alpha_j\partial_j\tilde{\varphi}+\Big(\frac{1}{2}\big(3|\phi|^2-1\big)+|\alpha|^2\Big)\tilde{\varphi}+2i\sum_{j=1}^2\big(D_{\alpha_j}\phi\big)\tilde{a}_j=\zeta_{0}\label{eq4}\\
     -\Delta\tilde{a_j}-2(i\tilde{\varphi},D_{\alpha_j}\phi)+\tilde{a}_j|\phi|^2=\zeta_{_{j}}\label{eq5}
\end{align}
\begin{claim}
\label{prop1}
\begin{align}
    ||\tilde{\varphi}||_{L^{\infty}}\le CA\label{eq6.new}\\
    ||\tilde{a}_j||_{L^{\infty}}\le CA\label{eq7}
\end{align}
for some constant $C$.
\end{claim}
\begin{proof}
By the Sobolev embedding theorem, we know that
\begin{equation}
     \label{eq50}
     ||\big(\tilde{\varphi},\tilde{a}\big)||_{_{L^{\infty}}}\le {C} ||\big(\tilde{\varphi},\tilde{a}\big)||_{_{H^2}}
 \end{equation}
 where $C$ is a constant. Using \eqref{eq2.new}, \eqref{eqexd.1}, and \eqref{eq50} we deduce \eqref{eq6.new} and \eqref{eq7}. 
\end{proof}
Using the fact that $\psi=(\tilde{\varphi},\tilde{a})\in H^2$ and the Sobolev embedding theorem, we deduce:
\begin{claim}
\label{c7}
\begin{align}
\lim_{r\to\infty}\max_{|x|=r}\big|\tilde{a}_j(x)\big|=0\label{eq501}\\
\lim_{r\to\infty}\max_{|x|=r}\big|\tilde{\varphi}(x)\big|=0\label{eq502}
\end{align}
\end{claim}
Now, if we rewrite \ref{eq4},\ref{eq5} as below:
\begin{align}
 -\Delta\tilde{\phi}+2i\alpha_j\partial_j\tilde{\phi}+\Big(\frac{1}{2}\big(3|\phi|^2-1\big)+|\alpha|^2\Big)\tilde{\phi}=\eta_{\phi}-2i\big(D_{\alpha_j}\phi\big)\tilde{a}_j\label{eq13.new}\\
     -\Delta\tilde{a_j}+\tilde{a}_j|\phi|^2=\eta_{_{A_j}}+2(i\tilde{\phi},D_{\alpha_j}\phi)\label{eq14}   
\end{align}
then, using the estimates in the above propositions, and the exponential decay for $\zeta$, we can say that there exists $N$ such that
\begin{align}
   \Bigg{|} -\Delta\tilde{\varphi}+2i\alpha_j\partial_j\tilde{\varphi}+\Big(\frac{1}{2}\big(3|\phi|^2-1\big)+|\alpha|^2\Big)\tilde{\varphi}\Bigg{|}\le N e^{-\beta |X|}\label{eq15.new}\\
    \Bigg{|} -\Delta\tilde{a_j}+\tilde{a}_j|\phi|^2\Bigg{|}\le Ne^{-\beta |X|} \label{eq16.new}   
\end{align}
Now, note that \ref{eq16.new} is in the form of lemma \ref{al1}. Based on lemma ???, we get the expected exponential decay estimate for $\tilde{a}_{j}$. 
\\
\\
Now, we want to prove the exponential decay of $\tilde{\varphi}$. Note that
\begin{equation}
    \label{eq21.new}
\Delta|\tilde{\varphi}|^2=2(\Delta\tilde{\varphi},\tilde{\varphi})+2|\nabla{\tilde{\varphi}}|^2
\end{equation}
According to equation \ref{eq15.new}, by taking the inner product with $\tilde{\varphi}$, we have:
\begin{equation}
\label{eq22.new}
    \Big(-\Delta\tilde{\varphi}+2i\alpha_j\partial_j\tilde{\varphi}+\big(\frac{1}{2}(3|\phi|^2-1)+|\alpha|^2\big)\tilde{\varphi}, \tilde{\varphi}\Big{)}\le N e^{-\beta|X|}|\tilde{\varphi}|
\end{equation}
This implies that:
\begin{equation}
    \label{eq23.new}
    (\Delta\tilde{\phi},\tilde{\phi}) \ge \big(\frac{1}{2}(3|\phi|^2-1)+|\alpha|^2\big)|\tilde{\varphi}|^2-\Bigg|\Big(2i\alpha_j\partial_j\tilde{\varphi}, \tilde{\varphi}\Big)\Bigg|-N e^{-\beta |X|}|\tilde{\varphi}|
\end{equation}
Now, note that 
\begin{equation}
    \label{eq24.new}
    \begin{split}
    \Bigg|\Big(2i\alpha_j\partial_j\tilde{\varphi}, \tilde{\varphi}\Big)\Bigg|&=\Bigg|\Big(2i\partial_j\tilde{\varphi}, \alpha_j\tilde{\varphi}\Big)\Bigg|\\
    &\le |\partial_j\tilde{\varphi}|^2+|\alpha_j|^2|\tilde{\varphi}|^2\\
    &=|\nabla\tilde{\varphi}|^2+|\alpha|^2|\tilde{\varphi}|^2
    \end{split}
\end{equation}
Now, suppose that $R_2$ is chosen such that and if $|x|>R_2$, then $|{\phi}(x)|>\frac{1}{3}$. Then combining \ref{eq21.new}, \ref{eq23.new}, \ref{eq24.new}, we can say that if $|x|>R_2$, we have:
\begin{equation}
   \label{eq25.new}
   \begin {split}
   \Delta|\tilde{\varphi}|^2&\ge \big(\frac{1}{2}(3|\phi|^2-1)\big)|\tilde{\varphi}|^2-N e^{-\beta |x|}|\tilde{\varphi}|\\
   &=|\tilde{\varphi}|^2+ \frac{3}{2}\big(|\phi|^2-1\big)|\tilde{\varphi}|^2-N e^{-\beta |x|}|\tilde{\varphi}|
   \end{split}
   \end{equation}
But we know that $\big(|\phi|^2-1\big)$ decays to $0$ faster then the exponential $e^{-r|x|}$ for any $0<r<1$. Using this and claim \ref{prop1} in \eqref{eq25.new}, we can say that if $|x|>R_2$, then:
\begin{equation}
    \label{eq26}
    \Delta|\tilde{\varphi}|^2\ge |\tilde{\varphi}|^2- N^{\prime} e^{-\beta |x|}
\end{equation}
for some $N^{\prime}>0$. Consider the function
\begin{equation}
\label{eqdefv}
v=|\tilde{\varphi}|^2-Ke^{-\beta|X|}
\end{equation}
where the constant $K$ is such that
\begin{enumerate}
    \item 
    \[
    v\big|_{|x|=R_2}\le 0
    \]
    \item
    \[
    K\ge \frac{N^{\prime}}{1-\beta^2}
    \]
\end{enumerate}
 Using \ref{eq26}, if $|x|>R_2$, we have:
\begin{equation}
    \label{eq27.new}
    \begin{split}
       \Delta v &\ge |\tilde{\varphi}|^2- N^{\prime} e^{-\beta |x|}-K\Delta(e^{{-\beta}|x|})\\ 
       &=|\tilde{\varphi}|^2- N^{\prime} e^{-\beta |x|}-K\Big(-\frac{\beta}{|x|}+\beta^2\Big)e^{{-\beta}|x|}\\
       &\ge |\tilde{\varphi}|^2 -N^{\prime} e^{-\beta |x|}-K\beta^2e^{{-\beta}|x|}\\
       & \ge |\tilde{\varphi}|^2 - Ke^{{-\beta}|x|}\\
       &=v
    \end{split}
\end{equation}
 
Using \eqref{eq27.new} and the fact that $|v|\to 0$ as $|x|\to \infty$, and the maximum principle, we deduce that $v\le 0$ on $|x|\ge R_2$. Therefore, if $|x|>R_2$, then
\[
|\tilde{\varphi}|^2\le K e^{-\beta|x|}
\]
This and equation \eqref{eq2.new} proves the exponential decay estimates \eqref{al2.eq1}. We will use the following standard lemma in elliptic PDEs:
\begin{lemma}
\label{n1l1}
For any $n\in \mathbb{N}$, there exists $C_n>0$ such that for any $f\in C^2(\mathbb{R}^n)$,
\begin{equation}
    \label{eq1n1l1}
    ||\nabla f||_{_{L^{\infty}(B_1)}}<C_n\big(||f||_{_{L^{\infty}(B_2)}}+||\Delta f||_{_{L^{\infty}(B_2)}}\big)
\end{equation}
\end{lemma}
According to equation \ref{eq14}, there exists $N$ and $R$ such that if $|x|>R$, then 
\begin{equation}
    \label{n1eq1}
    |\Delta\tilde{a}_j(x)|<Ne^{-\gamma|x|}\hspace{.3cm}j=1,2
\end{equation}
Therefore, according to lemma \ref{n1l1}, there exists $R$ and $N$ such that if $|x|>R$, then 
\begin{equation}
    \label{n1eq2}
    |\partial_k\tilde{a}_j(x)|<Ne^{-\gamma |x|}\hspace{.3cm}j=1,2\hspace{.3cm},\hspace{.3cm}k=1,2
\end{equation}
According to equation \ref{eq13.new}, the fact that $|\alpha_j|(x)\to 0$ as $|x|\to\infty$, and lemma \ref{n1l1}, we deduce that there exists $R$ and $N$ such that if $|x|>R$, then 
\begin{equation}
    \label{n1eq2}
    |\partial_k\tilde{{\varphi}}(x)|<N e^{-\gamma|x|}\hspace{.4cm}k=1,2
\end{equation}
This proves \eqref{al2.eq1.1}.
\\
Now, suppose that $\zeta\in\big(T_{_{p}}M_N\big)^{\perp}\cap H^2_{_{loc}}$ satisfies \eqref{geq1}. We take the $j$th partial derivative of equation \eqref{eq5}:
\\
    \begin{equation}
    \label{3eq61}
\begin{split}   
-\Delta\partial_j\tilde{a}_j-2(i\partial_j\tilde{\varphi}, D_{\alpha_j}\phi)-2(i\tilde{\varphi}, \partial_j D_{\alpha_j}\phi)+(\partial_j\tilde{a}_j)|\phi|^2+\tilde{a}_j\partial_j|\phi|^2=\partial_j(\zeta)_j
\end{split}
    \end{equation}
    But we have:
    \begin{equation}
    \label{3eq62}
    \begin{split}
        &\partial_jD_{\alpha_j}\phi=\partial_{j}^2\phi-i(\partial_ja_j)\phi-ia_j\partial_j\phi\\
        &\partial_j|\phi|^2=2(\phi,\partial_{j}\phi)
    \end{split}
    \end{equation}
    If we use equations in \eqref{3eq62} in \eqref{3eq61}, and add them up for all $j$, we get:
    \begin{equation}
        \label{3eq63}
        \begin{split}
            \nabla.(\zeta_1,\zeta_2)=&-\Delta\nabla.\tilde{a}-2(i\partial_j\tilde{\varphi}, \partial_j\phi)+2\sum_{j=1}^2a_j(\partial_j\tilde{\varphi},\phi)-2\big(i\tilde{\varphi},\Delta\phi\big)\\
            &+2\big(\tilde{\varphi},(\nabla.a)\phi+\sum_{j=1}^2a_j\partial_j\phi\big)+(\nabla.\tilde{a})|{\phi}|^2+2\sum_{j=1}^2\tilde{a}_j(\phi,\partial_j\phi)
        \end{split}
    \end{equation}
    Now, using the fact that
    \begin{equation}
        {L}_{\phi}[\Phi,A](\tilde{\phi},\tilde{a})=\zeta_{0}
    \end{equation}
   there holds:
    \begin{equation}
        \label{3eq64}
       \Big(i\phi, -\sum_{j=1}^2D_{\alpha_j}D_{\alpha_j}\tilde{\varphi}+\frac{1}{2}\big(3|\phi|^2-1\big)\tilde{\phi}+2i\sum_{j=1}^2\tilde{a}_jD_{\alpha_j}\phi\Big)=(i\phi,\zeta_{0})
    \end{equation}
    But we have:
    \begin{equation}
        \label{3eq65}
        D_{\alpha_j}D_{\alpha_j}\tilde{\varphi}=\partial_j^2\tilde{\varphi}-a_j^2\tilde{\varphi}-i(\partial_ja_j)\tilde{\varphi}-2ia_j\partial_j\tilde{\varphi}
    \end{equation}
    If we use equation \eqref{3eq65} in \eqref{3eq64}, we get:
    \begin{equation}
        \label{3eq66}
        \begin{split}
        (i\phi,\zeta_{0})=&(-i\phi, \Delta\tilde{\varphi})+{|a|}^2(i\phi,\tilde{\varphi})+\nabla.a (\phi, \tilde{\varphi})+2\sum_{j=1}^2a_j(\phi,\partial_j\tilde{\varphi})\\
        &+\frac{1}{2}\big(3|\phi|^2-1\big)(i\phi,\tilde{\varphi})+2\sum_{j=1}^2\tilde{a}_j(\phi,\partial_j\phi)
        \end{split}
    \end{equation}
    Now if we subtract \eqref{3eq66} from \eqref{3eq63}, we have:
    \begin{equation}
        \label{3eq67}
        \begin{split}
        &-\Delta\nabla.\tilde{a}+2\sum_{j=1}^2(i\partial_j{\phi}, \partial_j\tilde{\varphi})+2(i\Delta\phi,\tilde{\phi})+2\sum_{j=1}^2\big(\tilde{\varphi},(\nabla.a)\phi+a_j\partial_j\phi\big)+(\nabla.\tilde{a})|{\phi}|^2\\
        &+(i\phi, \Delta\tilde{\phi})-{|a|}^2(i\phi,\tilde{\varphi})-\nabla.a (\phi, \tilde{\varphi})-\frac{1}{2}\big(3|\phi|^2-1\big)(i\phi,\tilde{\varphi})=0
        \end{split}
    \end{equation}
    But we have:
\begin{equation}
    \label{3eq69}
    \Delta(i\phi,\tilde{\varphi})=(i\Delta\phi,\tilde{\varphi})+2\sum_{j=1}^2(i\partial_j\phi,\partial_j\tilde{\varphi})+(i\phi,\Delta\tilde{\varphi})
\end{equation}
Using \eqref{3eq69} in \eqref{3eq67}, we have:
\begin{equation}
    \label{3eq74}
    -\Delta\nabla.\tilde{a}+\Delta(i\phi,\tilde{\varphi})+(\nabla.\tilde{a})|{\phi}|^2+\Big(i\Delta\phi+(\nabla.a)\phi+2\sum_{j=1}^2a_j\partial_j\phi-i|a|^2\phi-\frac{i}{2}\big(3|\phi|^2-1\big)\phi,\tilde{\varphi}\Big)=0
\end{equation}
But we know $(\phi,a)$ is a steady state solution for the Abelian Higgs model. Therefore, it satisfies the equation:
\begin{equation}
    -\sum_{j=1}^2D_jD_j\phi+\frac{1}{2}\big(|\phi|^2-1\big)\phi=0
\end{equation}
Which can be simplified as:
\begin{equation}
\label{3eq72}
    -\Delta\phi+|a|^2\phi+i(\nabla.a)\phi+2i\sum_{j=1}^2a_j\partial_j\phi+\frac{1}{2}\big(|\phi|^2-1\big)\phi=0
\end{equation}
Now, using equation \eqref{3eq72} in  \eqref{3eq74}, we get:
\begin{equation}
    \label{3eq75}
    -\Delta\nabla.\tilde{a}+\Delta(i\phi,\tilde{\varphi})+|\phi|^2\big(\nabla.\tilde{a}-(i\phi,\tilde{\varphi})\big)=0
\end{equation}
Or:
\begin{equation}
    \label{3eq76}
    \big(-\Delta+|\phi|^2\big)\big(\nabla.\tilde{a}-(i\phi,\tilde{\varphi})\big)=0
    \end{equation}
    Since $\tilde{a}\in H^2$, $(\nabla.\tilde{a})\in H^1$. Also, since $|\phi|$ is bounded and $\tilde{\varphi}\in L^2$, $(i\phi,\tilde{\varphi})\in L^2$. We have:
    \begin{equation}
    \label{eqnn1}
    \partial_{_j}(i\phi,\tilde{\varphi})=\big(iD_{_{\alpha_j}}\phi,\tilde{\varphi}\big)+\big(i\phi,D_{_{\alpha_j}}\tilde{\varphi}\big)
    \end{equation}
    Since $\tilde{\varphi}\in H^{1}$ and $D_{\alpha_j}\phi$ has exponential decay and $|\phi|$ is bounded, then by \eqref{eqnn1}, $\big(i\phi,\tilde{\varphi}\big)\in H^1$. Therefore, $\big(\nabla.\tilde{a}-(i\phi,\tilde{\varphi})\big)\in H^1$. Now, according to \eqref{3eq76} and lemma \ref{al1},
    \[\nabla.\tilde{a}-(i\phi,\tilde{\varphi})=0\]
\end{proof}

%% file: Regularity_in_Lemma_A.tex
\begin{proposition}
\label{p2}
Suppose that $U\subset \mathbb{R}^k$ for some $k>0$ is an open set and $s:U\to M_N$ is differentiable. Suppose that 
\begin{equation}
    \label{eqln211}
    f=\big(f_{_{0}},f_{_{1}},f_{_{2}}\big):U\to \Big(L^2\big(\mathbb{R}^2,\mathbb{C}\big),L^2\big(\mathbb{R}^2,\mathbb{R}\big),L^2\big(\mathbb{R}^2,\mathbb{R}\big)\Big)
\end{equation}
is differentiable. 
and
\begin{equation}
    \label{eqln212}
    f(p)\perp T_{_{s(p)}}M_N
\end{equation}
for each $p\in U$. Suppose that
\[
v:U\to \Big(H^2\big(\mathbb{R}^2,\mathbb{C}\big),H^2\big(\mathbb{R}^2,\mathbb{R}\big),H^2\big(\mathbb{R}^2,\mathbb{R}\big)\Big)
\] 
has the property
\begin{equation}
        \label{eqln213.1}
        v(p)\in \big(T_{s(p)}M_N\big)^{\perp}
    \end{equation}
    for any $p\in U$ and
\begin{equation}
    \label{eqln213}
    \begin{split}
        &{L}_{\varphi}[\phi(s(p)),\alpha(s(p))](v(p))=f_{_{0}}(p)\\
        &{L}_{A_j}[\phi(s(p)),\alpha(s(p))](v(p))=f_{_{j}}(p)
    \end{split}
\end{equation}
for $j=1,2$ and any $p\in U$. Then $v$ is differentiable and for each $p\in U$ and $\tau\in T_p U$, we have:
\begin{equation}
    \label{eqln213.1}
    Dv(p)(\tau)=m_p(\tau)+E_p(\tau)
\end{equation}
where $m_p(\tau)$ is an element of $T_{_{s(p)}}(M_N)$ which satisfies
\begin{equation}
    \label{eqln213.2}
    \Big(m_p(\tau),n_{\mu}\big(s(p)\big)\Big)=-\Big(v(p),Dn_{\mu}\big(s(p)\big)\big(Ds(p)(\tau)\big)\Big)
\end{equation}
for each $\mu\in\{1,2,3,4\}$ and $E_p(\tau)$ is an element of 
\[
H^2\big(\mathbb{R}^2,\mathbb{C}\times \mathbb{R}^2\big)\cap\big(T_{_{s(p)}}M_N\big)^{\perp}
\] 
which satisfies
\begin{equation}
    \label{eq0l40}
    \begin{split}
    {L}_{\varphi}[\phi(s(p)),\alpha(s(p))](Dv_p(\tau))=&Df_{0}(p)(\tau)\\
    -&2i\sum_{j=1}^2\Big(D\alpha_j(s(p))\big(Ds(p)(\tau)\big)\Big)\Big(\big(D_{\alpha_j}v_{0}\big)(p)(\tau)\Big)\\
    -&2\sum_{j=1}^2\alpha_j\big(s(p)\big)\Big(D\alpha_j(s(p))\big(Ds(p)(\tau)\big)\Big)v_{0}(p)\\
    -&3\Big(\phi(p),D\phi\big(s(p)\big)\big(Ds(p)(\tau)\big)\Big)v_{0}(p)\\
    -&2i\sum_{j=1}^2v_{j}(p)D(D_{\alpha_j}\phi)\big(s(p)\big)\big(Ds(p)(\tau)\big)
    \end{split}
\end{equation}
and:
\begin{equation}
    \label{eq0.1l40}
    \begin{split}
    {L}_{A_j}[\phi(s(p)),\alpha(s(p))](Dv_p(\tau))=&Df_{j}(p)(\tau)\\
    +&2\Big(iv_{0}(p),D(D_{\alpha_j}\phi)\big(s(p)\big)\big(Ds(p)(\tau))\Big)\\
    -&v_{j}(p)D|\phi|^2\big(s(p)\big)\big(Ds(p)(\tau)\big)
    \end{split}
\end{equation}
for $j=1,2$.
\end{proposition}
\begin{proof}
\begin{claim}
\label{c1l1000}
For every $p\in U$, there exists $C,\delta>0$ such that for every $q\in U$ with $|p-q|\le \delta$ , we have
\begin{equation}
    \label{eq1c1l1000}
    ||v(.;q)-v(.;p)||_{_{H^2(\mathbb{R}^2)}} < C|p-q|
\end{equation}
\end{claim}
\begin{proof}
Suppose that $p_1,p_2\in U$.
Let 
\[
w=v(p_1)-v(p_2)
\]
Suppose that 
\[
w=w_1+w_2
\]
where $w_1\in T_{_{s(p_1)}}M_N$ and $w_2\in \big(T_{_{s(p_1)}}M_N\big)^{\perp}$.
\begin{claim}
\label{cl1001}
There exists $C_1,\delta_1>0$ such that if $|p_1-p_2|<\delta_1$, then
$||w_1||_{_{H^2}}<C_1|p_1-p_2|$.
\end{claim}
\begin{proof}
Consider $C_1,\delta_1>0$ such that if $|p_1-p_2|<\delta_1$, then for every $\mu\in\{1,2,3,4\}$, we have:
\begin{equation}
    \label{eq1cl1001}
    ||n_{_{\mu}}(p_1)-n_{_{\mu}}(p_2)||_{_{L^2}}< C_1|p_1-p_2|
\end{equation}
We have:
\begin{equation}
    \label{eq2cl1001}
    \langle v(p_1)-v(p_2),n_{_{\mu}}(p_1)\rangle=-\langle v(p_2),n_{_{\mu}}(p_1)-n_{_{\mu}}(p_2)\rangle
\end{equation}
Therefore, there exists $C_2>0$ such that if $|p_1-p_2|<\delta$, then
\begin{equation}
    \label{eq3cl1001}
    \langle v(p_1)-v(p_2),n_{_{\mu}}(p_1)\rangle < C_2|p_1-p_2|
\end{equation}
Therefore, there exists $C_3>0$ such that if $|p_1-p_2|<\delta$, then
\begin{equation}
    \label{eq4cl1001}
    ||w_1||_{_{H^2}}<C_3|p_1-p_2|
\end{equation}
\end{proof}
There holds
\begin{equation}
    \label{eq2c1l1000}
    \begin{split}
    {L}_{\phi}[\phi({s}(p_1)),\alpha({s}(p_1))](w_2)=&f_{0}(p_1)-f_{0}(p_2)\\
    -&2i\sum_{j=1}^2\big(\alpha_j(s(p_1))-\alpha_j(s(p_2))\big)\partial_j v_{0}(p_2)\\
    -&\big(|\alpha|^2(s(p_1))-|\alpha|^2(s(p_2))\big)v_{0}(p_2)\\
    -&\frac{3}{2}\big(|\varphi|^2(s(p_1))-|\varphi|^2(s(p_2))\big)v_{\varphi}(p_2)\\
    -&2i\sum_{j=1}^2\Big(\big(D_{\alpha_j}\phi\big)(s(p_1))-\big(D_{\alpha_j}\phi\big)(s(p_2))\Big)v_{j}(p_2)
    \end{split}
\end{equation}
and:
\begin{equation}
    \label{eq3c1l1000}
    \begin{split}
    {L}_{A_j}[\phi(s(p_1)),\alpha(s(p_1))](w_2)=&f_{j}(p_1)-f_{j}(p_2)\\
    +&2\big(iv_{0}(s(p_2)),D_{\alpha_j}\phi(s(p_1))-D_{\alpha_j}\phi(s(p_2))\big)\\
    -&v_{_{j}}(s(p_2))\big(|\phi|^2(s(p_1))-|\phi|^2(s(p_2))\big)
    \end{split}
\end{equation}
\begin{claim}
\label{cl101}
For every $p_1\in U$, there exists $C,\delta>0$ such that for every $p_2\in U$ with $|p_1-p_2|\le \delta$
\begin{align}
        &||\alpha_j(s(p_1))-\alpha_j(s(p_2))||_{L^{\infty}}\le C|p_1-p_2|\hspace{.5cm}j=1,2\label{eq4.1c1l1000}\\
        &|||\alpha|^2(s(p_1))-|\alpha|^2(s(p_2))||_{L^{\infty}}\le C|p_1-p_2|\label{eq4.2c1l1000}\\
        &|||\phi|^{2}(s(p_1))-|\phi|^2(s(p_2))||_{L^{\infty}}\le C|p_1-p_2|\label{eq4.3c1l1000}\\
        &||\big(D_{\alpha_j}\phi\big)(s(p_1))-\big(D_{\alpha_j}\phi\big)(s(p_2))||_{L^{\infty}}\le C|p_1-p_2|\label{eq4.4c1l1000}
\end{align}
\end{claim}
\begin{proof}
This follows from the estimates in proposition \ref{p30.1}.
\end{proof}
\begin{claim}
\label{cl103.1}
For every $p_1\in U$, there there exists $C,\delta>0$ such that for every $p_2\in U$ with $|p_1-p_2|< \delta$,
\begin{equation}
    \label{eq1cl103}
    ||v(.;p_2)||_{_{H^2(\mathbb{R}^2)}}< C
\end{equation}
\end{claim}
\begin{proof}
This simply follows from lemma \ref{al2}.
\end{proof}
According to claims \ref{cl101}, \ref{cl103.1}, equations \ref{eq2c1l1000}, \ref{eq3c1l1000}, and lemma \ref{al2}, there exists $\delta_1, C_1>0$ such that if $|p_1-p_2|<\delta_1$, then 
\[
||w_2||_{_{H^2}}<C_1|p_1-p_2|
\]
Therefore, according to claim \ref{cl1001}, claim \ref{c1l1000} follows.
\end{proof}
\begin{claim}
\label{cl201}
Suppose that $p\in U$, $\tau\in T_{p}U$, and $\zeta=\big(Ds(p)\big)(\tau)$. There exists a function
\begin{equation}
    \label{eq1cl201}
    E_p(\tau)\in H^{2}\cap \big(T_{s(p)}M_N\big)^{\perp}
\end{equation}
such that 
\begin{equation}
    \label{eq2cl201}
    \begin{split}
        {L}_{\varphi}[\phi(s(p)),\alpha(s(p))](E_p(\tau))=&\big(Df_{0}(p)\big)(\tau)-2i\sum_{j=1}^2\big((D\alpha_j)(s(p))(\zeta)\big)\big(D_{\alpha_j}v_{\varphi}\big)\\
        -&3\big(\phi, D\phi(s(p))(\zeta)\big)v_{0}-2\sum_{j=1}^2v_{j}\big(D(D_{\alpha_j}\phi)\big)(s(p))(\zeta)
    \end{split}
\end{equation}
and
\begin{equation}
    \label{eq3cl201}
    \begin{split}
        {L}_{A_j}[\phi(s(p)),\alpha(s(p))](E_p(\tau))=&\big(D f_{j}(p)\big)(\tau)+2\Big(iv_{0},(D(D_{\alpha_j}\phi)\big)(s(p))(\zeta)\Big)\\
        -&v_{j}D\big(|\phi|^2\big)(s(p))(\zeta)
    \end{split}
\end{equation}
\end{claim}
\begin{proof}
Let 
\begin{equation}
    \label{eq4cl201}
    \begin{split}
    R_{0}=&\big(Df_{0}(p)\big)(\tau)-2i\sum_{j=1}^2\big((D\alpha_j)(s(p))(\zeta)\big)\big(D_{\alpha_j}v_{0}\big)\\
        -&3\big(\phi, (D\phi(s(p))(\tau)\big)v_{0}-2i\sum_{j=1}^2v_{j}\big(D(D_{\alpha_j}\phi)\big)(s(p))(\zeta)
    \end{split}
\end{equation}
and
\begin{equation}
    \label{eq5cl201}
    \begin{split}
        R_{j}=&\big(D f_{j}(p)\big)(\tau)+2\Big(iv_{0},(D(D_{\alpha_j}\phi)\big)(s(p))(\zeta)\Big)\\
        -&v_{j}D\big(|\phi|^2\big)(s(p))(\zeta)
    \end{split}
\end{equation}
for $j=1,2$.
\begin{claim}
\label{cl204}
$R=(R_0,R_1,R_2)\in L^2(\mathbb{R}^2)$ and 
\begin{equation}
    \label{eq1cl204}
    R\perp T_{_{s(p)}}M_N
\end{equation}
\end{claim}
\begin{proof}
According to estimates \ref{p30.1}, the facts that $v\in H^2$ and $\alpha(s(p))\in L^{\infty}$, and the fact that $f$ is differentiable, we deduce that $R\in L^2$.
\\
\\
We have:
\begin{equation}
    \label{eq2cl204}
    f\perp T_{_{s(p)}}M_N
\end{equation}
Therefore, 
\begin{equation}
    \label{eq4cl204}
    \big{\langle} Df(p)(\tau),n_{\mu}(s(p))\big{\rangle}_{L^2}+\big{\langle}(f(p),\big(Dn_{\mu}\big)(s(p))(\zeta)\big{\rangle}_{L^2}=0
\end{equation}
We have:
\begin{equation}
    \label{eq5cl204}
    \begin{split}
    f_{0}(p)&=-\Delta v_{0}(p)+2i\sum_{j=1}^2\alpha_j\partial_jv_{0}+|\alpha|^2v_{0}\\
    &+\frac{1}{2}\big(3|\phi|^2-1\big)v_{0}+2i\sum_{j=1}^2v_{j}D_{\alpha_j}\phi
    \end{split}
\end{equation}
and
\begin{equation}
    \label{eq6cl204}
    \begin{split}
        f_{j}(p)=-\Delta v_{j}(p)-2\big(iv_{0},D_{\alpha_j}\phi(s(p))\big)+v_{j}(p)|\phi|^2(s(p))
    \end{split}
\end{equation}
According to \eqref{eq4cl204}, \eqref{eq5cl204}, and \eqref{eq6cl204}, we have
\begin{equation}
    \label{eq7cl204}
    \begin{split}
        \Big(Df(p)(\tau),n_{\mu}(s(p))\Big)_{L^2}&=+\Bigg(v_{0}(p),\Big(D(\Delta n_{\mu,\varphi})(s(p))\Big)(\zeta)\Bigg)_{L^2}\\
        &-2\sum_{j=1}^2\Big(v_{0},i\alpha_j D\big(\partial_j n_{_{\mu,\varphi}}\big)(s(p))(\zeta)\Big)_{L^2}\\
        &-\Big(v_{0},|\alpha|^2(s(p))\big(Dn_{\mu,\varphi}\big)(s(p))(\zeta)\Big)_{L^2}\\
        &-\Big(v_{0},\frac{1}{2}\big(3|\phi|^2-1\big)\big(Dn_{\mu,0}\big)(s(p))(\zeta)\Big)_{L^2}\\
        &-2\sum_{j=1}^2\Big(v_{j},\Big(i\big(D_{\alpha_j}\phi(s(p)),\big(Dn_{\mu,\varphi}\big)(s(p))(\zeta)\big)\Big)\Big)_{_{L^2}}\\
        &+2\sum_{j=1}^2\Big(v_{j}(p), \big(D(\Delta n_{\mu,j})\big)(s(p))(\zeta)\Big)_{L^2}\\
    &-2\sum_{j=1}^2\Big(v_{0},iD_{\alpha_j}\phi(s(p))\big(Dn_{\mu,A_j}\big)(s(p))(\zeta) \Big)_{L^2}\\
        &-\Big(v_{j}(p),|\phi|^2(s(p))\big(Dn_{\mu,j}\big)(s(p))(\zeta)\Big)_{L^2}
        \end{split}
\end{equation}
\begin{equation}
    \label{eq8cl204}
    \begin{split}
        \Big(2i\big(D\alpha_j\big)(s(p))(\zeta))(D_{\alpha_j}v_{\varphi}), n_{_{\mu,\varphi}}\Big)_{L^2}&=\Big(2i\big(D\alpha_j\big)(s(p))(\zeta))\partial_j v_{0},n_{\mu,0}\Big)_{L^2}\\
        &+\Big(\big(D|\alpha|^2\big)(s(p))(\zeta))v_{0},n_{\mu,\varphi}\Big)_{L^2}
    \end{split}
\end{equation}
According to the fact that $v_{0}\in H^2(\mathbb{R}^2)$
\begin{equation}
    \label{eq9cl204}
    \big(D\alpha_j\big)(s(p))(\zeta)) v_{0}\in  H^2(\mathbb{R}^2)
\end{equation}
Therefore,
\begin{equation}
    \label{eq10cl204}
    \Big(2i\big(D\alpha_j\big)(s(p))(\zeta))\partial_j v_{0},n_{\mu,0}\Big)_{L^2}=-\Big(2i\big(D\alpha_j\big)(s(p))(\zeta))v_{0},\partial_jn_{\mu,0}\Big)_{L^2}
\end{equation}
According to \eqref{eq4cl204},\eqref{eq5cl204}, \eqref{eq7cl204}, \eqref{eq8cl204}, and \eqref{eq10cl204}, we have:
\begin{equation}
    \label{eq11cl204}
    \begin{split}
        (R,n_{\mu})_{L^2}&=-\Big(v_{0}, \big(D{L}_{\varphi}[\phi(s(p)),\alpha(s(p))](n_{\mu})\big)(\zeta)\Big)_{_{L^2}}\\
        &-\sum_{j=1}^2\Big(v_{j},\big(D{L}_{A_j}[\phi(s(p)),\alpha(s(p))](n_{\mu})\big)(\zeta)\Big)_{_{L^2}}\\
        &=0
    \end{split}
\end{equation}
\end{proof}
Therefore, according to Lemma \ref{al2}, the statement holds.
\end{proof}
Suppose that $p\in U$ and $\tau\in T_p U$ and $\zeta=Ds(p)(\tau)$. Suppose that the function $E_p(\tau):\mathbb{R}^2\to\mathbb{C}\times \mathbb{R}^2$ is the function which satisfies \eqref{eq1cl201}, \eqref{eq2cl201}, and \eqref{eq3cl201}. Let $m_p(\tau)\in T_{s(p)}M_N$ be such that 
\begin{equation}
    \label{eq6l40}
    \big(m_p(\tau),n_{\mu}(s(p))\big)_{_{L^2}}=-\Big(v(p),\big(Dn_{\mu}(s(p))(\zeta)\big)\Big)_{_{L^2}}
\end{equation}
for any $\mu\in\{1,2,3,4\}$. Let:
\begin{equation}
    \label{eq7l40}
    S_p(\tau)=E_p(\tau)+m_p(\tau)
\end{equation}
Let
\begin{equation}
    \label{eq8l40}
    w_p(\tau)=v(p+\tau)-v(p)-S_p(\tau)
\end{equation}
\begin{claim}
\label{cl60l40}
For any $p\in U, \epsilon >0$, there exists $\delta>0$ such that if $|\tau|<\delta$, then $|w_p(\tau)|_{H^2}<\epsilon |\tau|$.
\end{claim}
\begin{proof}
Suppose that $A_{\mu}(\tau)=\langle w_p(\tau),n_{\mu}(s(p))\rangle_{_{L^2}}$.
\begin{claim}
\label{cl60.5l40}
For any $\epsilon >0$, there exists $\delta>0$ such that if $|\tau|<\delta$, then $|A_{\mu}(\tau)|<\epsilon|\tau|$.
\end{claim}
\begin{proof}
\begin{equation}
    \label{eq1l60.5l40}
    \begin{split}
      A_{\mu}(\tau)=&-\Big(v(p), n_{\mu}\big(s(p+\tau)\big)-n_{\mu}\big(s(p)\big)-Dn_{\mu}\big(s(p)\big)(\zeta)\Big)_{L^2}\\
      &+\Big(v(p)-v(p+\tau),n_{\mu}(s(p+\tau))-n_{\mu}(s(p))\Big)_{L^2}
    \end{split}
\end{equation}
Therefore,
\begin{equation}
    \label{eq2l60.5l40}
    \begin{split}
      |A_{\mu}(\tau)|\le &||v(p)||_{_{L^2}} \Big{|}\Big{|}n_{\mu}\big(s(p+\tau)\big)-n_{\mu}\big(s(p)\big)-Dn_{\mu}\big(s(p)\big)(\zeta)\Big{|}\Big{|}_{_{L^2}}\\
      &+\big{|}\big{|}v(p)-v(p+\tau)\big{|}\big{|}_{_{L^2}}\big{|}\big{|}n_{\mu}(s(p+\tau))-n_{\mu}(s(p))\big{|}\big{|}_{_{L^2}}
    \end{split}
\end{equation}
According to claim \ref{c1l1000}, the fact that $s$ is differentiable, and \eqref{eq2l60.5l40}, the statement follows.
\end{proof}
Suppose that $\hat{w}_p(\tau)$ is the projection of $w_{p}(\tau)$ on $\big(T_{s(p)}M_N\big)^{\perp}$, with respect the $L^2$ inner product. There holds:
\begin{equation}
    \label{eq1cl60l40}
    \begin{split}
       &{L}_{\phi}[\phi(s(p)),\alpha(s(p))](\hat{w}_p(\tau))=F_{_{0,p}}(\tau)\\
       &{L}_{A_j}[\phi(s(p)),\alpha(s(p))](\hat{w}_p(\tau))=F_{_{j,p}}(\tau)
    \end{split}
\end{equation}
for $j=1,2$ where:
\begin{equation}
    \label{eq2cl60l40}
    \begin{split}
    F_{_{0,p}}(\tau)=&f_{0}(p+\tau)-f_{0}(p)-Df_{0}(p)(\tau)\\
    -&2i\sum_{j=1}^2\big(\alpha_j(s(p+\tau))-\alpha_j(s(p))\big)\big(\partial_j v_{0}(p+\tau)-\partial_jv_{0}(p)\big)\\
    -&2i\sum_{j=1}^2\Big(\alpha_j\big(s(p+\tau)\big)-\alpha_j\big(s(p)\big)-\big(D\alpha_j\big)\big(s(p)\big)(\zeta)\Big)\big(\partial_jv_{0}(p)\big)\\
    -&\big(|\alpha|^2(s(p+\tau))-|\alpha|^2(s(p))\big)\big(v_{0}(p+\tau)-v_{0}(p)\big)\\
    -&\Big(|\alpha|^2\big(s(p+\tau)\big)-|\alpha|^2\big(s(p)\big)-\big(D|\alpha|^2\big)\big(s(p)\big)(\zeta)\Big)v_{_{0}}(p)\\
    -&\frac{3}{2}\big(|\phi|^2(s(p+\tau))-|\phi|^2(s(p))\big)\big(v_{0}(p+\tau)-v_{0}(p)\big)\\
    -&\frac{3}{2}\Big(|\phi|^2(s(p+\tau))-|\phi|^2(s(p))-\big(D|\phi|^2\big)\big(s(p)\big)(\zeta)\Big)v_{_{\phi}}(p)\\
    -&2i\sum_{j=1}^2\Big(\big(D_{\alpha_j}\phi\big)(s(p+\tau))-\big(D_{\alpha_j}\phi\big)(s(p))\Big)\big(v_{j}(p+\tau)-v_{j}(p)\big)\\
    -&2i\sum_{j=1}^2\Big(\big(D_{\alpha_j}\phi\big)(s(p+\tau))-\big(D_{\alpha_j}\phi\big)(s(p))-D\big(D_{\alpha_j}\phi\big)\big(s(p)\big)(\zeta)\Big)v_{a_j}(p)
    \end{split}
\end{equation}
and:
\begin{equation}
    \label{eq3cl60l40}
    \begin{split}
    F_{_{j,p}}(\tau)=&f_{j}(p+\tau)-f_{j}(p)-\big(D f_{j}(p)\big)(\tau)\\
    +&2\big(iv_{0}(p+\tau)-iv_{0}(p),D_{\alpha_j}\phi(s(p+\tau))-D_{\alpha_j}\phi(s(p))\big)\\
    +&2\Big(iv_{_{0}}(p),\big(D_{_{\alpha_j}}\phi\big)\big(s(p+\tau)\big)-\big(D_{\alpha_j}\phi\big)\big(s(p)\big)-D\big(D_{\alpha_j}\phi\big)\big(s(p)\big)(\zeta)\Big)\\
    -&\Big(v_{_{j}}(p+\tau)-v_{_{j}}(p)\Big)\big(|\phi|^2(s(p+\tau))-|\phi|^2(s(p))\big)\\
    -&v_{_{j}}(p)\Big(|\phi|^2\big(s(p+\tau)\big)-|\phi|^2\big(s(p)\big)-\big(D|\varphi|^2\big)\big(s(p)\big)(\zeta)\Big)
    \end{split}
\end{equation}
According to estimates \ref{p30.1}, the fact that $\alpha\in L^{\infty}$ and  differentiability of $s$, for every $\epsilon>0$, there exists $\delta>0$ such that if $|\tau|<\delta$, then
\begin{equation}
    \label{eq4cl60l40}
    \begin{split}
    &\Big{|}\Big{|}\alpha_j\big(s(p+\tau)\big)-\alpha_j\big(s(p)\big)-\big(D\alpha_j\big)\big(s(p)\big)(\zeta)\Big)\Big{|}\Big{|}_{L^{\infty}}<\epsilon|\tau|\\
    &\Big{|}\Big{|}|\alpha|^2\big(s(p+\tau)\big)-|\alpha|^2\big(s(p)\big)-\big(D|\alpha|^2\big)\big(s(p)\big)(\zeta)\Big{|}\Big{|}_{L^{\infty}}<\epsilon|\tau|\\
    &\Big{|}\Big{|}|\phi|^2(s(p+\tau))-|\phi|^2(s(p))-\big(D|\phi|^2\big)\big(s(p)\big)(\zeta)\Big{|}\Big{|}_{L^{\infty}}<\epsilon|\tau|\\
    &\Big{|}\Big{|}\big(D_{\alpha_j}\phi\big)(s(p+\tau))-\big(D_{\alpha_j}\phi\big)(s(p))-D\big(D_{\alpha_j}\phi\big)\big(s(p)\big)(\zeta)\Big{|}\Big{|}_{L^{\infty}}<\epsilon|\tau|
    \end{split}
\end{equation}
Therefore, according to claims \ref{c1l1000} and \ref{cl101}, equations \eqref{eq2cl60l40}, \eqref{eq3cl60l40}, and differentiability of $f$ for every $\epsilon>0$, there exists $\delta>0$ such that if $|\tau|<\delta$, then
\begin{equation}
    \label{eq5l40l60}
    \big{|}\big{|}F_{0,p}(\tau)\big{|}\big{|}_{L^2},\big{|}\big{|}F_{_{1,p}}(\tau)\big{|}\big{|}_{L^2},\big{|}\big{|}F_{_{2,p}}(\tau)\big{|}\big{|}_{L^2}<\epsilon |\tau|
\end{equation}
Since $\hat{w}_p(\tau)\in H^{2}$, according to \eqref{eq1cl60l40},
\begin{equation}
    \label{eq6l40l60}
    \begin{pmatrix}
    &F_{0,p}(\tau)\\
    &F_{1,p}(\tau)\\
    &F_{2,p}(\tau)
    \end{pmatrix}
    \perp T_{s(p)}M_N
\end{equation}
Therefore, according to lemma \ref{al2}, for every $\epsilon>0$, there exists $\delta>0$ such that if $|\tau|<\delta$, then $||\hat{w}_p(\tau)||_{H^2}<\epsilon|\tau|$. Therefore, according to claim \ref{cl60.5l40}, for every $\epsilon>0$, there exists $\delta>0$ such that if $|\tau|<\delta$, then $|w_p(\tau)|_{H^2}<\epsilon|\tau|$. 
\end{proof}
Claim \ref{cl60l40} implies proposition \ref{p2}.
\end{proof}
\begin{lemma}
\label{l40.4}
Suppose that 
\[
\eta=\Big(\eta_{0},\eta_{1},\eta_{2}\Big)\in \Big(H^4(\mathbb{R}^2,\mathbb{C}), H^4(\mathbb{R}^2,\mathbb{R}), H^4(\mathbb{R}^2,\mathbb{R})\Big)
\]
Suppose that 
\[
p=(\phi,\alpha)\in K\subset M_N
\]
where $K$ is compact in $M_N$. Suppose that
\[
\eta\perp T_{p}M_N
\]
and
\begin{equation}
    \label{eqln211.9}
    \begin{split}
        &{L}_{\varphi}[\tilde{\varphi},\tilde{\alpha}]=\eta_{_{0}}\\
        &{L}_{A_j}[\tilde{\varphi},\tilde{\alpha}]=\eta_{_{j}}
    \end{split}
\end{equation}
for $j=1,2$. Suppose that $|\eta|_{H^1}< M$ and
\[
\eta(x)<Ae^{-\gamma|x|}
\]
Then, there exist a constant
\[
B=B(K,A,M)
\]
such that
\[
|u(x)|<Be^{-\frac{\gamma}{2}|x|}
\]
and
\[
|Du(x)|<Be^{-\frac{\gamma}{2}|x|}
\]
\end{lemma}
\begin{proof}
According to lemma \ref{al2}, there exists $R=R(K)$ and $A^{\prime}=A^{\prime}(K,A)$ such that if $|x|>R$, then
\begin{equation}
\label{new.new.eq1}
|u(x)|< A^{\prime}e^{-\frac{\gamma}{2}|x|}
\end{equation}
Suppose that
\begin{equation}
    \label{eq1.9l40.1}
    u=
    \begin{pmatrix}
    u_{0}\\
    u_{_{1}}\\
    u_{_{2}}
    \end{pmatrix}
\end{equation}
We define the functions
\begin{equation}
    \label{eq2140.1}
    \begin{split}
    &s:\mathbb{R}^2\to M_N\\
        &\eta_{0}:{\mathbb{R}^2}\to L^2\big(\mathbb{R}^2,\mathbb{C}\big)\\
        &\eta_{_{j}}:\mathbb{R}^2\to L^2\big(\mathbb{R}^2,\mathbb{R}\big)
        \hspace{.5cm}{j=1,2}\\
        &v_{0}: \mathbb{R}^2\to H^2\big(\mathbb{R}^2,\mathbb{C}\big)\\
        &v_{_{j}}: \mathbb{R}^2\to H^2\big(\mathbb{R}^2,\mathbb{R}\big)\hspace{.5cm}{j=1,2}
    \end{split}
\end{equation}
by
\begin{equation}
    \label{eq3l40.1}
    \begin{split}
    &s(y)=(p-y)\\
        &\eta_{_{k}}(y)(x)=\eta_{A_j}(x-y)\hspace{.2cm}k=0,1,2\\
        &v_{_{y}}(y)(x)=u_{k}(x-y)\hspace{.2cm}k=0,1,2
    \end{split}
\end{equation}
where $(p-y)$ refers to translation of the vortex centers by the vector $y$ in the standard coordinate system in $\mathbb{R}^2$. 
Then, for any $z\in \mathbb{R}^2$, we have:
\begin{equation}
    \label{eq3.1l40.1}
    \begin{split}
        &{L}_{\varphi}[\phi(s(z),\alpha(s(z)](v(z))=\eta_{_{0}}(z)\\
        &{L}_{A_j}[\phi(s(z)),\alpha(s(z)](v(z))=\eta_{_{j}}(z)\hspace{.2cm}j=1,2
    \end{split}
\end{equation}
Since $\eta\in H^4$, the function $\eta:\mathbb{R}^2\to L^2\big(\mathbb{R}^2,\mathbb{R}\big)$ is differetiable. Therefore, by proposition \ref{p2}, the function $v: \mathbb{R}^2\to H^2\big(\mathbb{R}^2,\mathbb{C}\times\mathbb{R}^2\big)$ is differentiable. Therefore, by the Sobolev embedding, the function $u:\mathbb{R}^2\to \mathbb{C}\times \mathbb{R}^2$ is differentiable. According to proposition \ref{p2} and the Sobolev embedding, there is a constant $C_1={C}_1(K,A)$ such that
\[
|\partial_j u(x)|\le C||Dv||_
{H^2}\le C_1 
\]
for any $x$. Lemma \ref{al2} and this imply the statement. 
\end{proof}

\begin{lemma}
\label{l40.5}
Suppose that $U\subset \mathbb{R}^k$ for some $k>0$ is an open set and $s:U\to M_N$ is $n$-times differentiable. Suppose that 

\begin{equation}
    \label{eqln211}
   f= \big(f_{_{0}},f_{_{1}},f_{_{2}}\big):U\to \Big(L^2\big(\mathbb{R}^2,\mathbb{C}\big),L^2\big(\mathbb{R}^2,\mathbb{R}\big),L^2\big(\mathbb{R}^2,\mathbb{R}\big)\Big)
\end{equation}
is $m$-times differentiable. Assume 
\begin{equation}
    \label{eqln212}
    f(p)\perp T_{_{s(p)}}M_N
\end{equation}
for each $p\in U$. Suppose that
\[
v:U\to \Big(H^2\big(\mathbb{R}^2,\mathbb{C}\big),H^2\big(\mathbb{R}^2,\mathbb{R}\big),H^2\big(\mathbb{R}^2,\mathbb{R}\big)\Big)
\] 
has the property
\begin{equation}
        \label{eqln213.1}
        v(p)\in \big(T_{s(p)}M_N\big)^{\perp}
    \end{equation}
    for any $p\in U$ and
\begin{equation}
    \label{eqln213}
    L[\phi(s(p)),\alpha(s(p))]v(p)=f(p)
\end{equation}
for $j=1,2$ and any $p\in U$. Then, $v$ is $m$-times differentiable, provided $n\ge m$.
\end{lemma}
\begin{proof}
We use induction on $m$. The base case holds by lemma \ref{l40.5}. Suppose that the lemma holds when $m=l$. Let $m=l+1$. 
Suppose that $p\in U$. According to proposition \ref{p2},
\begin{equation}
    \label{eqln213.1}
    \partial_k v=m+E
\end{equation}
where $m$ is an element of $T_{_{s(p)}}M_N$ which satisfies
\begin{equation}
    \label{eqln213.2}
    \Big(m,n_{\mu}\big(s(p)\big)\Big)=-\Big(v(p),\Big(\partial_k \big(n_{\mu}\big(s(p)\big)\big)\Big)
\end{equation}
for each $\mu\in\{1,2,3,4\}$ and $E=\big(E_{0}, E_{_{1}}, E_{_{2}}\big)$ is an element of 
\[
H^2\big(\mathbb{R}^2,\mathbb{C}\times \mathbb{R}^2\big)\cap\big(T_{_{s(p)}}M_N\big)^{\perp}
\] 
which satisfies
\begin{equation}
    \label{eq0l40}
    \begin{split}
    {L}_{\varphi}[\phi(s(p)),\alpha(s(p))](E_{0})=&\partial_kf_{0}(p)\\
    -&\sum_{j=1}^22i\Big(\partial_k \big(\alpha_j(s(p))\big)\Big)\Big(\big(D_{\alpha_j}v_{0}\big)(p)\Big)\\
    -&2\sum_{j=1}^2\alpha_j\big(s(p)\big)\Big(\partial_k\big(\alpha_j(s(p))\big)\Big)v_{0}(p)\\
    -&3\Big(\varphi(p),\partial_k\big(\varphi\big(s(p)\big)\big)\Big)v_{0}(p)\\
    -&2i\sum_{j=1}^2v_{j}(p)\partial_k\big((D_{\alpha_j}\phi)\big(s(p)\big)\big)
    \end{split}
\end{equation}
and:
\begin{equation}
    \label{eq0.1l40}
    \begin{split}
    {L}_{A_j}[\phi(s(p)),\alpha(s(p))](E_{_{j}})=&\partial_kf_{j}(p)\\
    +&2\Big(iv_{\phi}(p),\partial_k\big((D_{\alpha_j}\phi)\big(s(p)\big)\big)\Big)\\
    -&v_{j}(p)\partial_k\big(|\phi|^2\big(s(p)\big)\big)
    \end{split}
\end{equation}
for $j=1,2$. 
According to the estimates \ref{p30.1}, the functions
\[
\partial_k\big(\alpha_1\circ s\big), \partial_k\big(\alpha_2\circ s\big), \partial_k\big(\phi\circ s\big),\partial_k\big( |\phi|^2\circ s\big), \partial_k\big(\big(D_{_{\alpha_j}}\phi\big)\circ s\big):U\to L^{\infty}\big(\mathbb{R}^2,\mathbb{R}\big)
\]
are $(n-1)$-times differentiable. On the other hand, by the induction hypothesis, the functions
\[
v_{0}, v_{_{1}}, v_{_{2}}, \partial_{1}v_{_{0}}, \partial_{2}v_{_{0}},\partial_k f_{_{0}},\partial_{k}f_{_1},\partial_k f_{_{2}}:U\to L^{2}\big(\mathbb{R}^2,\mathbb{R}\big)
\]
are $l$-times differentiable. Therefore, according to equations \eqref{eq0l40} and \eqref{eq0.1l40}, we deduce that the functions
\begin{align}
    &{L}_{\varphi}[\phi(s(p)),\alpha(s(p))](E):U\to L^{2}\big(\mathbb{R}^2,\mathbb{R}\big)\\
    &{L}_{A_j}[\phi(s(p)),\alpha(s(p))](E):U\to L^{2}\big(\mathbb{R}^2,\mathbb{R}\big) 
\end{align}
for $j=1,2$, are $l$-times differentiable. Therefore, according to lemma \ref{l40.5}, the function $E:U\to H^{2}\big(\mathbb{R}^2, \mathbb{R}\big)$ is $l$-times differentiable. Since \eqref{eqln213.2} holds for any $\mu\in\{1,2,3,4\}$, by the induction hypothesis, the function $m:U\to H^2\big(\mathbb{R}^2,\mathbb{R}\big)$ is $l$-times differentiable. Therefore, the function $\partial_k v:U\to H^2\big(\mathbb{R}^2, \mathbb{R}\big)$ is $l$-times differentiable, for any $k$. Therefore, the function $v$ is $(l+1)$-times differentiable.

\end{proof}
\begin{lemma}
\label{ln4}
Suppose that $m\ge 4$, $U\subset \mathbb{R}^k$ is an open set and 
\begin{equation}
    \label{eqln211}
    f=\begin{pmatrix}
    f_{0}\\
    f_{1}\\
    f_{2}
    \end{pmatrix}
    \in 
    \mathcal{E}_m(\mathbb{R}^2,U,\mathbb{C}\times \mathbb{R}^2)
\end{equation}
for each $p\in U$. Suppose that $q:U\to M_N$ is $n$-times differentiable and $q(U)$ is a precompact subset of $M_N$ and $D^s(q)$ is bounded for every multi-index $s$ with $|s|\le n$ and
\begin{equation}
    \label{eqln212}
    f(.;p)\perp T_{_{q(p)}}M_N
\end{equation} 
and:
\begin{equation}
    \label{eqln213}
    \begin{split}
        &{L}_{\varphi}[\phi(q(p))),\alpha(q(p))](u(.;p))=f_{_{0}}(.;p)\\
        &{L}_{A_j}[\phi(q(p)),\alpha(q(p))](u(.;p))=f_{_{j}}(.;p)
    \end{split}
\end{equation}
for $j=1,2$, where
    \begin{equation}
        \label{eqln213.1}
        u(.;p)\in H^{2}\cap \big(T_{q(p)}M_N\big)^{\perp}
    \end{equation}
    for each $p\in U$. Then,
\begin{equation}
    \label{eqln214}
    u\in \mathcal{E}_{m-4}(\mathbb{R}^2,U,\mathbb{C}\times \mathbb{R}^2)
\end{equation}
if $n$ is large enough.
\end{lemma}
\begin{proof}
Suppose that $V=\mathbb{R}^2\times U$ and 
\begin{equation}
    \label{eq1.9l40}
    u=
    \begin{pmatrix}
    u_{0}\\
    u_{_{1}}\\
    u_{_{2}}
    \end{pmatrix}
\end{equation}
We define the functions
\begin{equation}
    \label{eq2140}
    \begin{split}
    &\hat{q}:V\to M_N\\
        &g_{0}:{V}\to L^2\big(\mathbb{R}^2,\mathbb{C}\big)\\
        &g_{_{j}}:V\to L^2\big(\mathbb{R}^2,\mathbb{R}\big)
        \hspace{.5cm}{j=1,2}\\
        &v_{0}: V\to H^2\big(\mathbb{R}^2,\mathbb{C}\big)\\
        &v_{_{j}}: V\to H^2\big(\mathbb{R}^2,\mathbb{R}\big)\hspace{.5cm}{j=1,2}
    \end{split}
\end{equation}
by
\begin{equation}
    \label{eq3l40}
    \begin{split}
    &\hat{q}(y,p)(x)=q(p)(x-y)\\
        &g_{_{k}}(y,p)(x)=f_{j}(x-y,p)\hspace{.2cm}k=0.1,2\\
        &v_{_{k}}(y,p)(x)=u_{k}(x-y,p)\hspace{.2cm}k=0,1,2
    \end{split}
\end{equation}
for any $x,y\in \mathbb{R}^2$  and any $p\in U$.
Then, for any $M\in V$, we have:
\begin{equation}
    \label{eq3.1l40}
    \begin{split}
        &{L}_{\varphi}[\varphi(\hat{q}(M),\alpha(\hat{q}(M)](v(M))=g_{_{\varphi}}(M)\\
        &{L}_{A_j}[\varphi(\hat{q}(M)),\alpha(\hat{q}(M)](v(M))=g_{_{A_j}}(M)\hspace{.2cm}j=1,2
    \end{split}
\end{equation}
Since $f\in \mathcal{E}_m(\mathbb{R}^2,U,\mathbb{C}\times \mathbb{R}^2)$, then the function $g:V\to L^2\big(\mathbb{R}^2,\mathbb{R}\big)$ is $m$-times differentiable. Therefore, by lemma \ref{l40.5}, the function $v: V\to H^2\big(\mathbb{R}^2,\mathbb{R}\big)$ is $m$-times differentiable. Therefore, the function $u:U\to H^2(\mathbb{R}^2,\mathbb{C}\times \mathbb{R}^2)$ is $m$-time differentiable.
\begin{claim}
\label{cl1n1l0.9}
There exists a number $a>0$ such that for any multi-index $r$ with $|r|<(m-2)$
\[
||D^ru(x,p)||_{H^2}<a
\]
\end{claim}
\begin{proof}
We proceed by induction on $|r|$. If $|r|=0$, the claim follows by lemma \ref{al2}, the fact that $s(U)$ is a precompact subset of $M_N$ and $f\in \mathcal{E}_m(\mathbb{R}^2,U,\mathbb{C}\times \mathbb{R}^2)$. Suppose that the statement holds for $|r|\le l$ where $l<(m-3)$. Suppose that $r$ is a multi-index with $|r|=l$. Then, the statement follows by the facts that $q(U)$ is a precompact subset of $M_N$ and $D^{s}(q)$ is bounded for every multi-index $s$ with $|s|\le (m-2)$, the induction hypothesis and lemma \ref{l40.4}.
\end{proof}
\begin{claim}
\label{cl1n1l1}
For any multi-index $r$ with $|r|<(m-3)$, there exists $A,\gamma>0$ such that
\[
D^ru(x,p)< Ae^{-\gamma|x|}
\]
and
\[
\partial_j D^ru(x,p)<Ae^{-\gamma|x|}
\]

for $j=1,2$.
\end{claim}
\begin{proof}
We use induction on $|r|$. Suppose that $|r|=0$. Then, since $f\in \mathcal{E}_m(\mathbb{R}^2,U,\mathbb{C}\times \mathbb{R}^2)$, according to lemma \ref{al2}, lemma \ref{l40.4} and claim \ref{cl1n1l0.9}, there exists $A_1,\gamma_1>0$ such that
\[
u(x,p)< A_1e^{-\gamma_1|x|}
\]
and
\[
\partial_j u(x,p)<A_1e^{-\gamma_1|x|}
\]

for $j=1,2$, for any $p\in \mathbb{R}^2$. Suppose that the statement holds for $|r|\le k$ where $k<(m-4)$. Let $|r|=k+1$. According to equations \eqref{eqln213} and the induction hypothesis, we have:
\begin{equation}
    \begin{split}
        &{L}_{\phi}\big[\phi(q(p)), \alpha(q(p))\big](D^ru(.,p))=g^r_{0}(p)\\
        &{L}_{j}\big[\phi(q(p)), \alpha(q(p))\big](D^ru(.,p))=g^r_{j}(p)\hspace{.4cm}j=1,2
    \end{split}
\end{equation}
where
\[
g^r_{0}(p), g^r_{j}(p)< A_2e^{-\gamma_2}x
\]
for every $p\in U$ and by claim \ref{cl1n1l0.9}, there exists a number $a>0$ such that
\[
||g_{0}^r(p)||_{H^2}, ||g_{j}^r(p)||_{H^2}< a
\]
for any $p\in U$. Suppose that
\begin{equation}
    \label{eq4l40}
    D^ru(.,p)=m_r(p)+E_r(p)
\end{equation}
where $m_r\in \big(T_{q(p)}M_N\big)^{\perp}$ and $E_r\in T_{q(p)}M_N$. By claim \ref{cl1n1l0.9}, there exist $A_3,\gamma_3>0$ such that
\[
E_r(p)(x),\partial_1 E_r(p)(x),\partial_2 E_r(p)(x)< A_3 e^{-\gamma_3|x|}
\]
for every $p\in U$. By lemmas \ref{al2} and \ref{l40.4}, there exist $A_4,\gamma_4>0$ such that for any $p\in U$
\[
m_r(p)(x), \partial_1 m_r(p)(x),\partial_2 m_r(p)(x)< A_4 e^{-\gamma_4|x|}
\]
This finishes the proof.
\end{proof}
Claim \ref{cl1n1l1} imply lemma \ref{ln4}.
\end{proof}